\documentclass[10pt]{amsart}

\usepackage{epsf,latexsym,amsmath,amssymb,amscd,datetime}
\usepackage{amsmath,amsthm,amssymb,enumerate,eucal,url,calligra,mathrsfs}

\usepackage{graphicx}
\usepackage{color}
\newcommand{\ourv}[1]{{\bf #1}}

\DeclareMathOperator{\Sky}{Sky}
\DeclareMathOperator{\CoSky}{CoSky}

\newcommand{\red}{\color[rgb]{1.0,0.2,0.2}} 

\DeclareMathOperator{\SHom}{\mathscr{H}\text{\kern -3pt {\calligra\large om}}\,}

\DeclareMathOperator{\Ob}{Ob}
\DeclareMathOperator{\Fl}{Fl}

\IfFileExists{my_xrefs}{\input my_xrefs}{}

\usepackage{mathrsfs}
\usepackage{amssymb}
\usepackage{dsfont}
\usepackage{verbatim}
\usepackage{url}

\theoremstyle{plain}
\newtheorem{theorem}{Theorem}[section]
\newtheorem{lemma}[theorem]{Lemma}

\newtheorem{corollary}[theorem]{Corollary}

\theoremstyle{definition}
\newtheorem{definition}[theorem]{Definition}

\newtheorem{xca}{Exercise}[section]

\newtheorem{example}[theorem]{Example}

\newtheorem{remark}[theorem]{Remark}

\newcommand{\isom}{\simeq} 

\newcommand{\ignore}[1]{}

\newcommand{\Hom}{{\rm Hom}}

\newcommand{\integers}{{\mathbb Z}}

\newcommand{\complex}{{\mathbb C}}

\DeclareMathAlphabet{\mathcal}{OMS}{cmsy}{m}{n}

\newcommand\cA{\mathcal{A}}
\newcommand\cB{\mathcal{B}}
\newcommand\cC{\mathcal{C}}
\newcommand\cD{\mathcal{D}}

\newcommand\cF{\mathcal{F}}
\newcommand\cG{\mathcal{G}}
\newcommand\cH{\mathcal{H}}

\newcommand\cL{\mathcal{L}}
\newcommand\cM{\mathcal{M}}

\newcommand\cO{\mathcal{O}}
\newcommand\cP{\mathcal{P}}

\newcommand\cS{\mathcal{S}}

\newcommand\cY{\mathcal{Y}}

\def\from{\colon}

\def\isom{\simeq}

\def\tensor{\otimes}
\def\eqdef{\overset{\text{def}}{=}}

\DeclareMathOperator{\Ext}{Ext}

\def\Hom{\qopname\relax o{Hom}}
\DeclareMathOperator{\id}{id}

\def\Image{\qopname\relax o{Im}}

\DeclareRobustCommand
  \rddots{\mathinner{\mkern1mu\raise\p@
    \vbox{\kern7\p@\hbox{.}}\mkern2mu
    \raise4\p@\hbox{.}\mkern2mu\raise7\p@\hbox{.}\mkern1mu}}

\newcommand\xhookrightarrow[2][]{\ext@arrow 0062{\hookrightarrowfill@}{#1}{#2}}
\def\hookrightarrowfill@{\arrowfill@\lhook\relbar\rightarrow}

\usepackage{relsize}
\usepackage{tikz}
\usetikzlibrary{matrix,arrows,decorations.pathmorphing}
\usepackage{tikz-cd}
\usetikzlibrary{cd}

\usepackage[pdftex,colorlinks,linkcolor=blue]{hyperref}

\newcommand{\CTwoV}{{\cC_{\rm 2V}}}
\newcommand{\XTwoV}{X_{\rm 2V}}

\begin{document}

\title[Sheaves and Duality in Graph Riemann-Roch]
{Sheaves and Duality in the Two-Vertex Graph Riemann-Roch Theorem}

\author{Nicolas Folinsbee}
\address{Department of Mathematics, University of British Columbia,
        Vancouver, BC\ \ V6T 1Z2, CANADA. }
\curraddr{}
\email{{\tt nicolasfolinsbee@gmail.com}} 
\thanks{Research supported in part by an NSERC grant.}

\author{Joel Friedman}
\address{Department of Computer Science, 
        University of British Columbia, Vancouver, BC\ \ V6T 1Z4, CANADA. }
\curraddr{}
\email{{\tt jf@cs.ubc.ca}}
\thanks{Research supported in part by an NSERC grant.}

\date{July 26, 2022} 



\begin{abstract}
For each graph on two vertices, and each divisor on the graph in the sense of
Baker-Norine, we describe a sheaf of vector spaces on a finite category whose
zeroth Betti number is the Baker-Norine ``Graph Riemann-Roch'' rank of the
divisor plus one.  We prove duality theorems that generalize the Baker-Norine
``Graph Riemann-Roch'' Theorem.
\end{abstract}

\maketitle
\setcounter{tocdepth}{3}
\tableofcontents






\newcommand{\ProjRes}{

\begin{tikzpicture}

\node[] at (0,3){ deg  2};
\node[] at (2.5,3){ deg 1};
\node[] at (5,3){ deg 0 };
\node[] at (7.5,3){ deg -1};

\node at (0,1)(1){\LARGE 0};

\node at (2.5,1.5)(21){$\mathcal{F} (A_1)$};
\node at (2.5,1)(22){$ \bigoplus$};
\node at (2.5,.5)(23){$\mathcal{F} (A_2)$};

\node at (5,2)(31){$\mathcal{F} (B_1)$};
\node at (5,1.5)(32){$ \bigoplus$};
\node at (5,1)(33){$\mathcal{F} (B_2)$};
\node at (5,.5)(34){$ \bigoplus$};
\node at (5,0)(35){$\mathcal{F} (B_3)$};

\node at (7.5,1)(4){\LARGE 0};

\node[] at (0,-3){ deg  2};
\node[] at (2.5,-3){ deg 1};
\node[] at (5,-3){deg 0 };
\node[] at (7.5,-3){ deg -1};

\node at (0,-5)(12){\LARGE 0};

\node at (2.5,-4)(312){$\mathcal{F} (B_1)$};
\node at (2.5,-4.5)(322){$ \bigoplus$};
\node at (2.5,-5)(332){$\mathcal{F} (B_2)$};
\node at (2.5,-5.5)(342){$ \bigoplus$};
\node at (2.5,-6)(352){$\mathcal{F} (B_3)$};

\node at (5,-4.5)(212){$\mathcal{F} (A_1)$};
\node at (5,-5)(222){$ \bigoplus$};
\node at (5,-5.5)(232){$\mathcal{F} (A_2)$};

\node at (7.5,-5)(42){\LARGE 0};

\draw[->,line width=.3mm, shorten <=.5cm] (22) -- (1);

\draw[->,line width=.3mm,shorten <=-.1cm, shorten >=-.15cm] (31) -- (21);
\draw[->,line width=.3mm, shorten >=-.15cm] (33) -- (23);
\draw[->,line width=.3mm,shorten <=-.1cm, shorten >=-.15cm] (35) -- (21);
\draw[->,line width=.3mm, shorten <=-.15cm, shorten >=-.15cm] (35) -- (23);

\draw[->,line width=.3mm, shorten >= .3cm] (4) -- (33);

\draw[->,line width=.3mm,shorten >= .3cm] (12) -- (332);

\draw[->,line width=.3mm,shorten <=-.1cm, shorten >=-.15cm] (312) -- (212);
\draw[->,line width=.3mm, shorten >=-.15cm] (332) -- (232);
\draw[->,line width=.3mm] (352) -- (212);
\draw[->,line width=.3mm,shorten <=-.1cm, shorten >=-.15cm] (352) -- (232);

\draw[->,line width=.3mm, shorten <=.5cm] (222) -- (42);

 \draw [black, thick] (2.5,1) ellipse (8mm and 12mm); 
 \draw [black, thick] (5,1) ellipse (9mm and 16mm); 
 
  \draw [black, thick] (5,-5) ellipse (8mm and 12mm); 
 \draw [black, thick] (2.5,-5) ellipse (9mm and 16mm);

\draw[shorten >=1.1cm,shorten <=1cm,->, line width=.3mm, red] (33) -- (332);
\draw[shorten >=0.8cm,shorten <=.7cm,->, line width=.3mm, red] (22) -- (222);

\end{tikzpicture}
 }

\newcommand{\SkyRes}{

\begin{tikzpicture} [dot/.style={circle,inner sep=1.5pt,fill}]

\node at (1.5,2)(v11){\footnotesize$yk[y]$};
\node at (1.5,1)(v12){$0$};
\node at (1.5,0)(v13){$0$};
\node at (-.5,1.5)(v14){\footnotesize$k[x_1, x_1^{-1}]$};
\node at (-.5,0.5)(v15){$0$};

\node at (6,2)(v21){\footnotesize$k[y_1]$};
\node at (6,1)(v22){$0$};
\node at (6,0)(v23){$0$};
\node at (4,1.5)(v24){\footnotesize $k[x_1, x_1^{-1}]$};
\node at (4,0.5)(v25){$0$};

\node at (10.5,2)(v31){$\tilde{k}$};
\node at (10.5,1)(v32){$0$};
\node at (10.5,0)(v33){$0$};
\node at (9,1.5)(v34){$0$};
\node at (9,0.5)(v35){$0$};

\node[] at (12.5,1)(v4){\LARGE 0};

\draw[->,line width=.3mm] (v11) -- (v14);
\draw[->,line width=.3mm] (v12) -- (v15);
\draw[->,line width=.3mm] (v13) -- (v14);
\draw[->,line width=.3mm] (v13) -- (v15);

\draw[->,line width=.3mm] (v21) -- (v24);
\draw[->,line width=.3mm] (v22) -- (v25);
\draw[->,line width=.3mm] (v23) -- (v24);
\draw[->,line width=.3mm] (v23) -- (v25);

\draw[->,line width=.3mm] (v31) -- (v34);
\draw[->,line width=.3mm] (v32) -- (v35);
\draw[->,line width=.3mm] (v33) -- (v34);
\draw[->,line width=.3mm] (v33) -- (v35);

\draw [->,red] (v11) to [bend right=-5](v21);
\draw [->,red] (v12) to [bend right=-3]  (v22);
\draw [->,red] (v13) to [bend left=-5](v23);
\draw [->,red] (v14) to [bend right=-5] (v24);
\draw [->,red] (v15) to [bend right=-5](v25);

\draw [->,red] (v21) to [bend right=-5](v31);
\draw [->,red] (v22) to [bend right=-3]  (v32);
\draw [->,red] (v23) to [bend left=-5](v33);
\draw [->,red] (v24) to [bend right=-5] (v34);
\draw [->,red] (v25) to [bend right=-5](v35);

\draw[->,line width=.4mm, red] (v32) -- (v4);

\end{tikzpicture}
}

\newcommand{\ProjResTwo}{

\begin{tikzpicture} [dot/.style={circle,inner sep=1.5pt,fill}]

\node at (1.5,2)(v11){$0$};
\node at (1.5,1)(v12){$0$};
\node at (1.5,0)(v13){$0$};
\node at (0,1.5)(v14){$\tilde{S_1}$};
\node at (0,0.5)(v15){$\tilde{S_2}$};

\node at (6,2)(v21){$R_1$};
\node at (6,1)(v22){$R_2$};
\node at (6,0)(v23){$R_3$};
\node at (4,1.5)(v24){$S_1 \oplus S_1$};
\node at (4,0.5)(v25){$S_2 \oplus S_2$};

\node at (10.5,2)(v31){$R_1$};
\node at (10.5,1)(v32){$R_2$};
\node at (10.5,0)(v33){$R_3$};
\node at (9,1.5)(v34){$S_1$};
\node at (9,0.5)(v35){$S_2$};

\node[] at (12.5,1)(v4){\LARGE 0};

\draw[->,line width=.3mm] (v11) -- (v14);
\draw[->,line width=.3mm] (v12) -- (v15);
\draw[->,line width=.3mm] (v13) -- (v14);
\draw[->,line width=.3mm] (v13) -- (v15);

\draw[->,line width=.3mm] (v21) -- (v24);
\draw[->,line width=.3mm] (v22) -- (v25);
\draw[->,line width=.3mm] (v23) -- (v24);
\draw[->,line width=.3mm] (v23) -- (v25);

\draw[->,line width=.3mm] (v31) -- (v34);
\draw[->,line width=.3mm] (v32) -- (v35);
\draw[->,line width=.3mm] (v33) -- (v34);
\draw[->,line width=.3mm] (v33) -- (v35);

\draw [->,red] (v11) to [bend right=-5](v21);
\draw [->,red] (v12) to [bend right=-3]  (v22);
\draw [->,red] (v13) to [bend left=-5](v23);
\draw [->,red] (v14) to [bend right=-5] (v24);
\draw [->,red] (v15) to [bend right=-5](v25);

\draw [->,red] (v21) to [bend right=-5](v31);
\draw [->,red] (v22) to [bend right=-3]  (v32);
\draw [->,red] (v23) to [bend left=-5](v33);
\draw [->,red] (v24) to [bend right=-5] (v34);
\draw [->,red] (v25) to [bend right=-5](v35);

\draw[->,line width=.4mm, red] (v32) -- (v4);

\end{tikzpicture}
}

\newcommand{\CV}{

\begin{tikzpicture}
\node at (0,1.5)(a1){$A_1$};
\node at (0,0)(a2){$A_2$};

\node at (3,1.5)(b1){$B_1$};
\node at (3,0)(b2){$B_2$};
\node at (3,-1.5)(b3){$B_3$};

\draw[->] (a1) -- (b1);

\draw[->] (a2) -- (b2);

\draw[->] (a1) -- (b3);
\draw[->] (a2) -- (b3);

\end{tikzpicture}
 }

\newcommand{\Md}{

\begin{tikzpicture}

\node at (0,1.5)(a1){\footnotesize$k[x_1, x_1^{-1}]$};
\node at (0,0)(a2){\footnotesize$k[x_2, x_2^{-1}]$};

\node at (3,1.5)(b1){\footnotesize$k[y_1]$};
\node at (3,0)(b2){\footnotesize$k[y_2]$};
\node at (3,-1.5)(b3){\footnotesize$(k[v, v^{-1}])^r$};

\draw[->] (b1) to node[above]{$f_1$} (a1);
\draw[->] (b2) to node[above, near start]{$f_2$} (a2);
\draw[->] (b3) to node[below, near start]{$g_1$} (a1);
\draw[->] (b3) to node[below]{$g_2$} (a2);
\end{tikzpicture}
 }

\newcommand{\Or}{
\begin{tikzpicture}

\node at (0,1.5)(a1){\footnotesize$k[x_1, x_1^{-1}]$};
\node at (0,0)(a2){\footnotesize$k[x_2, x_2^{-1}]$};

\node at (3,1.5)(b1){\footnotesize$k[y_1]$};
\node at (3,0)(b2){\footnotesize$k[y_2]$};
\node at (3,-1.5)(b3){\footnotesize$(k[v, v^{-1}])$};

\draw[->] (b1) to node[above]{\scriptsize $f_1$} (a1);
\draw[->] (b2) to node[above, near start]{\scriptsize $f_2$} (a2);
\draw[->] (b3) to node[below, near start]{\scriptsize$g_1$} (a1);
\draw[->] (b3) to node[below]{\scriptsize$g_2$} (a2);
\end{tikzpicture}
}

\newcommand{\inclusions}{

\begin{tikzpicture} [dot/.style={circle,inner sep=1.5pt,fill}]

\node at (0,3){\footnotesize$0 \ points$};
\node at (2.5,3){\footnotesize$1 \ point$};
\node at (5,3){\footnotesize$2 \ points$};
\node at (7.5,3){\footnotesize$3 \ points$};
\node at (10,3){\footnotesize$4 \ points$};
\node at (12.5,3){\footnotesize$5 \ points$};

\node at (0,0)(empty){\footnotesize$\emptyset$};

\node[red] at (2.5,1)(a1){\footnotesize$U_{A_1}$};
\node[red] at (2.5,-1)(a2){\footnotesize$U_{A_2}$};

\node[red] at (5,2)(b1){\footnotesize$U_{B_1}$};
\node at (5,0)(a1a2){\scriptsize$U_{A_1} \cup U_{A_2}$};
\node[red] at (5,-2)(b2){\footnotesize$U_{B_2}$};

\node at (7.5,2)(a2b1){\scriptsize$U_{A_2} \cup U_{B_1}$};
\node[red] at (7.5,0)(b3){\footnotesize$U_{B_3}$};
\node at (7.5,-2)(a1b2){\scriptsize$U_{A_1} \cup U_{B_2}$};

\node at (10,2)(b1b3){\scriptsize$U_{B_1} \cup U_{B_3}$};
\node at (10,0)(b1b2){\scriptsize$U_{B_1} \cup U_{B_2}$};
\node at (10,-2)(b2b3){\scriptsize$U_{B_2} \cup U_{B_3}$};

\node at (12.5,0)(x){\footnotesize$X$};

\draw[->] (empty) -- (a1);
\draw[->] (empty) -- (a2);

\draw[->,red] (a1) -- (b1);
\draw[->] (a1) -- (a1a2);
\draw[->] (a2) -- (a1a2);
\draw[->,red] (a2) -- (b2);

\draw[->] (b1) -- (a2b1);
\draw[->] (a1a2) -- (a2b1);
\draw[->] (a1a2) -- (b3);
\draw[->] (a1a2) -- (a1b2);
\draw[->] (b2) -- (a1b2);

\draw[->] (a2b1) -- (b1b3);
\draw[->] (a2b1) -- (b1b2);
\draw[->] (b3) -- (b1b3);
\draw[->] (b3) -- (b2b3);
\draw[->] (a1b2) -- (b2b3);
\draw[->] (a1b2) -- (b1b2);

\draw[->] (b1b3) -- (x);
\draw[->] (b1b2) -- (x);
\draw[->] (b2b3) -- (x);

\draw[->,red] (a1) -- (b3);
\draw[->,red] (a2) -- (b3);
\end{tikzpicture}
}

\newcommand{\EulerShortExact}{

\begin{tikzpicture} [dot/.style={circle,inner sep=1.5pt,fill}]

\node at (2.5,3)(v11){\footnotesize$k[y_1]$};
\node at (2.5,1.5)(v12){\scriptsize$k[v,v^{-1}]^{\oplus r}$};
\node at (2.5,0)(v13){\footnotesize$k[y_2]$};
\node at (0,2.25)(v14){\scriptsize$k[x_1, x_1^{-1}]$};
\node at (0,.75)(v15){\scriptsize$k[x_2, x_2^{-1}]$};

\node at (8,3)(v21){\footnotesize$k[y_1]$};
\node at (8,1.5)(v22){\scriptsize$k[v,v^{-1}]^{\oplus r}$};
\node at (8,0)(v23){\footnotesize$k[y_2]$};
\node at (5,2.25)(v24){\scriptsize$k[x_1, x_1^{-1}]$};
\node at (5,.75)(v25){\scriptsize$k[x_2, x_2^{-1}]$};

\node at (12.5,3)(v31){$k$};
\node at (12.5,1.5)(v32){$0$};
\node at (12.5,0)(v33){$0$};
\node at (11,2.25)(v34){$0$};
\node at (11,.75)(v35){$0$};

\draw[->,line width=.3mm] (v11) --  node[midway,above, black]{\tiny{ $\rho_{1, d_1}$ }}(v14);
\draw[->,line width=.3mm] (v12) -- (v14);
\draw[->,line width=.3mm] (v12) -- (v15);
\draw[->,line width=.3mm] (v13) --  node[midway,below, black]{\tiny{ $\rho_{2, d_2}$ }}(v15);

\draw[->,line width=.3mm] (v21) --node[above, black, near end]{\tiny{$\rho_{1, d_1 + 1}$}} (v24);
\draw[->,line width=.3mm] (v22) -- (v24);
\draw[->,line width=.3mm] (v22) -- (v25);
\draw[->,line width=.3mm] (v23) -- node[near end,below, black]{\tiny{ $\rho_{2, d_2}$ }}(v25);

\draw[->,line width=.3mm] (v31) -- (v34);
\draw[->,line width=.3mm] (v32) -- (v34);
\draw[->,line width=.3mm] (v32) -- (v35);
\draw[->,line width=.3mm] (v33) -- (v35);

\draw [->,blue] (v11) to node[midway,above, black]{\scriptsize{ $p(y_1) \mapsto y_1 p(y_1) $ }} (v21);
\draw [->,red] (v12) to  (v22);
\draw [->,red] (v13) to (v23);
\draw [->,red] (v14) to (v24);
\draw [->,red] (v15) to (v25);

\draw [->,blue] (v21) to (v31);
\draw [->,red] (v22) to  (v32);
\draw [->,red] (v23) to (v33);
\draw [->,red] (v24) to (v34);
\draw [->,red] (v25) to (v35);

\end{tikzpicture}


}

\newcommand{\CTwoVTop}{
\begin{tikzpicture}
\node at (0,2){$A_1$};
\node at (0,0){$A_2$};

\node at (3,3){$B_1$};
\node at (3,1){$B_3$};
\node at (3,-1){$B_2$};

\draw[black!30!green] (0,0) circle (10pt);
\draw[black!30!green] (0,2) circle (10pt);

\draw [red] plot [smooth cycle] coordinates {(3.5,1) (-.3,2.5) (-.3,-.5)};

\draw [red] plot [smooth cycle] coordinates {(3.5,3.5) (-.5,2.5) (-.5,1.5) (3.5,2.5)};

\draw [red] plot [smooth cycle] coordinates {(3.5,-.5) (-.5,.5) (-.5,-.5) (3.5,-1.5)};

\end{tikzpicture}
}

\newcommand{\CoSkyBOne}{

\begin{tikzpicture}
\node at (1.5,2)(v11){\footnotesize$M_1$};
\node at (1.5,1)(v12){$0$};
\node at (1.5,0)(v13){$0$};
\node at (-.5,1.5)(v14){\footnotesize$N_1$};
\node at (-.5,0.5)(v15){$0$};

\draw[->] (v11) -- (v14);
\draw[->] (v12) -- (v14);
\draw[->] (v12) -- (v15);
\draw[->] (v13) -- (v15);

\end{tikzpicture}

}

\newcommand{\CoSkyBThree}{

\begin{tikzpicture}
\node at (1.5,2)(v11){$0$};
\node at (1.5,1)(v12){\footnotesize$M_3$};
\node at (1.5,0)(v13){$0$};
\node at (-.5,1.5)(v14){\footnotesize$N_1$};
\node at (-.5,0.5)(v15){\footnotesize$N_2$};

\draw[->] (v11) -- (v14);
\draw[->] (v12) -- (v14);
\draw[->] (v12) -- (v15);
\draw[->] (v13) -- (v15);

\end{tikzpicture}

}

\newcommand{\CoSkyBTwo}{

\begin{tikzpicture}
\node at (1.5,2)(v11){$0$};
\node at (1.5,1)(v12){$0$};
\node at (1.5,0)(v13){\footnotesize$M_2$};
\node at (-.5,1.5)(v14){$0$};
\node at (-.5,0.5)(v15){\footnotesize$M_2$};

\draw[->] (v11) -- (v14);
\draw[->] (v12) -- (v14);
\draw[->] (v12) -- (v15);
\draw[->] (v13) -- (v15);

\end{tikzpicture}

}

\newcommand{\CoSkyAOne}{

\begin{tikzpicture}
\node at (1.5,2)(v11){$0$};
\node at (1.5,1)(v12){$0$};
\node at (1.5,0)(v13){$0$};
\node at (-.5,1.5)(v14){\footnotesize$N_1$};
\node at (-.5,0.5)(v15){$0$};

\draw[->] (v11) -- (v14);
\draw[->] (v12) -- (v14);
\draw[->] (v12) -- (v15);
\draw[->] (v13) -- (v15);

\end{tikzpicture}

}

\newcommand{\CoSkyATwo}{

\begin{tikzpicture}
\node at (1.5,2)(v11){$0$};
\node at (1.5,1)(v12){$0$};
\node at (1.5,0)(v13){$0$};
\node at (-.5,1.5)(v14){$0$};
\node at (-.5,0.5)(v15){\footnotesize$N_2$};

\draw[->] (v11) -- (v14);
\draw[->] (v12) -- (v14);
\draw[->] (v12) -- (v15);
\draw[->] (v13) -- (v15);

\end{tikzpicture}

}
\newcommand{\SkyBOne}{

\begin{tikzpicture}
\node at (1.5,2)(v11){\footnotesize$M_1$};
\node at (1.5,1)(v12){$0$};
\node at (1.5,0)(v13){$0$};
\node at (-.5,1.5)(v14){$0$};
\node at (-.5,0.5)(v15){$0$};

\draw[->] (v11) -- (v14);
\draw[->] (v12) -- (v14);
\draw[->] (v12) -- (v15);
\draw[->] (v13) -- (v15);

\end{tikzpicture}

}

\newcommand{\SkyBThree}{

\begin{tikzpicture}
\node at (1.5,2)(v11){$0$};
\node at (1.5,1)(v12){\footnotesize$M_3$};
\node at (1.5,0)(v13){$0$};
\node at (-.5,1.5)(v14){$0$};
\node at (-.5,0.5)(v15){$0$};

\draw[->] (v11) -- (v14);
\draw[->] (v12) -- (v14);
\draw[->] (v12) -- (v15);
\draw[->] (v13) -- (v15);

\end{tikzpicture}

}

\newcommand{\SkyBTwo}{

\begin{tikzpicture}
\node at (1.5,2)(v11){$0$};
\node at (1.5,1)(v12){$0$};
\node at (1.5,0)(v13){\footnotesize$M_2$};
\node at (-.5,1.5)(v14){$0$};
\node at (-.5,0.5)(v15){$0$};

\draw[->] (v11) -- (v14);
\draw[->] (v12) -- (v14);
\draw[->] (v12) -- (v15);
\draw[->] (v13) -- (v15);

\end{tikzpicture}

}

\newcommand{\SkyAOne}{

\begin{tikzpicture}
\node at (1.5,2)(v11){\footnotesize$N_2$};
\node at (1.5,1)(v12){\footnotesize$N_2$};
\node at (1.5,0)(v13){$0$};
\node at (-.5,1.5)(v14){\footnotesize$N_2$};
\node at (-.5,0.5)(v15){$0$};

\draw[->] (v11) -- (v14);
\draw[->] (v12) -- (v14);
\draw[->] (v12) -- (v15);
\draw[->] (v13) -- (v15);

\end{tikzpicture}

}

\newcommand{\SkyATwo}{

\begin{tikzpicture}
\node at (1.5,2)(v11){$0$};
\node at (1.5,1)(v12){\footnotesize$N_1$};
\node at (1.5,0)(v13){\footnotesize$N_1$};
\node at (-.5,1.5)(v14){$0$};
\node at (-.5,0.5)(v15){\footnotesize$N_1$};

\draw[->] (v11) -- (v14);
\draw[->] (v12) -- (v14);
\draw[->] (v12) -- (v15);
\draw[->] (v13) -- (v15);

\end{tikzpicture}

}

\newcommand{\Category}{

\begin{tikzpicture}
\node at (1.5,2)(v11){\footnotesize$B_1$};
\node at (1.5,1)(v12){\footnotesize$B_3$};
\node at (1.5,0)(v13){\footnotesize$B_2$};
\node at (-.5,1.5)(v14){\footnotesize$A_1$};
\node at (-.5,0.5)(v15){\footnotesize$A_2$};

\draw[->] (v14) -- (v11);
\draw[->] (v14) -- (v12);
\draw[->] (v15) -- (v12);
\draw[->] (v15) -- (v13);

\end{tikzpicture}

}

\newcommand{\Struct}{

\begin{tikzpicture}
\node at (1.5,2)(v11){\footnotesize$R_1$};
\node at (1.5,1)(v12){\footnotesize$R_3$};
\node at (1.5,0)(v13){\footnotesize$R_2$};
\node at (-.5,1.5)(v14){\footnotesize$S_1$};
\node at (-.5,0.5)(v15){\footnotesize$S_2$};

\draw[->] (v11) -- (v14);
\draw[->] (v12) -- (v14);
\draw[->] (v12) -- (v15);
\draw[->] (v13) -- (v15);

\end{tikzpicture}

}

\newcommand{\SheafModules}{

\begin{tikzpicture}
\node at (1.5,2)(v11){\footnotesize$M_1$};
\node at (1.5,1)(v12){\footnotesize$M_3$};
\node at (1.5,0)(v13){\footnotesize$M_2$};
\node at (-.5,1.5)(v14){\footnotesize$N_1$};
\node at (-.5,0.5)(v15){\footnotesize$N_2$};

\draw[->] (v11) -- (v14);
\draw[->] (v12) -- (v14);
\draw[->] (v12) -- (v15);
\draw[->] (v13) -- (v15);

\end{tikzpicture}

}

\newcommand{\CoSkyHom}{

\begin{tikzpicture} [dot/.style={circle,inner sep=1.5pt,fill}]

\node at (2.5,3)(v11){\footnotesize$0$};
\node at (2.5,1.5)(v12){\footnotesize$M$};
\node at (2.5,0)(v13){\footnotesize$0$};
\node at (0,2.25)(v14){\scriptsize$M \otimes_{R_3} S_1$};
\node at (0,.75)(v15){\scriptsize$M \otimes_{R_3} S_2$};

\node at (8,3)(v21){\scriptsize$\mathcal{F}(B_1)$};
\node at (8,1.5)(v22){\scriptsize$\mathcal{F}(B_3)$};
\node at (8,0)(v23){\scriptsize$\mathcal{F}(B_2)$};
\node at (5,2.25)(v24){\scriptsize$\mathcal{F}(A_1)$};
\node at (5,.75)(v25){\scriptsize$\mathcal{F}(A_2)$};

\draw[->,line width=.3mm] (v11) --  (v14);
\draw[->,line width=.3mm] (v12) -- (v14);
\draw[->,line width=.3mm] (v12) -- (v15);
\draw[->,line width=.3mm] (v13) -- (v15);

\draw[->,line width=.3mm] (v21) -- (v24);
\draw[->,line width=.3mm] (v22) -- (v24);
\draw[->,line width=.3mm] (v22) -- (v25);
\draw[->,line width=.3mm] (v23) -- (v25);

\draw [->,red,dashed] (v11) to  (v21);
\draw [->,red] (v12) to  node[midway,above, black]{\scriptsize{ $\psi$ }} (v22);
\draw [->,red,dashed] (v13) to (v23);
\draw [->,red, dashed] (v14) to (v24);
\draw [->,red, dashed] (v15) to (v25);

\end{tikzpicture}


}

\newcommand{\SkyHom}{

\begin{tikzpicture} [dot/.style={circle,inner sep=1.5pt,fill}]

\node at (2.5,3)(v11){\scriptsize$\mathcal{F}(B_1)$};
\node at (2.5,1.5)(v12){\scriptsize$\mathcal{F}(B_3)$};
\node at (2.5,0)(v13){\scriptsize$\mathcal{F}(B_2)$};
\node at (0,2.25)(v14){\scriptsize$\mathcal{F}(A_1)$};
\node at (0,.75)(v15){\scriptsize$\mathcal{F}(A_2)$};

\node at (8,3)(v21){\footnotesize$M$};
\node at (8,1.5)(v22){\footnotesize$0$};
\node at (8,0)(v23){\footnotesize$0$};
\node at (5,2.25)(v24){\footnotesize$0$};
\node at (5,.75)(v25){\footnotesize$0$};

\draw[->,line width=.3mm] (v11) --  (v14);
\draw[->,line width=.3mm] (v12) -- (v14);
\draw[->,line width=.3mm] (v12) -- (v15);
\draw[->,line width=.3mm] (v13) -- (v15);

\draw[->,line width=.3mm] (v21) -- (v24);
\draw[->,line width=.3mm] (v22) -- (v24);
\draw[->,line width=.3mm] (v22) -- (v25);
\draw[->,line width=.3mm] (v23) -- (v25);

\draw [->,red] (v11) to  (v21);
\draw [->,red, dashed] (v12) to  (v22);
\draw [->,red,dashed] (v13) to (v23);
\draw [->,red, dashed] (v14) to  (v24);
\draw [->,red, dashed] (v15) to (v25);

\end{tikzpicture}


}

\newcommand{\SkyHomTwo}{

\begin{tikzpicture} [dot/.style={circle,inner sep=1.5pt,fill}]

\node at (2.5,3)(v11){\scriptsize$\mathcal{F}(B_1)$};
\node at (2.5,1.5)(v12){\scriptsize$\mathcal{F}(B_3)$};
\node at (2.5,0)(v13){\scriptsize$\mathcal{F}(B_2)$};
\node at (0,2.25)(v14){\scriptsize$\mathcal{F}(A_1)$};
\node at (0,.75)(v15){\scriptsize$\mathcal{F}(A_2)$};

\node at (8,3)(v21){\footnotesize$M$};
\node at (8,1.5)(v22){\footnotesize$M$};
\node at (8,0)(v23){\footnotesize$0$};
\node at (5,2.25)(v24){\footnotesize$M$};
\node at (5,.75)(v25){\footnotesize$0$};

\draw[->,line width=.3mm] (v11) --  (v14);
\draw[->,line width=.3mm] (v12) -- (v14);
\draw[->,line width=.3mm] (v12) -- (v15);
\draw[->,line width=.3mm] (v13) -- (v15);

\draw[->,line width=.3mm] (v21) -- (v24);
\draw[->,line width=.3mm] (v22) -- (v24);
\draw[->,line width=.3mm] (v22) -- (v25);
\draw[->,line width=.3mm] (v23) -- (v25);

\draw [->,red, dashed] (v11) to  (v21);
\draw [->,red, dashed] (v12) to  (v22);
\draw [->,red,dashed] (v13) to (v23);
\draw [->,red] (v14) to  (v24);
\draw [->,red, dashed] (v15) to (v25);

\end{tikzpicture}


}

\section{Introduction}

The main goals of this article are to give an algebraic model of the Graph
Riemann-Roch Theorem of Baker-Norine \cite{baker_norine} in the
special case of a graph with two vertices, and---in our
algebraic model---prove a duality theorem that generalizes the Graph
Riemann-Roch Theorem.  Our algebraic model was formulated by a process of
trial and error and seems a bit ad hoc to us; there may well be 
entirely different ways
to algebraically model the Graph Riemann-Roch Theorem.

Perhaps the main result in this paper is our method of modeling the
Graph Riemann-Roch theorem with sheaves and algebra.
The machinery we use to prove more general duality theorems 
is based on standard tools used for modern, sheaf
theoretic proofs of the classical Riemann-Roch theorem for algebraic curves
and common ideas in topos theory.
However, there is some work needed to formulate and verify
the duality theorems we prove that
generalize the Graph Riemann-Roch theorem.

We believe---for a number of reasons---that this article 
just scratches the surface of the connection between sheaves and
the Baker-Norine theorem.  First, we only study the Baker-Norine
theorem for graphs with two vertices.  Second, there are likely
better ways to understand the sheaves we use in this article.
Third, there are a number of aspects for further study of these sheaves;
we mention some in the closing section of this article.

This is part of a general line of research
(\cite{friedman_cohomology,friedman_cohomology2,friedman_linear,
friedman_memoirs_hnc,abgirth,alice_thesis}) to investigate Grothendieck
topologies on finite categories and their possible applications to problems
in computer science theory, discrete mathematics, and linear algebra.
The main inspiration of this line of research is based
on Ben-Or's suggestion---in view of work on algebraic computations
\cite{dobkin,steele,benor}---that to settle problems in Boolean
complexity theory one should search for 
cohomology theories connected to Boolean functions; according to Ben-Or's
work, even if the zeroth Betti number in such a cohomology theory is
not particularly interesting, the higher Betti numbers may also be able
to produce lower bounds in complexity theory.  
This investigation of Grothendieck topologies has uncovered a number of
surprising connections to graph theory and linear algebra, notably a solution
to the Hanna Neumann Conjecture of the 1950's
\cite{friedman_memoirs_hnc}; another such connection is to linear
algebra, namely a study of
{\em $2$-independence} in \cite{alice_thesis}.
Furthermore, there are other problems in
computer science theory that can be stated in terms of
sheaves on
graphs, such as the construction of
relative expanders 
\cite{friedman_relative,friedman_kohler,mssI,mssIV}.
For the above reasons we believe that the application of
Grothendieck topologies to these fields is likely to yield further results.

We believe that the model and duality theorems in this article
may not only shed light on the
work of Baker-Norine, but also adds to the foundational tools and examples
of interesting sheaves based on finite categories.

The sheaves we study look similar to what one would get from a \v{C}ech
cohomology computation 
of an algebraic curve, but seem fundamentally different.  
One reason for this is that the first Betti number of the structure 
sheaf is not generally finite.  


Our methods are simple adaptations of known methods in algebraic
topology and geometry that closely
mimic standard methods and theorems for algebraic curves.
The duality theorems we prove are akin to the
Serre duality theorem, although only ``half'' of the duality theorem holds
for sheaves of interest to us; {\em strong duality} seems to only hold for
certain {\em torsion skyscraper} sheaves.  
We prove our duality theorems using the {\em method of Grothendieck},
whereby we prove the theorems for certain sheaves, and infer the rest
by short exact sequences and the five-lemma.


While do not entirely understand our models in a simple way,
we will make a number of remarks regarding our model.
The category we use for the models in this article is a category
with five objects, and is the opposite category of those used in
\cite{friedman_memoirs_hnc,alice_thesis}.  We discuss the possibility
of {\em local methods} using such categories, and explain how this
approach may be of interest to related discrete structures, such
as more general graphs and graphs of groups (\cite{serre_trees}).

{\color{red}
This version corrects the original version of this preprint.
In the original version, an erroneous version of Serre duality was
claimed for the $\cO$-modules in this paper that model the
two-vertex Riemann-Roch rank.
We believe that there is a version of Serre duality in this situation,
although it is more involved and subtle.
We hope to address this in a future paper.
For ease of reading, we will make our corrections in RED PRINT.
All the other results, including a number of interesting foundational
remarks, hold.
}

The rest of this article is organized as follows.
In Section~\ref{se_overview} we give an overview of the main results
in this paper, using terminology to be defined in later subsections;
the reader familiar sheaf theory and/or
with the Baker-Norine Graph Riemann-Roch Theorem
\cite{baker_norine} will likely understand some of this terminology.
In Section~\ref{se_simpleI} we give some fundamental definitions
regarding the categories and sheaves that we use.
Section~\ref{se_simpleI.5} discusses the connection of our notion of
a sheaf to the classical notion of a sheaf on a topological space;
it is not essential to the rest of the paper, but the reader familiar
with sheaves on topological spaces will likely benefit from the discussion.
The main modeling theorem in this paper is proven in
Section~\ref{se_models_grr}, which expresses the Baker-Norine
Graph Riemann-Roch rank plus one as the zeroth Betti number of
a sheaf on a finite category.
Sections~\ref{se_simpleII} and \ref{se_some_Betti} make some
Betti number computations we use in Section~\ref{se_rr},
where we compute the Euler characteristics of our sheaves;
this computation shows that
Baker-Norine theorem is equivalent to a duality result regarding
the zeroth and first Betti numbers of our sheaves that model the
Graph Riemann-Roch rank.
In Section~\ref{se_simpleIII} we give some general statements about
tools that are useful in cohomology, and in Section~\ref{se_simpleIV}
we describe the specific implications for our situation.
In Sections~\ref{se_global_hom} and \ref{se_duality_proof}
we prove 
{\red partial results regarding duality and the Graph Riemann-Roch
Theorem}.
Although our generalization regards ``global sections'' and ``global Ext
groups,'' there are a number of indications that these are special
cases of more general ``local theorems'' that are more general and---in
a sense---easier to prove.
We give a very concrete statement of a ``local version'' of the 
{\red some of our partial duality theorems} in
Section~\ref{se_simpleV}.
In Section~\ref{se_future} we make a number of remarks regarding future
directions of research, beginning with
some longer remarks on ``local methods'' for finite 
categories---including graphs and graphs of groups---and
ending with a few brief remarks.

We wish to thank Ehud De Shallit and Luc Illusie
for discussions regarding our models; their observations
will be described in Section~\ref{se_overview}.

Throughout this article, [SGA4] refers to
\cite{sga4.1}; [EGA1]
refers to \cite{ega1}.

\section{Overview of Results}
\label{se_overview}

In this section we summarize the main results of this paper,
using terminology to be defined later in this paper.
The reader familiar with sheaf theory or
the Baker-Norine {\em graph Riemann-Roch theorem}
will likely be familiar with some of this terminology already.

%

\subsection{The Main Modeling Result}


Our main ``modeling'' result in that for any field, $k$,
integer $r\ge 1$, and $\ourv d\in\integers^2$, we 
prove that
\begin{equation}\label{eq_main_modeling_result}
b^0(\cM_{k;r,\ourv d}) = {\rm GRRR}_r(\ourv d) + 1 \ ,
\end{equation}
where
\begin{enumerate}
\item 
${\rm GRRR}_r(\ourv d)$ denotes the {\em graph Riemann-Roch rank}---in
the sense of Baker-Norine---of the divisor $\ourv d$
on a graph with two vertices joined by $r$ edges,
\item
$b^0$ is the zeroth Betti number, i.e., the dimension
of the $k$-vector space of global sections of a sheaf, and
\item
$\cM_{k;r,\ourv d}$ are sheaves of $k$-vector spaces 
on fixed finite category $\CTwoV$ (see Definition~\ref{de_model_grrr_two}
and Figures~\ref{fi_cat_Or} and \ref{fi_Mkrd}).
\end{enumerate}
We emphasize that the above formula is valid for an arbitrary field $k$,
and at times we omit the $k$ in our notation for brevity.

The category $\CTwoV$ is category with five elements, endowed with
its coarsest topology; a {\em sheaf} of $k$-vector spaces
is therefore a contravariant
functor from $\CTwoV$ to the category of $k$-vector spaces;
equivalently, we may think of $\CTwoV$ as a topological space with
five points, five irreducible open subsets (which are the ``stalks''
of the five points), and 13 open subsets
(see Figure~\ref{fi_inclusions}).

%
%

\subsection{Cohomology Groups and the First Betti Number}

For any sheaf of vector spaces, $\cF$, on $\CTwoV$,
the {\em cohomology groups} of $\cF$, denoted $H^0(\cF),H^1(\cF)$, are the
right derived functors of $\cF\mapsto \Gamma(\cF)$, and can be computed as
the kernel and cokernel (respectively) of
\begin{equation}\label{eq_cohomology_formula}
\cF(B_1)\oplus\cF(B_2)\oplus \cF(B_3) \to \cF(A_1)\oplus \cF(A_2) \ ;
\end{equation}
we set $b^i(\cF)=\dim_k H^i(\cF)$.
[The maps $\cF(B_3)\to\cF(A_i)$ are
minus the restriction maps, but this doesn't change the $b^i$.]

We can also see that to any short exact sequence of sheaves, there is
a long exact sequence in cohomology, as usual.

We give a general method to compute the first Betti number, 
i.e., $b^1$, for many sheaves of interest to us, including sheaves
of the form
$\cO_{k,r}$, $\cM_{k,r,\ourv d}$.


\subsection{The Euler Characteristic Formula}

Next we derive the formula
\begin{equation}\label{eq_Euler_char_computation}
\chi(\cM_{r,\ourv d}) = d_1+d_2 - (r-1),
\end{equation}
where $\chi(\cF)=b_0(\cF)-b_1(\cF)$.
Our proof is to verify this formula for $\ourv d = \ourv 0$, and
to use a short exact sequence
\begin{equation}\label{eq_typical_short_exact}
0\to \cM_{r,\ourv d} \to  \cM_{r,\ourv d+(1,0)} \to {\rm Sky}(B_1,k)\to 0
\end{equation}
where $\Sky(B_1,k)$ is a {\em skyscraper sheaf}
(this exact sequence mimicks the standard one for algebraic curves, 
e.g., \cite{hartshorne}, page~296, just above Remark~IV.1.3.1).
The long exact sequence for \eqref{eq_typical_short_exact} then 
implies---after a simple computation of the Betti numbers of the
above skyscraper sheaf---that
if both Betti numbers of $\cM_{r,\ourv d}$ are finite,
or both of $\cM_{r,\ourv d+(1,0)}$ are finite, then all are finite and
$$
\chi(\cM_{r,\ourv d+(1,0)}) = \chi(\cM_{r,\ourv d}) + 1 .
$$
Doing the same for $(0,1)$ replacing $(1,0)$ and 
$B_2$ replacing $B_1$ in the skyscraper), we see that
$$
\chi(\cM_{r,\ourv d}) = \chi(\cM_{r,0}) + d_1+d_2 = d_1+d_2-(r-1)
$$
which imples \eqref{eq_Euler_char_computation}.

\subsection{The Euler Characteristic and Graph Riemann-Roch Theorem} 

The formula \eqref{eq_Euler_char_computation} says that
$$
b_0(\cM_{r,\ourv d}) - b_1(\cM_{r,\ourv d}) = d_1+d_2-(r-1) .
$$
Using \eqref{eq_main_modeling_result}, we see that the
Baker-Norine Graph Riemann-Roch Theorem is equivalent to the formula
\begin{equation}\label{eq_Baker_Norine}
b_0(\cM_{r,\ourv d}) - b_0(\cM_{r,K_r-\ourv d}) = d_1+d_2-(r-1) 
\end{equation} 
for what Baker and Norine call the {\em canonical divisor}, $K_r=(r-2,r-2)$.
It follows that the Baker-Norine Theorem can be stated as
\begin{equation}\label{eq_main_duality}
b_1(\cM_{r,\ourv d})  = b_0(\cM_{r,K-\ourv d}) \ .
\end{equation}
We know that \eqref{eq_Baker_Norine} holds since the Baker-Norine
Theorem is true; we imagine that it may not be difficult to prove
\eqref{eq_Baker_Norine} from scratch, using
the formulas that we develop for $b^i(\cM_{r,\ourv d})$.
The rest of this paper is mainly devoted to proving a generalization
of \eqref{eq_main_duality}, which gives another proof of this
\eqref{eq_main_duality}.
Our generalization relies on general facts (such as the Yoneda pairing
and the {\em Method of Grothendieck}), but involves a smaller number of
specific calculations.

\subsection{$\cO_{k,r}$-Modules and the Duality Theorems}

For any {\em sheaf of rings}, $\cO$, on $\CTwoV$, there is a standard
notion of a
{\em sheaf of $\cO$-modules} and the morphisms $\cF\to\cG$ of any two
$\cO$-modules, which we denote $\Hom_\cO(\cF,\cG)$.
In the special case where $\cO$ is a sheaf of $k$-algebras, 
$\Hom_\cO(\cF,\cG)$ has the natural structure of a $k$-vector space.
In the further special case where
$\cO=\underline k$, the
constant sheaf $k$, the category of
{\em sheaves of $\underline k$-modules} is the
same thing as the category of sheaves of $k$-vector spaces.

In this article we will define sheaves of rings $\cO_{k,r}$ on $\CTwoV$
for each integer $r\ge 1$ and field, $k$.  
The sheaves
$\cM_{k,r,\ourv d}$ are sheaves of $\cO_{k,r}$-modules.

Let $\omega$ be a sheaf of $k$-vector spaces on $\CTwoV$
such that $b^1(\omega)=1$.
For any sheaf of $k$-algebras, $\cO$,
and $i=0,1$,
we define ${\rm Duality}_{\cO,i}(\omega)$ to be class of $\cO$-modules, $\cF$,
for which duality holds for $H^i(\cF)$, in the sense that
the Yoneda pairing
\begin{equation}\label{eq_perfect_pairing}
H^i(\cF) \times \Ext^{1-i}_{\cO}(\cF,\omega) \to H^1(\omega)\isom k
\end{equation}
is a perfect pairing, where $\Ext_\cO^i$ denotes the
derived functors of $\Hom_\cO$; we define
${\rm Strong-Duality}_\cO(\omega)$ to be the class of $\cO$-modules, $\cF$,
for which
{\em strong duality} holds (with respect to $\cO$ and $\omega$),
meaning that duality holds for all $i\ge 0$ (hence necessarily
$H^i(\cF)=\Ext^i_\cO(\cF,\omega)=0$ for $i\ge 2$).

Our generalization of \eqref{eq_Baker_Norine} consists of the following
two theorems: for any field $k$ and integer $r\ge 1$,
set $\omega_{k,r}=\cM_{k,r,K_r}$
with $K_r=(r-2,r-2)$; then
\begin{enumerate}
\item $b^1(\omega_{k,r})=1$ and $\omega_{k,r}$ has an injective resolution
of length two;
\item for $i=1,2$,
${\rm Strong-Duality}_{\cO_{k,r}}(\omega_{k,r})$ contains the
({\em torsion skyscraper}) sheaf
$\cS_i=\Sky(B_i,k[y_i]/y_i k[y_i])$;
\item
{\red a previous version claimed that
$\cM_{k,r,\ourv d}\in {\rm Duality}_{\cO_{k,r},1}(\omega_{k,r})$ for
any $\ourv d\in\integers^2$, however this is not generally true;}
\item
for any $\ourv d,\ourv d'\in\integers^2$ there is a canonical morphism
{\red a previous version claimed that this was an isomorphism}
\begin{equation}\label{eq_Hom_two_Ms}
\Hom_{\cO_{k,r}}(\cM_{r,\ourv d},\cM_{r,\ourv d'})\to
\Gamma\bigl( \cM_{r,\ourv d'-\ourv d} \bigr) .
\end{equation} 
\end{enumerate}
We will prove \eqref{eq_Hom_two_Ms} by a simple computation:
in the case where $\ourv d=\ourv 0$, the morphism is induced
a natural inclusion $\cO_{k,r}\to \cM_{k,r,\ourv 0}$;
the case of general $\ourv d$ is proved similarly.
We will see that strong duality cannot hold for sheaves of the
form $\cM_{k,r,\ourv d}$ or $\cO_{k,r}$, but we will see some
formulas akin to strong duality, both for $H^1$ duality and $H^0$
duality.

More specifically, we define line bundles, $\cL_{k,r,\ourv d}$,
that are subsheaves of $\cM_{k,r,\ourv d}$, where $\cL_{k,r,\ourv 0}$
is just the structure sheaf (just as for algebraic curves).  
From the duality theorems
and some easier computations, we derive a number of interesting formulas
including
$$
\cL_{k,r,\ourv d}\otimes \cM_{k,r,\ourv d'} \isom \cM_{k,r,\ourv d+d'},
$$
and the duality theorem
$$
H^1(\cM_{k,r,\ourv d}) \isom \Hom_{\cO_{k,r}}(\cM_{k,r,\ourv d},
\omega_{k,r})^*, \quad \mbox{where}\quad \omega_{k,r}=\cM_{k,r,K_r}.
$$
A consequence of the above results is (the less interesting result) that
$$
H^0(\cM_{k,r,\ourv d}) \isom \Ext^1_{\cO_{k,r}}(\cL_{k,r,\ourv d},
\omega_{k,r})^*
$$
(which is the correct $H^0$ statement resembling Serre duality, except
that $\cL_{k,r,\ourv d}$ must be used instead of $\cM_{k,r,\ourv d}$
for the right-hand-side); this result follows from the above since
$$
\Ext^1_\cO(\cL_{k,r,\ourv d},\omega_{k,r}) \isom 
\Ext^1_\cO(\cO,\cM_{k,r,K_r-\ourv d}) \isom H^1(\cM_{k,r,K_r-\ourv d})
$$
(with $\cO=\cO_{k,r}$), whereupon we apply the above duality theorem
and \eqref{eq_Hom_two_Ms}.

We will also prove that {\red there is a morphism}
$$
\SHom(\cM_{k,r,\ourv d},\cM_{k,r,\ourv d'}) 
{\red \to}
\cM_{k,r,\ourv d' - \ourv d},
$$
but that
$$
\cM_{k,r,\ourv d}\tensor \cM_{k,r,\ourv d'}
$$
has an interesting ``twisted'' structure.

\subsection{Proof of the Duality Theorems and the ``Method of Grothendieck''}

The ``Method of Grothendieck'' implies that for any sheaf of
$k$-algebras $\omega$ on
$\CTwoV$ with $b^1(\omega)=1$, 
\begin{enumerate}
\item
if any two
elements of any short exact sequence are contained 
in ${\rm Strong-Duality}(\omega)$,
then all three are; and
\item
if $0\to\cF_1\to\cF_2\to\cF_3\to 0$ is exact, 
$\Ext^2(\cF_1,\omega)=0$, and $H^1$ duality holds for $\cF_1$ and $\cF_2$,
then $H^1$ duality holds for $\cF_3$.
\end{enumerate}

We first verify that $\Sky(B_1,k)\in{\rm Duality}(\omega)$, where
when $\Sky(B_1,k)$ is viewed as an $\cO_{k,r}$-module by
the exact sequence \eqref{eq_typical_short_exact} and the structure
of the $\cM_{k,r,\ourv d}$ as $\cO_{k,r}$-modules: this means that
$k$ is really
$$
\widetilde k \eqdef k[y_1]/(y_1) = k[y_1]/y_1 k[y_1],
$$
so that $k$ is a set of representatives of the classes of $\widetilde k$,
and $k[y_1]$ acts on $\widetilde k$ in that $y_1$ annihilates $\widetilde k$.
We use $\widetilde k$ instead of $k$ to remind ourselves that
$\widetilde k$ is $k$ with this particular structure of a $k[y_1]$-algebra.
We similarly show that $\Sky(B_2,k)\in{\rm Duality}(\omega)$ where
this time $k$ really means $k[y_2]/y_2 k[y_2]$.

It then follows from the Method of Grothendieck that
$\cM_{k,r,\ourv d}\in{\rm Duality}(\omega)$ for all $\ourv d$
provided that this holds for any one value of $\ourv d$; we 
then verify that it holds for any $\ourv d$ such that $d_1+d_2$ is
sufficiently smaller than zero.
{\red However $\cM_{k,r,\ourv d}\in{\rm Duality}(\omega)$ can
only hold in very simple cases, such as $r=1$; hence a
duality theory for $\cM_{k,r,\ourv d}$ needs to remedy this problem
and is likely more subtle.}

\subsection{Remarks on $\cO_{k,r}$, Remarks of De Shallit and Illusie}

For $r=1$, the sheaf $\cO_{k,r}$ represents the algebraic curve of
genus zero, where the sphere is covered by two ``hemispheres,'' one
omitting $\infty$ and one omitting $0$.  For $r\ge 2$ the sheaf
$\cO_{k,r}$ does not represent an algebraic curve; one way to see this
is that $b^1(\cO_{k,r})$ is infinite for such $r$.

For $r\ge 2$, part of $\cO_{k,r}$ is identical to $\cO_{k,1}$ and
represents two hemispheres.
Ehud De Shallit has pointed out to us that for $r\ge 2$ one can view
the hemispheres as being glued by a clutching map (or cylinder) that
wraps $r$ times around in the clutching.
Hence we get a topological idea (at least for $k=\complex$) of how
to think of $\cO_{k,r}$.
The sheaves $\cM_{k,r}$ are standard line bundles away from this
``clutching,'' but are of rank $r$, rather than one, on the cylinder
that performs this clutching.

Luc Illusie has made a number of interesting remarks and asked us
some interesting questions regarding
the $\cO_{k,r}$, such as asking whether they represent an algebraic
curve (which they do not seem to, since $b^1(\cO_{k,r})=\infty$).
Regarding another remark of Illusie:
the sheaves $\cM_{k,r,\ourv d}$ are
all finitely generated $\cO_{k,r}$-modules; however the
sheaves $\cM_{k,r,\ourv d}$ are not coherent $\cO_{k,r}$ because
of the rank $r$ part of $\cM_{k,r}$.  
We remark that all values of $\cO_{k,r}$ are PID's,
and hence any finitely generated $\cM_{k,r}$ module has values that 
are finite sums of free modules and torsion modules.
We also mention that at times $\cM_{k,r,\ourv 0}$ seems like a
``good substitute'' for
$\cO_{k,r}$; for example, 
$b^0(\cM_{k,r,\ourv 0})=1$ and $b^1(\cM_{k,r,\ourv 0})=r$, which are
the correct Betti numbers for a curve of genus $r-1$ (which is
the genus that Baker-Norine define for the graph on two vertices
joined by $r$ edges).
Another interesting question of Illusie is whether or not for genus
$r-1\ge 1$ our construction chooses a particular curve (or algebraic
space, etc.), since there are infinitely many such curves of fixed genus
at least one over an
algebraically closed field.
Also the map 
\begin{equation}\label{eq_shom_map_usually_isom}
\bigl( \SHom_\cO(\cF,\cG) \bigr)_x \to \Hom_{\cO_x}(\cF_x,\cG_x),
\end{equation}
(see [EGA1], Section~0.5.2.6 or \cite{hartshorne}, Proposition~III.6.8)
is not an isomorphism at $x=B_3$ (this corresponds to the clutching
cylinder) for sheaves 
$\cF=\cM_{r,\ourv d}$ and $\cG=\cM_{r,\ourv d'}$ and $r\ge 2$.

\section{Foundations, Part 1: Bipartite Graphs, Categories, and Sheaves}
\label{se_simpleI}

In this section we define the minimum regarding bipartite categories
and sheaves in order to state our main modeling theorem,
Theorem~\ref{th_main_grrr_model}
in Section~\ref{se_models_grr}.
Beyond this we give some common terminology regarding bipartite graphs,
and give a few examples of sheaves on bipartite categories.


\subsection{Bipartite Graphs}

\begin{definition}
By a {\em bipartite graph} we mean a triple $G=(\cA,\cB,E)$ where
$\cA,\cB$ are finite sets---the 
sets of {\em left vertices of $G$} and
{\em right vertices $G$} respectively---and $E\subset\cA\times\cB$, called
the {\em edges of $G$}.
We refer to the disjoint union $\cA\amalg\cB$\footnote{
  The notation $\cA\amalg\cB$ refers to the disjoint union of
  $\cA$ and $\cB$.  One can define this to be the subset of
  $(\cA\cup\cB)\times\{0,1\}$ consisting
  $\cA\times\{0\}\cup \cB\times \{1\}$.
  If $\cA,\cB$ are disjoint sets, we view the disjoint union more
  simply as $\cA\cup\cB$.  Technically, $\cA\amalg\cB$ is a limit,
  defined only up to unique isomorphism.
} as the {\em set of vertices of $G$}.
\end{definition}
In graph theory, this would be called, more precisely, a finite bipartite
graph with a given bipartition and without multiple edges.
One could allow for $\cA,\cB,E$ to be arbitrary (possibly infinite) sets, but
in this article all graphs are finite unless otherwise indicated.

\begin{definition}\label{de_complete}
Let $G=(\cA,\cB,E)$ be a bipartite graph.
\begin{enumerate}
\item By a {\em subgraph of $G$}
we mean a graph $G'=(\cA',\cB',E')$ such that
$\cA'\subset\cA$, $\cB'\subset\cB$, and $E'\subset E$
(and therefore $E'\subset E\cap(\cA'\times\cB')$).
In this case we write $G'\subset G$ and say that $G'$ is {\em included}
in $G$.
\item
We say that $a\in\cA$ and $b\in\cB$ are {\em adjacent} if
$(a,b)\in E$.
\item
We say that $G$ is {\em connected} if for any two vertices, $v,v'\in
\cA\amalg \cB$ there is a sequence $v=v_0,v_1,\ldots,v_k=v'$ of
vertices such that $v_i$ and $v_{i+1}$ are adjacent for all $i=0,\ldots,k-1$.
\item
By a {\em connected component} of $G$ we mean a connected subgraph
$G'\subset G$ that is 
maximal with respect to inclusions (i.e., if $G'\subset G''\subset G$ and
$G''$ is connected, then $G''=G'$).
\end{enumerate}
\end{definition}

\subsection{Bipartite Categories}

\begin{definition}\label{de_simple}
Let $G=(\cA,\cB,E)$ be a bipartite graph.  
The {\em bipartite category associated to $G$} is the
category $\cC$ given as follows:
\begin{enumerate}
\item
its set of objects, $\Ob(\cC)$, is the set $\cA\amalg\cB$ of vertices
of $G$;
\item
its morphisms, $\Fl(\cC)$, consist of identity morphisms and 
$E$, where each $(a,b)\in E=\Fl(\cC)$
has source $a$ and target $b$.
\end{enumerate}
\end{definition}
Note that there is no need to define composition for simple categories,
since if two morphisms are composable then at least one morphism is
an identity morphism.

\begin{definition} 
By the {\em bipartite category $(\cA,\cB,E)$} we mean the category,
$\cC$, associated 
to a bipartite graph $G=(\cA,\cB,E)$ as above.  We refer to
$\cA$ in $\cA\amalg \cB$ as the {\em left objects} of $\cC$, and
similarly to $\cB$ as the {\em right objects}.
When we speak of $(a,b)\in E$, we view $b\in \cB$ as the corresponding right
object of $\cC$, and $a\in\cA$ as the corresponding left object.
\end{definition}
By our conventions $\cA,\cB$ are finite sets unless otherwise indicated.

\subsection{Sheaves of $k$-Vector Spaces on Bipartite Categories}

\begin{definition}\label{de_sheaves_of_vs}
Let $k$ be a field, $G=(\cA,\cB,E)$ a bipartite graph, and
$\cC$ the associated bipartite category.
By a {\em sheaf (of $k$-vector spaces) on $\cC$} we mean the
data, $\cF$, consisting of:
\begin{enumerate}
\item a $k$-vector space $\cF(P)$ for each object, 
$P\in\Ob(\cC)=\cA\amalg\cB$, and
\item for each non-identity morphism of $\cC$, i.e., each
$e=(a,b)\in E$, a linear map
$\cF(e)\from \cF(b)\to\cF(a)$.
\end{enumerate}
\end{definition}

\begin{example}\label{ex_constant_sheaf}
Consider any any $k,G,\cC$ in Definition~\ref{de_sheaves_of_vs}.
If $M$
is a $k$-vector space, then we define the {\em constant sheaf $M$},
denoted $\underline M$, to be the sheaf all of whose values are $M$,
and all of whose restrictions are the identity maps.  In particular
$\underline k$ is sheaf whose values are all $k$ and whose restrictions
are the identity map.
\end{example}

\begin{example}\label{ex_restrict_extend_zero}
Consider any any $k,G,\cC$ in Definition~\ref{de_sheaves_of_vs}, and
consider a subgraph 
$G'=(\cA',\cB',E')$ of $G$.  If $\cF$ is a sheaf on $G$, then we use
$\cF_{G'}$ to denote the following sheaf: its values and restrictions
are given by those of $\cF$ for vertices and edges in $G'$, and its
other values are $0$, and other restriction maps the zero map.
We call $\cF_{G'}$ the {\em extension by zero of
the restriction of $\cF$ to $G'$}.
\end{example}
The above example will be elaborated upon in
Definition~\ref{de_restriction_extension}.
The sheaves $\underline k_{G'}$ were a fundamental building block of the
sheaves in
\cite{friedman_memoirs_hnc}.

\subsection{Global Sections and the Zeroth Betti Number}

\begin{definition}
Let $\cF$ be a sheaf of $k$-vector spaces
on $\cC$, the bipartite category $(\cA,\cB,E)$.
By {\em global section} of $\cF$ we mean the data consisting
of elements $f_P\in \cF(P)$ for each $P\in\Ob(\cC)=\cA\amalg\cB$ such that for
each $(a,b)\in E$ we have $\cF(e)\cF(b)=\cF(a)$.
We denote the set of global sections by $\Gamma(\cC,\cF)$ or $\Gamma(\cF)$;
this is a $k$-vector space in evident fashion, and we define the
{\em zero-th Betti number of $\cF$}, denoted $b^0(\cF)$, to be
$$
b^0(\cF) = \dim_k \Gamma(\cF).
$$
\end{definition}

\begin{example}
If $k$ is a field and $\cC$ the bipartite category $(\cA,\cB,E)$,
then $b^0(\underline k)$ is the number of connected components of $G$.
More generally,
if $G'\subset G$ is a subgraph, then $b^0(\underline k_{G'})$ is the
number of connected components of $G'$.
\end{example}

\section{Foundations, Part 1.5: Topological Sheaf Theory}
\label{se_simpleI.5}

The following section is independent of the rest of this article.
However, this section will motivate some of the definitions in this
paper, especially that of {\em sheaf Hom} given in Section~\ref{se_simpleIII}.

In this section we assume the readier is familiar with sheaves on
topological spaces (e.g., \cite{hartshorne}, Section~2.1), and describe
the fundamentals of bipartite graphs, categories, and sheaves from this
point of view.

To fix ideas we first work with $\CTwoV$, and then make some remarks
on the more general case.

\begin{definition}
By the {\em two vertex graph Riemann-Roch topological space},
denoted $\XTwoV$, we mean the topological space on the five point set
$X=\XTwoV=\{ A_1,A_2,B_1,B_2,B_3 \}$ with a basis for the topology given
by the subsets:
\begin{enumerate}
\item for $i=1,2$
$$
U_{A_i}=\{A_i\}, \quad 
U_{B_i}=\{A_i,B_i\}, 
$$
and
\item 
$$
U_{B_3} = \{A_1,A_2,B_3\}.
$$
\end{enumerate}
\end{definition}
Figure~\ref{fi_CTwoVTop} shows the five points and $X$ and the basis
$\{U_x\}_{x\in X}$ for its topology.
\begin{figure}[h]  \centering 
\CTwoVTop
\caption{$X$ and the basis $\{U_x\}_{x\in X}$ of its topology.}
\label{fi_CTwoVTop}
\end{figure}
Figure~\ref{fi_two_cats} a diagram of this basis $\{U_x\}_{x\in X}$
and arrow representing the inclusions between these sets, and
the diagram of the points of $X$ and arrows indicating which
points are specializations of other points (these two diagrams are
essentially the same).
\begin{figure}[h]  \centering
\begin{tikzcd} 
U_{A_1} \arrow[r, rightarrow] \arrow[rd] & U_{B_1} \\
                                      & U_{B_3}\\
U_{A_2} \arrow[r, rightarrow] \arrow[ru] & U_{B_2}
\end{tikzcd}
\quad\quad
\begin{tikzcd}
A_1 \arrow[r, rightarrow] \arrow[rd] & B_1 \\
                                      & B_3 \\
A_2 \arrow[r, rightarrow] \arrow[ru] & B_2
\end{tikzcd}
\caption{Two (isomorphic) categories: 
The basis $\{U_x\}_{x\in X}$ under inclusion,
and the points of $X$ under specialization.}
\label{fi_two_cats}
\end{figure}

These diagrams are bipartite graphs, and the associated bipartite
category (which we define as $\CTwoV$ in Section~\ref{se_models_grr})
will be the underlying category used for all the main results in this paper.

The category $\CTwoV$ corresponds to either of these diagrams
(that are essentially the same).
Figure~\ref{fi_inclusions}
depicts the entire collection of open subsets of $X$, with arrows 
representing inclusions.

\begin{figure}[h] \centering 
\inclusions 
\caption{The inclusion maps of the topological space $X$. Red sets are elements of the basis.  }
\label{fi_inclusions}
\end{figure}

It is not hard to check that (1) for each $x\in X$, $U_x$ is the smallest
open subset of $X$ containing $x$ (so the value of a sheaf $\mathcal{F}$ on $U_x$ represents the stalk of $\mathcal{F}$ at $x$)
and (2) the $\{U_x\}_{x\in X}$ are precisely the open subsets that are
{\em irreducible} in the sense that they cannot be written as a union
of proper subsets\footnote{The empty set is the empty union and hence
 is not irreducible; the reader may alternatively just accept that by
 definition the empty set is not irreducible.
}

A sheaf, $\cF$, on $X$ in the classical sense (see Figure~\ref{fi_inclusions})
gives rise to a sheaf $\cF'$ in our sense by restricting its values and
restrictions to the subcategory of the open sets $\{U_x\}_{x\in X}$,
which is the category $\CTwoV$.
This gives a ``restriction functor,'' and we easily verify that this
is an equivalence of categories as follows:
given any sheaf $\cF'$ on $\CTwoV$ in our sense, 
we can extend $\cF'$ to a sheaf $\cF$ on $X$ as follows: 
(1) given any open subset
$W\subset X$, let $\CTwoV|_W$ be the full subcategory of $\CTwoV$ on the
objects $x\in W$;
(2) define $\cF(W)$ to be the
limit of the functor (as in [SGA4] Definition~I.2.1)
that is the restriction of $\cF'$ to
$\CTwoV|_W$ 
(this limit exists and is unique up to unique
isomorphism);
(3) we check that $\cF$ is a sheaf on $X$ in the classical sense on $\cF$;
(4) clearly the restriction of this $\cF$ to $\CTwoV$ is just $\cF'$;
(5) we check that the ``restriction functor'' is fully faithful.

Let us discuss the relative advantages of working with a classical sheaf
$\cF$ as opposed with its restriction $\cF'$ to $\CTwoV$, which is the
notion we use in this article (and used in \cite{friedman_memoirs_hnc}).
For one, $\Gamma(X,\cF)=\Gamma(\CTwoV,\cF')$ is simply $\cF(X)$,
which provides useful intuition and explains the term {\em global section}.
On the other hand, organizing the information as $\cF'$ seems simpler
than the classical notion; also the construction of cokernels and other
limits is done ``value-by-value'' in our notion; classically one needs
to sheafify certain limits.

One pedagogical advantage of the classical notion of sheaf is that
it is quite commonly used in many areas of mathematics,
and there are many excellent expositions that describe a number of concepts
that work well in practice.
For example, we know what is good definition of
{\em sheaf Hom} (see, e.g., \cite{hartshorne}, Exercise~II.1.15) and
what is meant by saying that to verify that an evident morphism
$$
\SHom(\cM_{k,r,\ourv d},\cM_{k,r,\ourv d'}) \to \cM_{k,r,\ourv d'-\ourv d}
$$
is an isomorphism; it suffices to check this {\em locally}.

One advantage of our notion of sheaf, i.e., working with
presheaves on a category $\cC$ endowed with the {\em topologie 
grossi\`ere}, is that it is more general (than finite topological spaces).
There are finite structures, such as graphs with {\em whole-loops}
and {\em half-loops}, and {\em graphs of groups}, which are not topological
spaces, where topos theory gives us generalizations of definitions
that tend to work well;
for example, with the {\em topologie grossi\`ere}, the notion of
{\em locally} at an object $P\in\Ob(\cC)$ means we consider
the {\em slice category} $\cC/P$ (i.e., of objects {\em over $P$})
and the functor $\cC/P\to\cC$.
This tells us how to define sheaf Hom and related notions more generally
for the aforementioned finite structures, and explains why locally
these structures have the same structure as graphs;
we shall make some more remarks on this in
Section~\ref{se_future}.

\section{Sheaves Modeling the Graph Riemann-Roch Rank
for a Graph With Two Vertices}
\label{se_models_grr}

In this section we will define a simple category $\CTwoV$ and
a class of sheaves, $\cM_{k,r,\ourv d}$, on $\CTwoV$, depending on
a field, $k$, an integer $r\ge 1$, and a $\ourv d=(d_1,d_2)\in\integers^2$.
The main theorem of this section is that
$$
b_0(\cM_{r,\ourv d}) -1 = {\rm GRRR}_r(d_1,d_2),
$$
where the right-hand-side denotes
the Baker-Norine
{\rm graph Riemann-Roch rank} of the {\em divisor} $\ourv d$
on a graph with two vertices.
We begin by defining the $\cM_{k,r,\ourv d}$ in two ways,
and then review the
Baker-Norine notion of rank for a divisor on a graph.

\subsection{The Definition of $\CTwoV$, $\cO_{k,r}$, and $\cM_{k,r,\ourv d}$}

In this section we give two equivalent definitions of $\cM_{r,\ourv d}$.

\subsubsection{Definition of $\cM_{r,\ourv d}$ as Sheaves of Vector Spaces}

\begin{definition}\label{de_CTWoV}
Let $\CTwoV$ be the bipartite category $(\cA,\cB,E)$ where
$$
\cA=\{A_1,A_2\}, \quad \cB=\{B_1,B_2,B_3\}
$$
and
$$
E = \{ (A_1,B_1), \ (A_2,B_2),\ (A_1,B_3),\ (A_2,B_3) \}.
$$
\end{definition}
We depict $\CTwoV$ in
in Figure~\ref{fi_cat_Or}.

Recall that if $k$ is a field and $x$ an indeterminate, then $k[x]$ denotes
the polynomials in $x$ with coefficients in $k$;
and similarly, $k[x,1/x]$ denotes the polynomials in $x$ and $1/x$, i.e.
the Laurent polynomials in $x$.

\begin{definition}\label{de_model_grrr_two}   
For a field, $k$, integer $r\ge 1$, and $\ourv d\in\integers^2$, we define
a sheaf of $k$-vector spaces $\cM=\cM_{k,r,\ourv d}$ on $\CTwoV$ as follows:
\begin{enumerate}
\item its values are
$$
\cM(B_1)=k[y_1], \quad \cM(B_2)=k[y_2], \quad \cM(B_3)=k[v,1/v]^{\oplus r}
$$
(i.e., the direct sum of $r$ copies of the $k$-vector space $k[v,1/v]$),
$$
\cM(A_1)=k[x_1,1/x_1], \quad \cM(A_2)=k[x_2,1/x_2],
$$
where $x_1,x_2,y_1,y_2,v$ are indeterminates;
\item its restriction maps are given as follows:
\begin{enumerate}
\item for $i=1,2$, $\cM(A_i,B_i)(p(y_i))=x_i^{d_i}p(1/x_i)$;
\item $\cM(A_1,B_3)(q_1(v),\ldots,q_r(v)) = 
q_1(x_1^r)+x_1 q_2(x_1^r) + \cdots+ x_1^{r-1} q_{r-1}(x_1^r)$;
\item $\cM(A_2,B_3)(q_1(v),\ldots,q_r(v)) = 
q_1(x_2^{-r})+ x_2 q_2(x_2^{-r})+\cdots+ x_2^{r-1} q_{r-1}(x_2^{-r})$.
\end{enumerate}
\end{enumerate}
\end{definition}
We will often suppress the field $k$ and the integer $r$ in the notation
$\cM_{k,r,\ourv d}$.
All the results in this article are independent of $k$ (which is not assumed
to be algebraically closed or even infinite).

\subsubsection{Intuitive Definition of $\cM_{k,r,\ourv d}$ Via
$\cO_{k,r}$}

Here we give a simpler and more intuitive way to describe the 
$\cM_{k,r,\ourv d}$ in terminology that we now explain.
We view $\cM_{k,r,\ourv d}$ (in Figure~\ref{fi_Mkrd})
as a {\em sheaf of $\cO_{k,r}$-modules} for 
a {\em sheaf of $k$-algebras} $\cO_{k,r}$ (see Figure~\ref{fi_cat_Or})
that we now define.

\begin{figure}
$$
\begin{tikzcd}
A_1 \arrow[r, rightarrow] \arrow[rd] & B_1 \\
    			              & B_3 \\
A_2 \arrow[r, rightarrow] \arrow[ru] & B_2 
\end{tikzcd}
\qquad
\begin{tikzcd} 
\cO(A_1)=k[x_1,1/x_1] 
\arrow[rr, leftarrow, "1/x_1\, \leftarrow\!\shortmid\; y_1"] 
&\qquad & k[y_1] =\cO(B_1) \\
&& \arrow[lld, "x_2^{-r} \, \leftarrow\!\shortmid\; v"'] 
  \arrow[llu, "x_1^{r} \, \leftarrow\!\shortmid\; v"] 
  k[v,1/v]  = \cO(B_3)
\\
\cO(A_2)=k[x_2,1/x_2]  
\arrow[rr, leftarrow, "1/x_2\, \leftarrow\!\shortmid\; y_2"] && 
k[y_2]=\cO(B_2)
\end{tikzcd}
$$
\caption{Depiction of $\cC=\cC_{\rm 2V}$ and $\cO=\cO_{k,r}$.}
\label{fi_cat_Or}
\end{figure}

\begin{figure}
$$
\begin{tikzcd} 
\cM(A_1)=k[x_1,1/x_1] 
\arrow[rr, leftarrow, "x_1^{d_1}\, \leftarrow\!\shortmid\; 1"] 
&\qquad & k[y_1]=\cM(B_1)  \\
&& \arrow[lld, "x_2^{j-1} \, \leftarrow\!\shortmid\; e_j"'] 
  \arrow[llu, "x_1^{j-1} \, \leftarrow\!\shortmid\; e_j"] 
  k[v,1/v]^{\oplus r} = \cM(B_3)
\\
\cM(A_2)=k[x_2,1/x_2] 
\arrow[rr, leftarrow, "x_2^{d_2}\, \leftarrow\!\shortmid\; 1"] && 
k[y_2]=\cM(B_2)
\end{tikzcd}
$$
\caption{Depiction of $\cM=\cM_{k,r,\ourv d}$ as an $\cO=\cO_{k,r}$-module.}
\label{fi_Mkrd}
\end{figure}

\begin{definition}\label{de_cO_r}
For a field, $k$ and an integer $r\ge 1$
we define 
$\cO_{k,r}$
as follows: its values are
the $k$-algebras
$$
\cO_{k,r}(B_1)=k[y_1], \quad \cO_{k,r}(B_2)=k[y_2], \quad \cO_{k,r}(B_3)=k[v,1/v] ,
$$
$$
\cO_{k,r}(A_1)=k[x_1,1/x_1], \quad \cO_{k,r}(A_2)=k[x_2,1/x_2],
$$
where $x_1,x_2,y_1,y_2,v$ are indeterminates and
the restrictions maps are the unique morphisms of $k$-algebras for which we have
$$
y_1 \mapsto 1/x_1, \quad y_2\mapsto 1/x_2, \quad
v \mapsto x_1^r, \quad v\mapsto x_2^{-r}.
$$
\end{definition}

Hence, the sheaf $k$-vector spaces $\cO_{k,r}$ can be viewed as a {\em sheaf
of rings} or as a {\em sheaf of $k$-algebras} in that its values and
restriction maps can be viewed in the category of rings or of
$k$-algebras.

In terms of $\cM=\cM_{k,r,\ourv d}$, for each
$P\in\Ob(\cC)=\cA\amalg\cB$, $\cM(P)$ is an $\cO(P)$-module
(free of rank $r$ for $P=B_3$ and otherwise free of rank $1$), and the
restriction maps of $\cM$ satisfy:
\begin{enumerate}
\item $\cM(A_i,B_i)(1)=x_i^{d_i}$ for $i=1,2$;
\item $\cM(A_i,B_3)(e_j) = x_i^{j-1}$ for $i=1,2$ and $j=1,\ldots,r$
where $e_j$ is the standard basis vector of $\cO(B_3)^{\oplus r}$
(that is $1$ in the $j$-th component, and $0$ elsewhere).
\end{enumerate}
This information uniquely determines the restriction maps of $\cM$ provided
that we insist that $\cM$ is a {\em sheaf of $\cO_{k,r}$-modules},
in a sense that we will formalize and discuss at length
in Section~\ref{se_simpleIII}.
Formally this means that for each $(a,b)\in E$, $r\in\cO(b)$ and
$m\in\cM(b)$,
\begin{equation}\label{eq_sheaf_of_mods_bi}
\cM(a,b) (rm) = \bigl( \cO(a,b)r\bigr) \, \bigl( \cM(a,b)m\bigr).
\end{equation}
Intuitively this means that $\cM(a,b)$ is compatible with the 
restriction map $\cO(a,b)$ and the module structures of the values of $\cM$.

\subsection{The Graph Riemann-Roch Rank of Baker and Norine}

In this section we review the Baker-Norine notion of
{\em Graph Riemann-Roch Rank}; see \cite{baker_norine,baker} for details.

Let $G=(V,E)$ be a graph (we allow multiple edges but not self-loops).
The Laplacian of $G$ is therefore 
a morphism $\integers^V\to\integers^V$. 
By a {\em divisor} we mean an element of $\integers^V$, and we say that
two divisors are {\em equivalent} if their difference is in the image
of the Laplacian. We say that a divisor is {\em effective} if all its 
components are non-negative.  If $\ourv d\in\integers^V$ is a divisor,
we define its {\em graph Riemann-Roch rank}, ${\rm GRRR}_G(\ourv d)$,
to be $-1$ if $\ourv d$ is not effective, and otherwise
to be the largest non-negative 
integer, $m$, such that for all $\ourv m \ge 0$, $\ourv m \in  \integers^V$, whose sum of components is
$m$,
we have that $\ourv d-\ourv m$ is effective.

In this article we restrict our attention to the case where $V=\{v_1,v_2\}$
and $E$ consists of $r$ edges joining $v_1$ and $v_2$.
In this case we identify $\integers^V$ with $\integers^2$ arbitrarily;
by symmetry, the rank of $(d_1,d_2)$ is the same as that of $(d_2,d_1)$,
so the rank is defined unambiguously.

\begin{definition}
Let $r\ge 1$ be an integer and $\ourv d=(d_1,d_2)\in\integers^2$.
We use
${\rm GRRR}_G(\ourv d)$
to denote the graph Riemann-Roch rank of $\ourv d$ as a divisor
on a graph of two vertices joined by $r$ edges
(in the Baker-Norine \cite{baker_norine} sense above).
\end{definition}
Hence, ${\rm GRRR}_G(\ourv d)$ depends only on the class of $\ourv d$
in $\integers^2/(r,-r)\integers$.

We may equivalently define
${\rm GRRR}_G(\ourv d)$ inductively as:
\begin{enumerate}
\item 
$-1$ if $d_1+d_2\le -1$, or if $d_1+d_2=0$ and $r$ does not divide $d_1=-d_2$;
\item 
$0$ if $d_1+d_2=0$ and $r$ does divide $d_1=-d_2$;
\item
$$
1+\min\bigl(  {\rm GRRR}_G(d_1-1,d_2),  {\rm GRRR}_G(d_1,d_2-1) \bigr)
$$
inductively for $d_1+d_2\ge 1$.
\end{enumerate}

\subsection{The Main Modeling Theorem}

The following theorem is the main result of this section.

\begin{theorem}\label{th_main_grrr_model}
Let $k$ be any field, $r\ge 1$ an integer, and 
$\ourv d=(d_1,d_2)\in\integers^2$.  Then
$$
b_0\bigl( \cM_{k,r,\ourv d}) = {\rm GRRR}_r(d_1,d_2) + 1 
$$
with $\cM_{k,r,\ourv d}$ as in
Definition~\ref{de_model_grrr_two}, and GRRR as above.
\end{theorem}

We arrived at this theorem by ad hoc process of
``trial and error,'' rather than
a systematic method; there may be other ways to algebraically model
graph Riemann-Roch questions.

The goal of the rest of this section is to prove 
Theorem~\ref{th_main_grrr_model}.

\subsection{A Formula for the Graph Riemann-Roch Rank for a Graph with
Two Vertices}

Let us give an explicit formula for the graph Riemann-Roch rank.

\begin{theorem}\label{th_GRRR_formula}
With notation as above, let
$\ourv d\in\integers^2$ with $0\le d_2\le r-1$.  Then
\begin{enumerate}
\item if $d_1\le -1$, we have ${\rm GRRR}_r(\ourv d)=-1$;
\item if $0\le d_1\le r-1$, we have ${\rm GRRR}_r(\ourv d)=\min(d_1,d_2)$;
\item if $r-1\le d_1$, we have
${\rm GRRR}_r(\ourv d)=d_1+d_2-r+1$.
\end{enumerate}
\end{theorem}
We remark that any divisor, $\ourv d$, is linearly equivalent to one where
$0\le d_2\le r-1$ by adding the appropriate multiple of $(r,-r)$; for this
reason the above theorem essentially determines all values of
${\rm GRRR}_r$.
\begin{proof}
The first claim is easy: if $(d_1,d_2)+m(r,-r)$ is effective, then we need
$m\ge 1$ in view of the first component, but then $d_2+m(-r)<0$ so that
$(d_1,d_2)+m(r,-r)$ can never be effective.

To prove the second claim, let $k=\min(d_1,d_2)\ge 0$.  To see that
${\rm GRRR}_r(d_1,d_2)\ge k$, we see that if $\ourv m\ge\ourv 0$
and $m_1+m_2=k$, then $m_i\le k$ for $1,2$, and hence $\ourv d-\ourv m\ge 0$,
which is effective; hence ${\rm GRRR}_r(d_1,d_2)\ge k$.
Let us show that equality holds: 
by symmetry we may assume that $d_1\le d_2$, and hence $d_1=k$; then
$(d_1,d_2)-(k+1,0)=(-1,d_2)$,
and by the first claim this divisor is not linearly equivalent to an
effective divisor.  
Since $(k+1,0)$ is effective and has sum of components equal to $k+1$,
we have ${\rm GRRR}_r(d_1,d_2)<k+1$.
Hence ${\rm GRRR}_r(d_1,d_2)=k$.

To prove the third claim, let us first show that in this case
\begin{equation}\label{eq_ineq_grrr1}
{\rm GRRR}_r(d_1,d_2) < d_1+d_2-r+2.
\end{equation} 
Consider 
$\ourv m=(d_1-r+1,d_2+1)$, which is clearly effective; we have
$\ourv d-\ourv m=(r-1,-1)$, which by the first claim is not linearly 
equivalent to an effective divisor, and hence \eqref{eq_ineq_grrr1}.
Let us now show that
\begin{equation}\label{eq_ineq_grrr2}
{\rm GRRR}_r(d_1,d_2) \ge d_1+d_2-r+1.
\end{equation} 
It suffices to
consider a divisor $\ourv m\ge 0$ with 
$m_1+m_2= d_1+d_2-r+1$ and to show that $\ourv d-\ourv m$ is equivalent
to an effective divisor.
But $\ourv m$ is equivalent to
a divisor $\ourv m'$ with $d_2-r+1\le m_2'\le d_2$, and in this case
$m_1+m_2=m_1'+m_2'$ implies that 
$m_1'= d_1+d_2-r+1 -m_2'$.  Hence
$\ourv d-\ourv m'=(r+m_2'-1-d_2,d_2-m_2')$, whose second component is clearly
non-negative, and whose first component is
$$
r+m_2'-1-d_2 =r-2+(m_2'-d_2) \ge r-1+(1-r)\ge  0.
$$
Hence $\ourv d-\ourv m'$ is effective, and it is 
equivalent to $\ourv d-\ourv m$.
This establishes \eqref{eq_ineq_grrr2}.
Together with \eqref{eq_ineq_grrr1} we have proven the third claim.
\end{proof}

\subsection{Alternative Description of ${\rm GRRR}_r(\ourv d)$}

Here is another description of ${\rm GRRR}_r(\ourv d)$ that is easier
to visualize and may be simpler to work with when generalizing
the methods in this article to graphs with more than two vertices.

\begin{theorem}\label{th_GRRR_S_0}
For an integer $r\ge 1$ and $i\ge 0$,
let $S_i$ denote the subset of
points $(d_1,d_2)\in\integers^2$ for which
${\rm GRRR}_r(d_1,d_2) = i$.
Then
\begin{enumerate}
\item for $0\le d_2\le r-1$, $S_0$ consists of those $(d_1,d_2)$
with either $d_1=0$ or $d_2=0$;
\item if $\ourv d$ lies ``to the left of $S_0$,'' 
meaning that $(a,d_2)\in S_0$ for some
$a>d_1$ but $(d_1,d_2)\notin S_0$,
then ${\rm GRRR}_r(\ourv d)=-1$; and
\item ${\rm GRRR}_r(d_1,d_2)$ is the $L^1$ distance to $S_0$ if $\ourv d$
lies ``to the right of $S_0$.''
\end{enumerate}
\end{theorem}
We illustrate this theorem in Figure~\ref{fi_r_4};
as above, we reduce to the case $0\le d_2\le r-1$, which in Figure~\ref{fi_r_4}
corresponds to
the $r$ (horizontal) rows beginning with the $d_1$-axis $d_2=0$, and the
$r-1$ rows above this axis.
\begin{proof}
Since ${\rm GRRR}_r(\ourv d)$ and therefore $S_0$ are invariant under
adding multiples of $(r,-r)$, it suffices to prove them in the case where
$0\le d_2\le r-1$.

Claim~(1) is immediate from claim~(2) of Theorem~\ref{th_GRRR_formula}.
Claim~(2) is immediate from claim~(1) of Theorem~\ref{th_GRRR_formula}.
Claim~(3) is immediate from the definition of the Graph Riemann-Roch
Rank.
\end{proof}


\begin{figure}[ht]
  \centering
  \begin{tikzpicture}[scale=0.6]
    \coordinate (Origin)   at (0,0);
    \coordinate (XAxisMin) at (-5,0);
    \coordinate (XAxisMax) at (10,0);
    \coordinate (YAxisMin) at (0,-5);
    \coordinate (YAxisMax) at (0,10);
    \draw [thin, gray,-latex] (XAxisMin) -- (XAxisMax);
    \draw [thin, gray,-latex] (YAxisMin) -- (YAxisMax);

    \clip (-5,-5) rectangle (10cm,10cm); 
    \foreach \x in {-12,-11,...,12}{
     \foreach \y in {-12,-11,...,12}{
      \node[draw,circle,inner sep=2pt,fill] at (1*\x,1*\y) {};
      }
    }
\draw [ultra thick,blue] (Origin)
        -- (0,3) node [below left] {$S_0$};
    \draw [ultra thick,blue] (0,3)
        -- (-1,4) ;
     \draw [ultra thick,blue] (-1,4)
        -- (-4,4) ;
     \draw [ultra thick,blue] (-4,4)
        -- (-4,7) ;
      \draw [ultra thick,blue] (-4,7)
        -- (-5,8) ;
      \draw [ultra thick,blue] (0,0)
        -- (3,0) ;
     \draw [ultra thick,blue] (3,0)
        -- (4,-1) ;
     \draw [ultra thick,blue] (4,-1)
        -- (4,-4) ;
     \draw [ultra thick,blue] (4,-4)
        -- (7,-4) ;
      \draw [ultra thick,blue] (7,-4)
        -- (8,-5) ;

      \draw [ultra thick,blue] (1,1)
        -- (1,3) node [below left] {$S_1$};
     \draw [ultra thick,blue] (1,3)
        -- (-1,5) ; 
      \draw [ultra thick,blue] (-1,5)
        -- (-3,5) ;
      \draw [ultra thick,blue] (-3,5)
        -- (-3,7) ;
       \draw [ultra thick,blue] (-3,7)
        -- (-5,9) ;
       \draw [ultra thick,blue] (0,3)
        -- (-1,4) ;
       \draw [ultra thick,blue] (1,1)
        -- (3,1) ;
       \draw [ultra thick,blue] (3,1)
        -- (5,-1) ;
       \draw [ultra thick,blue] (5,-1)
        -- (5,-3) ;
       \draw [ultra thick,blue] (5,-3)
        -- (7,-3) ;
        \draw [ultra thick,blue] (7,-3)
        -- (9,-5) ;

       \draw [ultra thick,blue] (2,2)
        -- (2,3) node [below left] {$S_2$};
     \draw [ultra thick,blue] (2,3)
        -- (-1,6) ; 
      \draw [ultra thick,blue] (-1,6)
        -- (-2,6) ;
      \draw [ultra thick,blue] (-2,6)
        -- (-2,7) ;
       \draw [ultra thick,blue] (-2,7)
        -- (-5,10) ;
       \draw [ultra thick,blue] (0,3)
        -- (-1,4) ;
       \draw [ultra thick,blue] (2,2)
        -- (3,2) ;
       \draw [ultra thick,blue] (3,2)
        -- (6,-1) ;
       \draw [ultra thick,blue] (6,-1)
        -- (6,-2) ;
       \draw [ultra thick,blue] (6,-2)
        -- (7,-2) ;
        \draw [ultra thick,blue] (7,-2)
        -- (10,-5) ;
  
  \draw [ultra thick,blue] (-4,10)
        -- (10,-4);
  \draw [ultra thick,blue] (-3,10)
        -- (10,-3) ;
  \draw [ultra thick,blue] (-2,10)
        -- (10,-2) ;
  \draw [ultra thick,blue] (-1,10)
        -- (10,-1) ;
  \draw [ultra thick,blue] (-0,10)
        -- (10,-0) ;
  \draw [ultra thick,blue] (1,10)
        -- (10,1) ;
  \draw [ultra thick,blue] (2,10)
        -- (10,2) ;
  \draw [ultra thick,blue] (3,10)
        -- (10,3) ;
  \draw [ultra thick,blue] (4,10)
        -- (10,4) ;
  \draw [ultra thick,blue] (5,10)
        -- (10,5) ;
  \draw [ultra thick,blue] (5,10)
        -- (10,5) ;
  \draw [ultra thick,blue] (6,10)
        -- (10,6) ;
  \draw [ultra thick,blue] (7,10)
        -- (10,7) ;
  \draw [ultra thick,blue] (8,10)
        -- (10,8) ;
   \draw [ultra thick,blue] (9,10)
        -- (10,9) ;

  \node[circle,fill=red] at (0,0){};
  \node[circle,fill=red] at (1,1){};
  \node[circle,fill=red] at (2,2){};
  \node[circle,fill=red] at (3,3){};
   
   \node[circle,fill=red] at (-4,4){};
   \node[circle,fill=red] at (-3,5){};
   \node[circle,fill=red] at (-2,6){};
   \node[circle,fill=red] at (-1,7){};
    
   \node[circle,fill=red] at (4,-4){};
   \node[circle,fill=red] at (5,-3){};
   \node[circle,fill=red] at (6,-2){};
   \node[circle,fill=red] at (7,-1){}; 
  
  \end{tikzpicture}
  \caption{The case $r=4$ in the $(d_1,d_2)$-plane}
  \label{fi_r_4}
\end{figure}

\subsection{Counting Lattice Points in the Down Cone}

In this subsection we prove the following result.

\begin{theorem}\label{th_LatCount}
For an integer $r\ge 1$, let
$$
L_r \eqdef \bigcup_{i=0}^{r-1}  \bigl( \integers (r,-r) + (i,i) \bigr) \ ;
$$
for a $\ourv d\in\integers^2$, let the {\em down cone at $\ourv d$} be
$$
{\rm DownCone}(\ourv d) \eqdef 
\{ \ourv x \in \integers^2 \ |\ \ourv x \le \ourv d\} \ ,
$$
and set
$$
{\rm LatCount}_r(\ourv d) \eqdef  
\bigl| L_r \cap {\rm DownCone}(\ourv d)\bigr| \ .
$$
Then for any $\ourv d\in\integers^2$, we have
\begin{equation}\label{eq_LatCount}
{\rm LatCount}_r(\ourv d) = {\rm GRRR}_r(\ourv d) + 1.
\end{equation} 
\end{theorem}
\begin{proof}
Since the functions of $\ourv d$ in \eqref{eq_LatCount} are
invariant under adding $(r,-r)\integers$, it suffices to prove
\eqref{eq_LatCount} for $0\le d_2\le r-1$.  So assume this; we will
discuss a number of cases.

If $d_1<0$, then $\ourv d$ is not $\ge (i,i)$ for any $i\ge 0$, and neither
can $\ourv d+m(r,-r)$ for any $m\in\integers$, since if $m\le 0$ then
the first component of $\ourv d+m(r,-r)$ is negative, and otherwise the
second component is.  Hence ${\rm LatCount}_r(\ourv d)=0$ in this case,
and hence, using Theorem~\ref{th_GRRR_formula}, \eqref{eq_LatCount} holds
in this case.

Next consider the case where $0\le d_1\le r-1$, and further assume that
$d_2\le d_1$.  The same argument as in the previous paragraph shows
that if $\ourv d+m(r,-r)$ cannot be $\ge (i,i)$ for any $i\ge 0$ unless
$m=0$; if $m=0$ and $0\le d_2\le d_1\le r-1$, then $\ourv d\ge (i,i)$
iff $i\in[d_2]$.  Hence 
$$
{\rm LatCount}_r(\ourv d)=d_2+1=\min(d_1,d_2)+1.
$$
We similarly see that this formula holds if $d_1\le d_2$.  
Theorem~\ref{th_GRRR_formula} now shows that \eqref{eq_LatCount} holds
in this case.

Next consider the case where $r\le d_1 \le 2r-1$.  A similar argument
as above shows that if $\ourv d+m(r,-r)$ cannot be $\ge (i,i)$ for any
$i\ge 0$ unless $m=0$ or $m=-1$.  For $m=0$, we have that
$\ourv d\ge (i,i)$ for $i\in [d_2]$; for $m=1$ we have that
$\ourv d+(-r,r)\ge (i,i)$ iff $i\le [d_1-r]$.  It follows that
$$
{\rm LatCount}_r(\ourv d)= (d_2+1)+(d_1-r+1)=d_1+d_2-r+2
$$
in this case, and hence \eqref{eq_LatCount} holds in this case.

More generally, if $d_1\ge 2r$ and $a\ge 2$ is the unique integer for which
$ar \le d_1 \le (a+1)r-1$, then
$\ourv d+m(r,-r)\ge (i,i)$ requires $m=0,-1,\ldots,-a$, and the of such
$i$ with $i\in[r-1]$ is $d_2+1$ for $m=0$, $d_1-ar+1$ for $m=-a$,
and $r$ for $m=-1,\ldots,-(a-1)$.  Hence the number of such pairs $(m,i)$
is
$$
(d_2+1)+(d_1-ar+1)+(a-1)r = d_1+d_2-r+2,
$$
and again \eqref{eq_LatCount} holds.

Hence \eqref{eq_LatCount} holds for all $d_1$ if $0\le d_2\le r-1$, and
hence it holds for all $\ourv d$.
\end{proof}

Let us sketch another proof of Theorem~\ref{th_LatCount}
based on Theorem~\ref{th_GRRR_S_0}.
We may assume that $0\le d_2\le r-1$.
If $\ourv d$ is to the left of $S_0$, 
we see that there are no lattice points in the down cone
of $\ourv d$.
Similarly, if $\ourv d$ lies on $S_0$, then $(0,0)$
is the single element of $L_r$ in the down cone of $\ourv d$.
Now we argue that the theorem holds for any $\ourv d\in S_i$ for any $i$;
the base case $i=0$ is given above.

For the inductive argument, say that the theorem holds for all 
$\ourv d\in S_i$ for some $i\ge 0$, and let $\ourv d\in S_{i+1}$.
We may assume $d_1\ge d_2$, or else $0\le d_1<d_2\le r-1$ and we may
use symmetry to exchange $d_1$ and $d_2$ and reduce to this case.
Given that $d_2\le d_1$, then $d_2\ge 1$ and 
$\ourv d-(0,1)$ lies on $S_i$; if $d_1\le r-1$, then $\ourv d$ 
contains exactly one more lattice point in its down cone
than $\ourv d-(0,1)$, namely $(d_2,d_2)$.

%

\subsection{Conclusion of the Proof of Theorem~\ref{th_main_grrr_model}}

\begin{proof}[Proof of Theorem~\ref{th_main_grrr_model}.]
By Theorem~\ref{th_LatCount}
it suffices to show that
$$
b^0(\cM_{k,r,\ourv d})=
\dim_k H^0(\cM_{k,r,\ourv d})= {\rm LatCount}(\ourv d)
$$
Each global section $\gamma\in\Gamma(\cM_{k,r,\ourv d})$ has a value
$$
\gamma(B_3) = \bigl( f_1(v),\ldots,f_r(v) \bigr) \in 
\bigl( k[v,1/v] \bigr)^{\oplus r}
$$
which uniquely determines $\gamma$, since $\cM_{k,r,\ourv d}(A_i)$ are
isomorphic to $\cM_{k,r,\ourv d}(B_3)$ as vector spaces, and since the 
restriction maps $\cM_{k,r,\ourv d}(B_i)\to \cM_{k,r,\ourv d}(A_i)$ are injections.
So it remains to check which tuples $(f_1,\ldots,f_r)\in \cM_{k,r,\ourv d}(B_3)$
extend to a global section.  If for $i\in [r]$ we have
$$
f_i = \sum_n c_{i,n} v^n
$$
(so only finitely many of the $c_{i,n}$ are nonzero), then the restriction
maps map $(f_1,\ldots,f_r)$ to the values
$$
g_1(x_1)\eqdef \sum_{n,i} c_{i,n} x_1^{rn+i-1} \in \cM_{k,r,\ourv d}(A_1) ,\quad
g_2(x_2) \eqdef \sum_{n,i} c_{i,n} x_2^{-rn+i-1} \in \cM_{k,r,\ourv d}(A_2) ,
$$
and these restriction maps are isomorphisms;
so this extends to $B_1,B_2$ and hence to a global section iff
for $i=1,2$, $g_i(x_i)$ lies in the image of the restriction map
to $A_i$ from $B_i$; in other words, iff
\begin{equation}\label{eq_global_section_cond}
rn+i-1\le d_1, \quad\mbox{and}\quad -rn+i-1\le d_2
\end{equation}
whenever $c_{i,n}\ne 0$.  It follows that a basis for 
$\Gamma(\cM_{k,r\ourv d})$ is indexed by the pairs
$(i,n)\in \{1,\ldots,r\}\times\integers$ satisfying 
\eqref{eq_global_section_cond}, i.e., for which
\begin{equation}\label{eq_condition_with_L_r}
n(r,-r)+(i-1,i-1) \le \ourv d.
\end{equation} 
But the set of points $n(r,-r)+(i-1,i-1)$ with $i\in\{1,\ldots,r\}$
and $n\in\integers$ is precisely the set $L_r$;
hence the number of pairs $(i,n)$ satisfying 
\eqref{eq_condition_with_L_r}
is precisely ${\rm LatCount}_r(\ourv d)$.
\end{proof}

\section{Foundations, Part 2: Morphisms of Sheaves, Short/Long Sequences}
\label{se_simpleII}

In this section we develop definitions and notions needed to
prove our theorem regarding the
{\em Euler characteristic} of the sheaves $\cM_{k,r,\ourv d}$.

\subsection{First Betti Number}

\begin{definition}\label{de_cohomology}
Let $\cF$ be a sheaf of $k$-vector spaces on the category $\CTwoV$.
We define $H^0(\cF)$ and $H^1(\cF)$ to be, respectively, the kernel and
cokernel of the map
\begin{equation}\label{eq_cohomology}
u\from \bigoplus_{j=1}^3 \cF(B_j)
\to
\bigoplus_{i=1}^2 \cF(A_i)
\end{equation}
where $u$ is given by map
$$
u(m_1,m_2,m_3) = \bigl( \cF(A_1,B_1)m_1 - \cF(A_1,B_3)m_3,
\cF(A_2,B_2)m_2 - \cF(A_2,B_3)m_3 \bigr) \ .
$$
We define the {\em first Betti number} of $\cF$ to be
$$
b^1(\cF) = \dim_k H^1(\cF).
$$
\end{definition}
Note that there is a simple one-to-one correspondence between $H^0(\cF)$ 
and $\Gamma(\cF)$, and hence we define $b^0(\cF)$ as 
$$
b^0(\cF) = \dim_k H^0(\cF) = \Gamma(\cF).
$$
In Section~\ref{se_simpleIII} we will show that the $H^i(\cF)$ above
are the right derived functors of the functor $\cF\mapsto \Gamma(\cF)$.

\subsection{The Short/Long Exact Sequence}

The main technique we will use to prove our Euler characteristic formula
is a standard short/long exact sequence.

\begin{definition}
Let $\cF,\cG$ be sheaves of $k$-vector spaces on the category,
$\cC$, associated to the bipartite graph
$G=(\cA,\cB,E)$.  We define a {\em morphism} $\cF\to \cG$ to be a collection
$u=\{u_P\}_{P\in\Ob(\cC)}$ of $k$-linear maps
$u_P\from\cF(P)\to\cG(P)$ such that
for each $(a,b)\in E$ we have
$u_a\cF(a,b)=\cG(a,b)u_b$.
\end{definition}

This notion of morphism turns the class of sheaves of $k$-vector spaces
on a biparatite category, $\cC$, into a category.
We readily verify that this category is abelian, and all projective
and injective limits (e.g., products, kernels, cokernels) are
computed ``value-by-value;''
for example the cokernel of a morphism $u\from\cF\to\cG$ is simply
the sheaf whose value at $P\in\Ob(\cC)$ is $\cG(P)/\cF(P)$,
and whose restriction maps are obtained from those of $\cF$ and $\cG$
in the natural fashion.
For details, see [SGA4], Section~I.3; our notion of sheaf is
the notion of presheaf in [SGA4] (endowed with the
coarsest topology, i.e., {\em la topologie grossi\`ere}).

\begin{theorem}
If $0\to\cF_1\to\cF_2\to\cF_3\to 0$ is any
short exact sequence of sheaves, then we have a correspond long exact 
sequence of sheaves
$$
0\to
H^0(\cF_1)\to H^0(\cF_2) \to H^0(\cF_3) \to
H^1(\cF_1)\to H^1(\cF_2) \to H^1(\cF_3) \to
0.
$$
\end{theorem}

The proof of this theorem is standard diagram chasing argument.
A second proof of this theorem is given in Section~\ref{se_simpleIII}:
the theorem is a standard result once we have proven that
the $H^i(\cF)$ are the right derived functors
of $\cF\to\Gamma(\cF)$.

\subsection{Restriction and Extension by Zero}

Later we will need some operations on sheaves; the ones we describe
here are {\em exact}, meaning they take short exact sequences to
short exact sequences

\begin{definition}\label{de_restriction_extension}
Let $G'=(\cA',\cB',E')$ be a subgraph of some 
bipartite graph, $G=(\cA,\cB,E)$, and let $\cC',\cC$ be the 
respective associated bipartite categories.
\begin{enumerate}
\item
If $\cF$ is a sheaf on $\cC$, then there is
an evident sheaf
{\em restriction of $\cF$ to $\cC'$}, denoted $\cF|_{G'}$,
which is the sheaf on $\cC'$ whose values and restrictions are taken
from $\cC$.
\item
If $\cF'$ is a sheaf on $\cC'$, we define the {\em extension by zero
of $\cF'$ to $\cC$} we mean the sheaf on $\cC$ whose values and
restrictions are taken from $\cF'$ for objects and morphisms in $\cC'$
and whose remaining values and restrictions are zero.
\item
If $\cF$ is a sheaf on $\cC$, we use $\cF_{\cC'}$ to denote the
sheaf on $\cC$ that is the extension by zero of the restriction of
$\cF$ to $\cC'$.
\end{enumerate}
\end{definition}
We easily see that the three above operations each define functors
that are {\em exact} in the sense that they preserve exactness of
short exact sequences (i.e., sequences $0\to\cF_1\to \cF_2\to \cF_3\to 0$).
Note that the above definition of $\cF_{\cC'}$
is consistent with the notation in Example~\ref{ex_restrict_extend_zero}.

\section{Some Betti Number Computations}
\label{se_some_Betti}

In this section we give a few simple
Betti number computations.
First, we give a procedure to compute $H^1(\cF)$ for various 
sheaves $\cF$, including all the $\cM_{k,r,\ourv d}$, the $\cO_{k,r}$, and
some {\em line bundles} $\cL_{k,r,\ourv d}$ that we define in this
section.
Second, we compute the Betti numbers of a simple type of
{\em skyscraper sheaf} that we will use in the next section.

\subsection{Partial-Line Bundles}

\begin{definition}\label{de_partial_line}
Let $k$ be a vector space and $\ourv d\in\integers^2$.  
By a {\em $\ourv d$-twisted partial-$k$-line bundle 
(on $\CTwoV$)}
(or simply a {\em partial-line bundle})
we mean a sheaf of $k$-vector spaces on $\CTwoV$ such that
\begin{enumerate}
\item for $i=1,2$ we have
$$
\cF(A_i) = k[x_i,1/x_i] ,\quad
\cF(B_i) = k[y_i]
$$
for intederminates $x_1,x_2,y_1,y_2$, and
\item for $i=1,2$ we have
$$
\cF(A_i,B_i) (p(y_i)) = x_i^{d_i} p(1/x_i).
$$
\item there is a basis $\{m_j\}_{j\in J}$ for $\cF(B_3)$ such that
there are injections $f,g$ from $J$ to $\integers$ such that
for each $j\in J$
$$
\cF(A_1,B_3)m_j = x_1^{f(j)}, \quad
\cF(A_2,B_3)m_j = x_2^{g(j)} \ .
$$
\end{enumerate}
\end{definition}
In terms to be given in Section~\ref{se_simpleIII}, such sheaves
$\cF$ are examples of {\em line bundles} (or {\em invertible sheaves}) when
restricted to $\{A_1,A_2,B_1,B_2\}$, and they have a particular form
on $B_3$ and the restriction maps from $B_3$ to the $A_i$.

The sheaves $\cO_{k,r}$ and $\cM_{k,r,\ourv d}$ are all partial-line
bundles.  Let us describe another set of sheaves that will be
of interest to us.

\begin{definition}\label{de_line_L}
Let $k$ be a field, $r\ge 1$ an integer, and $\ourv d\in\integers^2$.
We use $\cL=\cL_{k,r,\ourv d}$ to denote the partial-$k$-line bundle with
$\ourv d$ twists on $\CTwoV$ such that
$\cL=k[v,1/v]$ and 
$$
\cL(A_1,B_3)p(v)=p\bigl(x_1^r\bigr), \quad
\cL(A_2,B_3)p(v)=p\bigl(x_2^{-r}\bigr) 
$$
(setting $J=\integers$ and taking $\{v^j\}_{j\in J}$ as a basis for
$\cL(B_3)$, we see that $\cL$ is indeed a partial-line bundle).
\end{definition}

\subsection{The First Cohomology Formula for Many Cases of Interest}

In this subsection we describe a result that allows us to quickly
compute $H^1(\cF)$ (and hence $b^1(\cF)$) for $\cF$ being any of
$\cM_{k,r,\ourv d},\cO_{k,r}$
and some other sheaves of interest to us in this article.

\begin{lemma}\label{le_b1_comp}
Let $\cF$ be a 
partial $\ourv d $-twisted $\cO_{k,r}$-line bundle, and let
$J,f,g$ be as in Definition~\ref{de_partial_line}.
Then $H^1(\cF)$ (i.e., the cokernel of $u$ 
in \eqref{eq_cohomology}) has a basis consisting of coset representatives
of the following elements of $\cF(A_1)\oplus\cF(A_2)$:
\begin{equation}\label{eq_basis_cokernel}
\bigl\{x_1^i\bigr\}_{i\in I_1} \cup
\bigl\{x_2^i\bigr\}_{i\in I_2} \cup
\bigl\{x_1^{f(j)}\bigr\}_{j\in J_1}
\end{equation}
where
$$
I_1 \eqdef 
\bigl\{ i\in\integers \ |\ \mbox{$i>d_1$ and $i\notin \Image(f)$} \bigr\},
\quad
I_2 \eqdef 
\bigl\{ i\in\integers \ |\ \mbox{$i>d_2$ and $i\notin \Image(g)$} \bigr\},
$$
and
$$
J_1 \eqdef
\bigl\{ j\in J \ |\ \mbox{$f(j)>d_1$ and $g(j)>d_2$} \bigr\} \ .
$$
(For any $j\in J_1$, we could use $x_2^{g(j)}$ instead of $x_1^{f(j)}$
in the basis \eqref{eq_basis_cokernel}.)
\end{lemma}
\begin{proof}
A basis of $\cF(A_1)\oplus\cF(A_2)$
(see \eqref{eq_cohomology}) is given by all the $x_1^i$ and all the $x_2^i$
(with $i$ ranging over all $i\in\integers$);
the image of $u$ consists of the sum of the image of $u$ restricted
to $\cF(B_1)$, to $\cF(B_2)$, and to $\cF(B_3)$.
The image of $u$ restricted to $\cF(B_1)$ is $x_1^i$ with $i\le d_1$;
hence $x_1^i\sim 0$ for $i\le d_1$, where $\sim$ denotes equivalence
modulo the image of $u$.
Similarly for $\cF(B_2)$, so that $x_2^i\sim 0$ for $i\le d_2$.
Finally image of $u$ restricted to $\cF(B_3)$ implies that for each $j\in J$
$$
x_1^{f(j)}\sim - x_2^{g(j)}  \ .
$$
It follows that any element of $\cF(A_1)\oplus\cF(A_2)$ is equivalent to
an element of the form
\begin{equation}\label{eq_cokernel_reps}
\sum_{i\in I_1} \alpha_i x_1^i  \ + \  %
\sum_{i\in I_2} \beta_i x_2^i  \ + \  %
\sum_{j\in J_1} \gamma_j x_1^{f(j)}  
\end{equation} 
with $\alpha_i,\beta_i,\gamma_j$ in $k$.

To finish the proof it suffices to show that the vectors in
\eqref{eq_basis_cokernel} are linearly independent modulo the image of $u$.
In other words we must show that
\begin{equation}\label{eq_basis_independence}
\sum_{i\in I_1} \alpha_i x_1^i  \ + \  %
\sum_{i\in I_2} \beta_i x_2^i  \ + \  %
\sum_{j\in J_1} \gamma_j x_1^{f(j)} \sim 0
\end{equation} 
implies that
all the $\alpha_i,\beta_i,\gamma_j$ are equal to zero.  
So consider an equation of the form \eqref{eq_basis_independence};
the condition $\sim 0$ is equivalent to belonging to the image of $u$,
and therefore
$$
\sum_{i\in I_1} \alpha_i x_1^i  \ + \  %
\sum_{i\in I_2} \beta_i x_2^i  \ + \  %
\sum_{j\in J_1} \gamma_j x_1^{f(j)} 
=
\sum_{i\le d_1} \mu_i x_1^i \ + \  %
\sum_{i\le d_2} \mu_i x_2^i \ + \  %
\sum_{j\in J} x_1^{f(j)}+x_2^{g(j)}
$$
\end{proof}

\subsection{The First Betti Numbers 
of $\cM_{r,\ourv d},\cO_{k,r},\cL_{k,r,\ourv d}$}

Let us state some easy consequences of Lemma~\ref{le_b1_comp}.

\begin{corollary}\label{co_important_first_Betti}
Let $k$ be a field, $r\ge 1$ an integer, and $\ourv d\in\integers^2$.
Then 
\begin{enumerate}
\item if $d_1,d_2\ge 0$, then
$$
b^1(\cM_{k,r,\ourv d})=\max\bigl( 0, r-1-\max(d_1,d_2) \bigr) ;
$$
\item for any $r\ge 2$ we have that $H^1(\cL_{k,r,\ourv d})$
(and therefore $H^1(\cO_{k,r})$) is infinite dimensional.
\end{enumerate}
\end{corollary}
\begin{proof}
For the first claim, i.e., for $\cM_{k,r,\ourv d}$, we have that
$J=\{1,\ldots,r\}\times\integers$, and $f(i,n)=i-1+rn$, $g(i,n)=i-1-rn$.
It follows that (1) $f,g$ are bijections, (2) $f(i,0)=g(i,0)=i-1$,
(3) for $n\ge 1$, $f(i,n)>0$ and $g(i,n)\le 0$, and 
(4) for $n\le -1$, $f(i,n)\le 0$ and $g(i,n)>0$.
In particular, $f(i,n)>d_1\ge 0$ and $g(i,n)>d_2\ge 0$ can only occur
for $n=0$ and $1\le i\le r$.  Furthermore the number of $i$ between $1$
and $r$ for which $i-1=f(i,0)>d_1$ and $i-1=g(i,0)>d_2$ is precisely
the number of integers $i$ with $2+\max(d_1,d_2)\le i\le r$, 
of which there are 
$$
\max\bigl( 0, r-1-\max(d_1,d_2) \bigr).
$$

For the second claim, if $r\ge 2$, then we get the same functions $f,g$
as before, but they are restricted to $J'=\{1\}\times\integers$ instead
of $J=\{1,\ldots,r\}\times\integers$.  In particular, the image of
$f$ does not include $f(r,n)=r-1+rn$ for any value of $n$, i.e., does
not include any number congruent to $r-1$ modulo $r$.  Since there are
infinitely many such values that are greater than $d_1$, this means that
the subset $I_1$ of Lemma~\ref{le_b1_comp} is infinite.
\end{proof}
The first claim in Corollary~\ref{co_important_first_Betti} proves that
$b^1(\cM_{k,r,K_r})=1$, where $K_r=(r-2,r-2)$ is the {\em canonical
divisor} of Baker-Norine.
The second claim in the corollary indicates that one cannot expect 
a simple duality theory based on $\cO_{k,r}$ or a {\em line bundle} such
as $\cL_{k,r,\ourv d}$ for any value of $\ourv d$.

\subsection{The Betti Numbers of Certain Skyscrapers}

\begin{definition}
Let $k$ be a field.  By the {\em skyscraper at $B_1$ with value $k$},
${\rm Sky}(B_1,k)$  we mean
the sheaf of $k$-vector spaces whose only nonzero value is at $B_1$, with
value $k$ (all restriction maps are therefore zero).  We similarly define
the same with $B_1$ replaced with $B_2$.
\end{definition}
In Section~\ref{se_simpleIII} we will explain this terminology.

\begin{lemma}\label{le_sky_Betti}
We have that for $i=1,2$,
$$
b^0 \bigl(  {\rm Sky}(B_i,k) \bigr) = 1, \quad
b^1 \bigl(  {\rm Sky}(B_i,k) \bigr) = 0.
$$
\end{lemma}
\begin{proof}
This is immediate from the definition of the cohomology groups
\eqref{eq_cohomology}, since the map $u$ there,
$$
u\from \bigoplus_{j=1}^3 \cF(B_j)
\to
\bigoplus_{i=1}^2 \cF(A_i)
$$
amounts to $u\from k \to 0$.
\end{proof}

\ignore{
----------------------------------------------------------

{\tiny\red
Result: for general $\cF(B_3)$,
$$
H^1(\cF) = {\rm span}_{i\in\integers}(x_1^i,x_2^i) / \sim,
$$
where
$$
x_1^i\sim 0 \quad \mbox{if $i\le d_1$},
$$
$$
x_2^i\sim 0 \quad \mbox{if $i\le d_2$},
$$
and for any $m\in\cF(B_3)$,
$$
\cF(\phi_{13})m \sim - \cF(\phi_{23}) m.
$$
This shows that $H^1(\cO)$ is infinite dimensional, and we really need
$r$ copies of $k[v,1/v]$ if we have this to be finite dimensional,
and for $\cF=\cM_K$, with $K=(r-2,r-2)$ there is a natural isomorphism
$$
H^1(\cM_K) \isom {\rm span}_k(x_1^{r-1},x_2^{r-1}) / (x_1^{r-1}+x_2^{r-1}=0).
$$
This seems to be trying to act like a residue map. (???)

In more detail, the preimage of $x_1^{r-1}$ in $k[y_1]$ doesn't exist, but
would be $1/y_1$ if $1/y_1$ were to exist in $k[y_1]$; similarly for 
$x_1,y_1$ replaced with $x_2,y_2$, and the ``glue'' in $\cM_{\ourv d}$
between these two is given by the fact that $e_r$ maps to 
$(x_1^{r-1},x_2^{r-1})$.  Maybe we will be able to explain this better
in the future...
}

}

\section{The Euler Characteristic and Riemann-Roch Theorem}
\label{se_rr}

Here we calculate the Euler characteristic of the sheaves $\cM_{k,r,\ourv d}$,
whereupon the Baker-Norine graph Riemann-Roch Theorem becomes
equivalent to the formula
$$
b_1(\cM_{k,r,\ourv d}) = b_0(\cM_{k,r,K_r-\ourv d}).
$$
Our Euler characteristic calculation is simplified by the use of short/long
exact sequences.

\subsection{Euler Characteristic Computations}

\begin{definition}
Let $\cF$ be a sheaf of $k$-vector spaces on a simple category $\cC$
such that $b_i(\cF)$ are finite for $i=0,1$.  We define
$$
\chi(\cF) \eqdef b_0(\cF) - b_1(\cF)
$$
\end{definition}

\begin{theorem}\label{th_Euler_char}
For an integer $r\ge 1$, field $k$, and $\ourv d\in\integers^2$, let
$\cM_{r;\ourv d}$ be as Definition~\ref{de_model_grrr_two}.
Then
\begin{equation}\label{eq_Euler_char}
\chi(\cM_{r;\ourv d}) = d_1 + d_2 - (r-2).
\end{equation}
\end{theorem}
\begin{proof}
For $\ourv d=(0,0)$, we have 
$$
\chi(\cM_{r,\ourv 0}) = b_0(\cM_{r,\ourv 0})-
b_1(\cM_{r,\ourv 0}) = 1 - (r-1) = 2-r,
$$
which verifies \eqref{eq_Euler_char} in this case.

We claim that if \eqref{eq_Euler_char} holds for some $\ourv d$, then 
it also holds for $\ourv d + (1,0)$.  To see this, let us describe
a morphism $\mu$ giving rise to
an exact sequence of sheaves of $k$-vector spaces
\begin{equation}\label{eq_basic_exact}
0 \to \cM_{\ourv d} \xrightarrow{\mu}
 \cM_{\ourv d+(1,0)} \to \Sky(B_1,k) \to 0  \ :
\end{equation}
we simply take $\mu$ to be the identity map on $A_1,A_2,B_2,B_3$,
and let $\mu(B_1)$ be the map $1\mapsto y_1$ (this choice is
forced when $\mu(A_1)$ is the identity map, given the restrictions of
$\cM_{\ourv d},\cM_{\ourv d+(1,0)}$ along $A_1\to B_1$.
Then $\mu$ is surjective everywhere except at $B_1$, where it is the
map $k[y_1]\to k[y_1]$ of $k[y_1]$ modules taking $1$ to $y_1$, whose
cokernel is
$$
k[y_1] / (y_1 \, k[y_1]) \isom k.
$$
Hence the cokernel of $\mu$ is just the sheaf supported on $B_1$ 
(therefore a skyscraper sheaf at $B_1$) whose
value there is $k$.

The exact sequence \eqref{eq_basic_exact} is depicted in 
Figure~\ref{fi_EulerShortExact}.
\begin{figure}[h]  \label{fi_EulerShortExact}\centering
\EulerShortExact
\caption{The maps $\rho_{i, d_i}$, $i = 1,2$,
  send $p(y_i)$ to $x_i^{d_i}p(x_i^{-1})$.
Sheaf morphisms are indicated in blue arrows (the interesting arrows)
and red arrows (which are idenentity or zero maps).}
\end{figure}

In view of Lemma~\ref{le_sky_Betti} and
the short exact sequence \eqref{eq_basic_exact}, the associated
long exact sequence 
shows that
if one of $\cM_{\ourv d},\cM_{\ourv d+(0,1)}$ has finite Betti numbers, 
then the other one does and
$$
\chi\bigl( \cM_{\ourv d+(1,0)}\bigr)  
= \chi\bigl( \cM_{\ourv d}\bigr)  + \chi\bigl( \Sky(B_1,k) \bigr) 
= \chi\bigl( \cM_{\ourv d}\bigr)  + 1 \ .
$$

It follows that if \eqref{eq_Euler_char} holds for one of $\ourv d$ or $\ourv d+(1,0)$, 
then it holds for the other one as well.
It follows similarly for $\ourv d$ and $\ourv d+(0,1)$.
Hence if it holds for one value of $\ourv d\in\integers^2$, then it holds
for all elements of $\integers^2$.
Since we have verified \eqref{eq_Euler_char} for $\ourv d=(0,0)$,
it holds for all $\ourv d$.
\end{proof}

\subsection{Consequences of the Euler Characteristic Formula}

Let us gather some consequences of Theorem~\ref{th_Euler_char},
using results from Sections~\ref{se_models_grr} and \ref{se_some_Betti}.
The first consequence is a weaker form of the Riemann-Roch theorem for
graphs.

\begin{theorem}\label{th_weak_rr}
Let $k$ be a field, and $r\ge 1$ an integer.  Then
\begin{enumerate}
\item
for all $\ourv d=(d_1,d_2)$ with $d_1+d_2<0$ we have
\begin{equation}\label{eq_ourvd_small}
b^0\bigl( \cM_{k,r,\ourv d}\bigr)=0, \quad
b^1\bigl( \cM_{k,r,\ourv d}\bigr)=2(r-2)-d_1-d_2 \ ;
\end{equation} 
and
\item 
for all $\ourv d=(d_1,d_2)$ with $d_1+d_2\ge 2(r-2)$ we have
\begin{equation}\label{eq_ourvd_large}
b^1\bigl( \cM_{k,r,\ourv d}\bigr)=0, \quad
b^0\bigl( \cM_{k,r,\ourv d}\bigr)=d_1+d_2 - 2(r-2) \ .
\end{equation}
\end{enumerate}
\end{theorem}
\begin{proof}
The case where $d_1+d_2<0$ follows from 
Theorem~\ref{th_main_grrr_model}
and the fact that 
If $d_1+d_2<0$, $\ourv d$ cannot be equivalent to an
effective divisor (since adding any multiple of $(r,-r)$ to $\ourv d$
does not change the sum of the first and second components of the vector).
Hence ${\rm GRRR}_r(\ourv d)=-1$, and by Theorem~\ref{th_main_grrr_model}
we conclude the first equality in \eqref{eq_ourvd_small}.
The second equality in \eqref{eq_ourvd_small} then follows from the
first equality and Theorem~\ref{eq_Euler_char}.

Similarly, by Theorem~\ref{th_GRRR_formula} parts~(2) and (3), 
if $0\le d_2\le r-1$, then ${\rm GRRR}_r(\ourv d)=d_1+d_2-r+1$
if $d_1\ge r-1$; and $d_1\ge r-1$ is certainly true if $d_1+d_2\ge 2r-2$,
since $d_2\le r-1$.
Hence ${\rm GRRR}_r(\ourv d)=d_1+d_2-r+1$ whenever $d_1+d_2\ge 2r-2$
and $0\le d_2\le r-1$; but any vector $\ourv d$ can be brought to
a vector with $0\le d_2\le r-1$ by adding the appropriate multiple
of $(r,-r)$, and this multiple does not change the sum of the first and
second components.

Hence whenever $d_1+d_2\ge 2r-2$ we have that 
${\rm GRRR}_r(\ourv d)=d_1+d_2-r+1$, and this implies the first equality
of \eqref{eq_ourvd_large}, by Theorem~\ref{th_main_grrr_model}.
The second equality in \eqref{eq_ourvd_large} then follows from the
first equality and Theorem~\ref{eq_Euler_char}.
\end{proof}

\section{Foundations, Part 3: $\cO$-modules and Ext groups}
\label{se_simpleIII}

%
%
%

\newcommand{\NamePlusDiag}[2]{  	
\begin{array}{|c|}
\hline  {#1} \\  \hline {#2} \\ \hline
\end{array} 
}

\newcommand{\CoSkyBOneEdits}{

\begin{tikzpicture}
\node at (1.5,2)(v11){\footnotesize$M$};
\node at (1.5,1)(v12){$0$};
\node at (1.5,0)(v13){$0$};
\node at (-.5,1.5)(v14){\footnotesize$M\otimes_{R_1}S_1$};
\node at (-.5,0.5)(v15){$0$};

\draw[->] (v11) -- (v14);
\draw[->] (v12) -- (v14);
\draw[->] (v12) -- (v15);
\draw[->] (v13) -- (v15);

\end{tikzpicture}

}

\newcommand{\CoSkyBThreeEdits}{

\begin{tikzpicture}
\node at (1.5,2)(v11){$0$};
\node at (1.5,1)(v12){\footnotesize$M$};
\node at (1.5,0)(v13){$0$};
\node at (-.5,1.5)(v14){\footnotesize$M\otimes_{R_1}S_1$};
\node at (-.5,0.5)(v15){\footnotesize$M\otimes_{R_2}S_2$};

\draw[->] (v11) -- (v14);
\draw[->] (v12) -- (v14);
\draw[->] (v12) -- (v15);
\draw[->] (v13) -- (v15);

\end{tikzpicture}

}

\newcommand{\CoSkyBTwoEdits}{

\begin{tikzpicture}
\node at (1.5,2)(v11){$0$};
\node at (1.5,1)(v12){$0$};
\node at (1.5,0)(v13){\footnotesize$M$};
\node at (-.5,1.5)(v14){$0$};
\node at (-.5,0.5)(v15){\footnotesize$M\otimes_{R_2}S_2$};

\draw[->] (v11) -- (v14);
\draw[->] (v12) -- (v14);
\draw[->] (v12) -- (v15);
\draw[->] (v13) -- (v15);

\end{tikzpicture}

}

\newcommand{\CoSkyAOneEdits}{

\begin{tikzpicture}
\node at (1.5,2)(v11){$0$};
\node at (1.5,1)(v12){$0$};
\node at (1.5,0)(v13){$0$};
\node at (-.5,1.5)(v14){\footnotesize$N$};
\node at (-.5,0.5)(v15){$0$};

\draw[->] (v11) -- (v14);
\draw[->] (v12) -- (v14);
\draw[->] (v12) -- (v15);
\draw[->] (v13) -- (v15);

\end{tikzpicture}

}

\newcommand{\CoSkyATwoEdits}{

\begin{tikzpicture}
\node at (1.5,2)(v11){$0$};
\node at (1.5,1)(v12){$0$};
\node at (1.5,0)(v13){$0$};
\node at (-.5,1.5)(v14){$0$};
\node at (-.5,0.5)(v15){\footnotesize$N$};

\draw[->] (v11) -- (v14);
\draw[->] (v12) -- (v14);
\draw[->] (v12) -- (v15);
\draw[->] (v13) -- (v15);

\end{tikzpicture}

}
\newcommand{\SkyBOneEdits}{

\begin{tikzpicture}
\node at (1.5,2)(v11){\footnotesize$M$};
\node at (1.5,1)(v12){$0$};
\node at (1.5,0)(v13){$0$};
\node at (-.5,1.5)(v14){$0$};
\node at (-.5,0.5)(v15){$0$};

\draw[->] (v11) -- (v14);
\draw[->] (v12) -- (v14);
\draw[->] (v12) -- (v15);
\draw[->] (v13) -- (v15);

\end{tikzpicture}

}

\newcommand{\SkyBThreeEdits}{

\begin{tikzpicture}
\node at (1.5,2)(v11){$0$};
\node at (1.5,1)(v12){\footnotesize$M$};
\node at (1.5,0)(v13){$0$};
\node at (-.5,1.5)(v14){$0$};
\node at (-.5,0.5)(v15){$0$};

\draw[->] (v11) -- (v14);
\draw[->] (v12) -- (v14);
\draw[->] (v12) -- (v15);
\draw[->] (v13) -- (v15);

\end{tikzpicture}

}

\newcommand{\SkyBTwoEdits}{

\begin{tikzpicture}
\node at (1.5,2)(v11){$0$};
\node at (1.5,1)(v12){$0$};
\node at (1.5,0)(v13){\footnotesize$M$};
\node at (-.5,1.5)(v14){$0$};
\node at (-.5,0.5)(v15){$0$};

\draw[->] (v11) -- (v14);
\draw[->] (v12) -- (v14);
\draw[->] (v12) -- (v15);
\draw[->] (v13) -- (v15);

\end{tikzpicture}

}

\newcommand{\SkyAOneEdits}{

\begin{tikzpicture}
\node at (1.5,2)(v11){\footnotesize$N$};
\node at (1.5,1)(v12){\footnotesize$N$};
\node at (1.5,0)(v13){$0$};
\node at (-.5,1.5)(v14){\footnotesize$N$};
\node at (-.5,0.5)(v15){$0$};

\draw[->] (v11) -- (v14);
\draw[->] (v12) -- (v14);
\draw[->] (v12) -- (v15);
\draw[->] (v13) -- (v15);

\end{tikzpicture}

}

\newcommand{\SkyATwoEdits}{

\begin{tikzpicture}
\node at (1.5,2)(v11){$0$};
\node at (1.5,1)(v12){\footnotesize$N$};
\node at (1.5,0)(v13){\footnotesize$N$};
\node at (-.5,1.5)(v14){$0$};
\node at (-.5,0.5)(v15){\footnotesize$N$};

\draw[->] (v11) -- (v14);
\draw[->] (v12) -- (v14);
\draw[->] (v12) -- (v15);
\draw[->] (v13) -- (v15);

\end{tikzpicture}

}

%
%
%
%
%
%
%

In this section we discuss the notion of a $\cO$-module for a sheaf of
rings, $\cO$, on a category, $\cC$.  
In specific computations we often assume that $\cC$ is a bipartite
category, or even just $\CTwoV$; however, it is often conceptually
simpler to work with a general category, and working as such makes
it easier to cite the literature.

\subsection{$\cO$-Modules}

Just after Definition~\ref{de_cO_r} we explained that the
sheaves $\cO_{k,r}$ can be viewed a {\em sheaves of $k$-algebras},
and that the $\cM_{k,r,\ourv d}$ are $\cO_{k,r}$-modules.  Let 
us make this precise.

\begin{definition}
Let $\cC$ be a category.  By a {\em sheaf of rings on $\cC$} we mean
a contravariant functor, $\cO$, from $\cC$ to the category of rings.
\end{definition}
In other words, a sheaf of rings $\cO$
is the data consisting of
\begin{enumerate}
\item a ring, $\cO(P)$, for each $P\in\Ob(\cC)$, and
\item a morphism of rings $\cO(\phi)\from \cO(Q)\to\cO(P)$ for
each morphism $\phi\from P\to Q$ in $\cC$
\end{enumerate}
such that $\phi\to \cO(\phi)$ respects composition
(i.e., $\cO(\phi_2\circ\phi_1)=\cO(\phi_1)\circ\cO(\phi_2)$) and
takes identity maps to identity maps.
Hence,
if $\cC$ is the bipartite category $(\cA,\cB,E)$, a sheaf of rings, $\cO$,
therefore consists of giving a ring, $\cO(P)$, for each $P\in \cA\amalg\cB$,
and a morphism $\cO(A,B)\from\cO(B)\to\cO(A)$ for each $(A,B)\in E$;
there are no requirements on the $\cO(A,B)$, since there are no
nontrivial compositions in $\cC$ (i.e., if $\phi_2\circ\phi_1$ is defined,
then at least one of $\phi_1,\phi_2$ is an identity morphism).

For example, a sheaf of $\cO$-modules on $\CTwoV$ is illustrated in
Figure~\ref{fi_generic_Omod}; this consists of arbitrary rings
$R_1,R_2,R_2,S_1,S_2$ and, for $i=1,2$, arbitrary morphisms of rings
$R_i\to S_i$ and $R_3\to S_i$.

\begin{figure}[h]
\begin{center}
\begin{tabular}{ |c |c|c|}
\hline
$\mathcal{C}$ & $\mathcal{O}$ & $\mathcal{M}$\\
\hline
\Category & \Struct & \SheafModules \\
\hline
\end{tabular}
\end{center}
\caption{A sheaf of rings $\cO$ and a sheaf of $\cO$-modules, $\cM$ on
$\CTwoV$.}
\label{fi_generic_Omod}
\end{figure}

\begin{definition}
Let $\cC$ and $\cD$ be categories.  By a {\em sheaf on $\cC$ with values
in $\cD$} we mean a contravariant functor $\cC\to\cD$.
\end{definition}
This definition not only generalizes our definition of sheaves of
$k$-vector spaces and of rings---making precise what we mean by a
sheaf of $k$-algebras---but will also be useful later in this section.

\begin{definition}
If $\cO$ is a sheaf of rings on $\cC$, then a {\em sheaf of $\cO$-modules}
is a sheaf $\cM$ of abelian groups on $\cC$ such that
\begin{enumerate}
\item for each $P\in\Ob(\cC)$, $\cM(P)$ is endowed with the structure
of an $\cO(P)$-module, and
\item the ring structure is respected under restriction, i.e., for
each morphism $\phi\from P\to Q$ of $\cC$ we have
\begin{equation}\label{eq_sheaf_of_mods}
\cM(\phi) (rm) = \bigl( \cO(\phi)r\bigr) \, \bigl( \cM(\phi)m\bigr)
\end{equation} 
for all $r\in\cO(Q)$ and $m\in \cM(Q)$.
\end{enumerate}
\end{definition}
Note that \eqref{eq_sheaf_of_mods} is the same as 
\eqref{eq_sheaf_of_mods_bi}.
In this article all modules are assumed to be unital, i.e., the unit in
the ring acts as the identity element on the module\footnote{In [SGA4]
this is not generally assumed; see, for example, [SGA4]
Proposition~II.6.7, which speaks of {\em modules unitaires};
therefore, in [SGA4], a $k$-module $k[\epsilon]/(\epsilon^2)$ could
have $1$ acting on it via
$p(\epsilon)\mapsto (1+\epsilon)p(\epsilon)$.
}.

\begin{example}
If $k$ is a field, then a $k$-module is the same thing as a $k$-vector
space.
If $\cC$ is any category, then the constant sheaf
$\underline k$ is a sheaf of rings, and a $\underline k$-module is
the same thing as a sheaf of $k$-vector spaces.
\end{example}

\begin{example}
If $\cC=\CTwoV$ and $k$ is a field, then $\cO_{k,r}$
(Definition~\ref{de_cO_r}) is a sheaf of rings and, moreover, a sheaf of
$k$-algebras, and for each $\ourv d\in\integers^2$,
$\cM_{k,r,\ourv d}$ is a sheaf of $\cO_{k,r}$-modules.
\end{example}

\begin{definition}
A {\em morphism} $u\from \cF\to\cG$ of sheaves on $\cC$ with values in
$\cD$ is a natural transformation of functors.
\end{definition}
In other words, a morphism $u\from\cF\to\cG$ consists of the data
$u(P)\from\cF(P)\to\cG(P)$ for each $P\in\Ob(\cC)$ such that
for any $\phi\from P\to Q$ in $\cC$ we have
$u(P)\cF(\phi)=\cG(\phi)u(Q)$.

\begin{definition}
Let $\cO$ be a sheaf of rings on $\cC$.
If $\cM,\cM'$ are two sheaves of $\cO$-modules, we use
$$
\Hom_\cO(\cM,\cM')
$$
to denote the set of morphisms $u\from\cM\to\cM'$ such that respect
the $\cO$-structure of $\cM,\cM'$, i.e., such that for each
$P\in\Ob(\cC)$ we have $u(P)(rm)=r\,u(P)m$ (this is an equation in
$\cM'(P)$) for all $r\in\cO(P)$ and $m\in\cM(P)$.
\end{definition}

\subsection{Examples of $\Hom_\cO$}

\begin{example}
Let $k$ be a field and $r\ge 1$ an integer, and
let $\cM$ be a sheaf of $\cO_{k,r}$-modules on $\CTwoV$.  Then
there are natural isomorphisms
$$
\Gamma(\cM)\isom\Hom_{\underline k}(\underline k,\cM)
\isom\Hom_{\cO_{k,r}}(\cO_{k,r},\cM).
$$
\end{example}

\begin{example}
Let $k$ be a field and $y$ an indeterminate.
Let $\cO=\Delta_0$ be the category with one object, $0$, and one morphism,
i.e.,
the bipartite category $(\{0\},\emptyset,\emptyset)$.  Then
a sheaf of rings (of $k$-vectors spaces, etc.)
on $\Delta_0$ is just a ring (a $k$-vector spaces, etc.),
i.e., its value on $0\in\Ob(\Delta_0)$.
Consider a
$$
\phi\in \Hom_{k[y]}\bigl( k[y],k[y] \bigr) ,
$$
where $\Hom_{k[y]}$ means morphisms of $k[y]$-modules; then $\phi$
is determined by the image of $1\in k[y]$, i.e., by $\phi(1)\in k[y]$.
It follows that
$$
\Hom_{k[y]}\bigl( k[y],k[y] \bigr) \isom k[y].
$$
There is a natural inclusion
\begin{equation}\label{eq_k_alg_k_vs}
\Hom_{k[y]}\bigl( k[y],k[y] \bigr) \to
\Hom_k\bigl( k[y],k[y] \bigr),
\end{equation}
where $\Hom_k$ means morphisms of $k$-algebras, i.e., of $k$ vector spaces.
However a linear transformation $\phi\from \Hom_k(k[y],k[y])$ is
determined by $\phi(1),\phi(y),\phi(y^2),\ldots$, each of which has
no forced dependence on the others; hence
$$
\Hom_k\bigl( k[y],k[y] \bigr) \isom \bigl( k[y] \bigr)^\integers
$$
Hence the inclusion \eqref{eq_k_alg_k_vs} is strict.
Notice that this inclusion is an inclusion of sets, but can be
viewed as an inclusion in the category of $k$-vector spaces or
even left or right $k[y]$-algebras (where $k[y]$ acts on 
$\phi\in\Hom_k(k[y],k[y])$ by post- or pre-multiplication).
\end{example}

More generally, for a category $\cC$ and a sheaf of $k$-algebras,
$\cO$, there is an inclusion
$$
\Hom_{\cO}(\cM,\cM') \to \Hom_{\underline k}(\cM,\cM')
$$
which is often a strict inclusion.

\begin{example}
Let $k$ be a field and $r\ge 1$ an integer.  We shall see that
$$
\Hom_{\cO_{r,k}}(\cM_{k,r,\ourv d},\cM_{k,r,\ourv d'})
$$
is a finite dimensional $k$-vector space for any 
$\ourv d,\ourv d'\in\integers^2$, whereas
$$
\Hom_k(\cM_{k,r,\ourv d},\cM_{k,r,\ourv d'})
$$
for $\ourv d=\ourv d'=\ourv 0$
(and many other pairs $\ourv d,\ourv d'$)
is infinite dimensional.
\end{example}

\subsection{(Co)homology and Ext groups}

Below we will see that
for any sheaf of rings $\cO$ on $\CTwoV$, the category of
$\cO$-modules is an abelian category with enough projectives (and
injectives).

\begin{definition}
If $\cO$ is a sheaf of rings on a category $\cC$, and $\cF,\cG$
are two $\cO$-modules, we use $\Ext^i_\cO(\cF,\cG)$ to denote the 
derived bifunctors of
$$
(\cF,\cG) \mapsto \Hom_\cO(\cF,\cG).
$$
\end{definition}
Recall that, in practice, this means that we can compute 
$\Ext^i_\cO(\cF,\cG)$ from a projective resolution of $\cF$ or an
injective resolution of $\cG$, or from the double complex involving
both these resolutions (and the resulting Ext groups are
isomorphic via a unique isomorphism that can be constructed from
two different projective resolutions of $\cF$ and/or two injective
resolutions of $\cG$).
For a reference, see
\cite{weibel,hilton_stammbach} in the case of $R$-modules for a ring $R$;
these facts are valid for arbitrary abelian category, for
reasons discussed in \cite{hartshorne}, page~203,
and in the discussion regarding Theorem~1.6.1 of \cite{weibel}, i.e., 
regarding the Freyd-Mitchell Embedding Theorem 
(or \cite{freyd} or the original \cite{freyd_original}).

\subsection{The Yoneda Pairing}

Here is standard lemma regarding Ext groups that we will need,
usually called the {\em Yoneda product} or {\em Yoneda pairing}
for Ext groups,
described in \cite{hilton_stammbach}, Exercise~IV.9.3. 
Let us summarize the result that we need.

\begin{lemma}[Yoneda Product]\label{le_Yoneda_prod}
Let $\cF_1,\cF_2,\cF_3$ be any three elements of an abelian category with
either (1) enough projectives, or (2) enough injectives.
Then for any $i,j$ there is a
pairing
\begin{equation}\label{eq_Yoneda_pairing}
\Ext^i(\cF_1,\cF_2) \times \Ext^j(\cF_2,\cF_3) \to
\Ext^{i+j}(\cF_1,\cF_3) 
\end{equation} 
such that
\begin{enumerate}
\item
the paring is natural (i.e., functorial) in $\cF_1,\cF_2,\cF_3$, 
where naturality in $\cF_2$ means that
the morphism of $\cF_2$-covariant functors
\begin{equation}\label{eq_Yoneda_pairing_alt}
\Ext^i(\cF_1,\cF_2) \to
\Hom_\integers \bigl( \Ext^j(\cF_2,\cF_3), \Ext^{i+j}(\cF_1,\cF_3) \bigr)
\end{equation} 
is natural in $\cF_2$;
\item
if $j=0$ and $\cF_2=\cF_3$, then the pairing \eqref{eq_Yoneda_pairing}
acts via the functoriality of $\Ext^i(\cF_1,\cF_2)$ in the variable $\cF_2$
(and similarly for $i=0$ and $\cF_1=\cF_2$);
in particular,
the image $\id_{\cF_1}$ in \eqref{eq_Yoneda_pairing_alt} is the 
identity morphism
of $\Ext^j(\cF_1,\cF_3)$; and
\item
if $0\to\cG_1\to\cG_2\to\cG_3\to 0$ is any short exact sequence, then
the resulting
long exact sequence 
$$
\cdots\to \Ext^i(\cF_1,\cG_1)\to\Ext^i(\cF_1,\cG_2)\to\Ext^i(\cF_1,\cG_3)\to
\Ext^{i+1}(\cF_1,\cG_1)\to\cdots
$$
admits a chain map to the long exact sequence resulting from the
right-hand-side of \eqref{eq_Yoneda_pairing_alt}, which takes
$\Ext^i(\cF_1,\cG_j)$ to
$$
\Hom_\integers \bigl( \Ext^{k-i}(\cG_j,\cF_3), \Ext^k(\cF_1,\cF_3) \bigr)
$$
\end{enumerate}
\end{lemma}
We remark that 
in this article, we only need the case $i+j=1$, whereby one of
$i,j$ must be zero.  In this case one can verify that the pairing
of an $\alpha\in\Ext^1(\cF_1,\cF_2)$ with $\beta\in\Ext^0(\cF_2,\cF_3)$
is just the map $\Ext^1(\cF_1,\cF_2)\to\Ext^1(\cF_1,\cF_3)$ induced by
$\beta$ viewed as an element of $\Hom(\cF_2,\cF_3)$ by the fact
that $\Ext^1(\cF_1,\cF_2)$ is natural (i.e., functorial) in $\cF_2$.
A similar discussion holds for the pairing of $\Ext^0(\cF_1,\cF_2)$
with $\Ext^1(\cF_2,\cF_3)$.

\subsection{Skyscrapers and Coskyscrapers}

In this subsection we explain the construction of {\em skyscraper}
and {\em coskyscraper} sheaves.  We give these constructions for
more general categories than $\CTwoV$.
The notion of a skyscraper sheaf and their use to construct injective
resolutions is standard (see \cite{hartshorne}, Proposition~III.2.2).
The notion of a coskyscraper sheaf and their use to construct projective
resolutions is less standard---since this construction doesn't work
for (infinite) topological spaces.
However, for sheaves of $k$-vector spaces, or sheaves of $\cO$-modules
for any constant sheaf, $\cO$, the construction is immediate
from [SGA4] Proposition~I.5.1, as we will explain at the end of this
subsection; the adaptation to the more general situation is akin to
the standard definition of the pullback $f^*$ of modules where
$f$ is a morphism of ringed spaces (see \cite{hartshorne}, top of
page~110).

\begin{definition}\label{de_top_semitop}
We say that a category is {\em topological} all its Hom-sets are of size
one or zero.
We say that a category, $\cC$, is {\em semi-topological} if for for
$P\in\Ob(\cC)$, $\Hom(P,P)$ consists of only the identity map.
If $\cC$ is any category and $x,y\in\Ob(\cC)$, 
we use $x\le y$ to mean that $\Hom(x,y)$ is nonempty.
\end{definition}
The following properties are easy to verify:
if $\cC$ is semi-topological, then $x\le y$ and $y\le x$ implies that
$x$ and $y$ are isomorphic; if $\cC$ is a bipartite category, then
$x,y$ are isomorphic iff $x=y$;
if $\cC$ is finite and topological, then the downsets of $\cC$ are the
open subsets of a
topological space whose points are isomorphism classes of elements of $\cC$.
Semitopological categories have many properties in common with topological
spaces (see \cite{friedman_cohomology}).

\begin{definition}\label{de_top_sky}
Let $\cO$ be a sheaf of rings on a topological category, $\cC$.
For $P\in\Ob(\cC)$ and $M$ an $\cO(P)$-module we define
the {\em skyscraper (sheaf) at $P$ with value $M$} to be the sheaf
$\Sky=\Sky(P,M)$ 
\begin{enumerate}
\item whose values are
$$
\Sky(Q) = \left\{ \begin{array}{ll}    M & \mbox{$P\le Q$,} \\
0 & \mbox{otherwise;}  \end{array} \right.
$$
\item whose restriction maps are the identity map for all $Q_1,Q_2\ge P$;
\item where the $\cO(Q)$-module structure on $\cS(Q)$ is given by the
restriction map $\cO(Q)\to\cO(P)$ and the $\cO(P)$-module structure of $M$.
\end{enumerate}
\end{definition}

\begin{lemma}\label{le_top_sky} 
Let $\cO$ be a sheaf of rings on a topological
category $\cC$, and let $M$ be an $\cO(P)$-module.  Then for any $\cF$
there is a natural isomorphism
$$
\Hom_{\cO(P)}\bigl( \cF(P), M \bigr)
\to
\Hom_{\cO}\bigl(\cF, \Sky(P,M) \bigr) \ .
$$
If $M$ is an injective $\cO(P)$-module, then $\Sky(P,M)$ is an injective
$\cO$-module.
\end{lemma}
The proof is straightfoward; the reader is free to skip this general
proof and just verify this for $\CTwoV$ as indicated in the next section.

Similarly we have the construction of coskyscrapers.

\begin{definition}\label{de_top_cosky}
Let $\cO$ be a sheaf of rings on a topological category, $\cC$.
For $P\in\Ob(\cC)$ and $M$ an $\cO(P)$-module we define
the {\em coskyscraper (sheaf) at $P$ with value $M$} to be the sheaf
$\Sky=\Sky(P,M)$ 
\begin{enumerate}
\item whose values are
$$
\CoSky(Q) = \left\{ \begin{array}{ll}    
M\otimes_{\cO(P)}\cO(Q) & \mbox{if $Q\le P$,} \\
0 & \mbox{otherwise;} \end{array}  \right.
$$
\item whose restriction maps are induced by the map $m\otimes 1\mapsto
m\otimes 1$ for all $Q_1,Q_2\ge P$;
\item where the $\cO(Q)$-module structure on $\cS(Q)$ is given by acting
on the second component in the tensor product.
\end{enumerate}
\end{definition}

\begin{lemma}\label{le_top_cosky}
Let $\cO$ be a sheaf of rings on a topological
category $\cC$, and let $M$ be an $\cO(P)$-module.  Then for any $\cF$
there is a natural isomorphism
$$
\Hom_{\cO(P)}\bigl( M,\cF(P) \bigr)
\to
\Hom_{\cO}\bigl( \CoSky(P,M) , \cF \bigr) \ .
$$
If $M$ is a projective $\cO(P)$-module, then $\CoSky(P,M)$ is a projective
$\cO$-module.
\end{lemma}

It is a standard fact that for any ring, $A$, the category of $A$-modules
has enough injectives and projectives (see \cite{weibel}).

\begin{lemma}
Let $\cO$ be a sheaf of ring on any topological category $\cC$.
Then the category of $\cO$-modules has enough injectives and projectives.
\end{lemma}
\begin{proof}
Let $\cM$ be an $\cO$-module, and for each $x\in\Ob(\cC)$ let
$P(x)$ be a projective $\cO(x)$-module such that there is a surjection
$P(x)\to \cM(x)$.  Then the natural map of
$$
\bigoplus_{x\in\Ob(\cC)} \CoSky\bigl(x,P(x)\bigr)
$$
to $\cM$ is a surjection, and hence there are enough projectives.  
We argue similarly that there are enough injectives.
\end{proof}

All of the above constructions work for semitopological categories.
For example, $\Sky(P,M)$ is, more generally, constructed as having values
$$
Q\mapsto M^{\Hom(P,Q)}
$$
and there is a natural restriction map; similarly for $\CoSky(P,M)$.
For the case where $\cO$ is a constant sheaf, this follows from the following
discussion (the more general case is the above simple adaptation).

We remark that our construction of skyscrapers and coskyscrapers
follows from the very general Proposition~I.5.1 of [SGA4]; let us
summarize the main points.
If $u\from\cC\to\cC'$ is any functor, then if $\cF'$ is a sheaf
on $\cC'$ with values in any category, $\cD$
(i.e., a {\em presheaf} in the sense of [SGA4]), then one sets $u^*\cF'$
to be the sheaf on $\cC$ with values in $\cD$ given by $\cF=\cF'\circ u$
(so $\cF(P)=\cF'(u(P))$ for $P\in\Ob(\cC))$.  The functor $u^*$ has a
left adjoint $u_!$ and a right adjoint $u_*$ described in
[SGA4], Proposition~I.5.1, assuming that certain limits in $\cD$ are
representable (i.e., exist).
Consider the special case of the above, where $\cC=\Delta_0$ (the
terminal category, whose object is $0$ and has only the identity morphism at
$0$), $x\in\Ob(\cC')$,
and $u_x\from\Delta_0\to\cC'$ is the morphism taking $0$ to $x$.
Taking $\cD$ to be the category of $k$-vector spaces for a field $k$
yields the skycrapers and coskyscrapers, as $(u_x)_* M$ and $(u_x)_!M$,
respectively in the case of sheaves with values in $\cD$, i.e., sheaves
of $\underline k$-modules.


\section{Foundations, Part 4: $\cO$-modules and Ext groups for $\CTwoV$}
\label{se_simpleIV}

In this section we state the practical consequences of 
Section~\ref{se_simpleIII} regarding a projective resolution for
$\cO_{k,r}$
and an injective resolution for
$\omega_{k,r}=\cM_{k,r,\ourv d}$.
These method are much more general, and immediately yield
projective resolutions for any
sheaf of the form $\cL_{k,r,\ourv d}$ and injective resolutions
for any sheaf of the form $\cM_{k,r,\ourv d}$.

\subsection{The Standard Setup}

\begin{definition}\label{de_standard_setup}
By a {\em standard setup} on $\CTwoV$ we mean a sheaf of rings
$\cO$ on $\CTwoV$ all of whose restriction maps are injections.
We use the notation $S_i=\cO(A_i)$ for $i=1,2$ and $R_j=\cO(B_j)$
for $j=1,2,3$; similarly, we often
discuss an $\cO$-module $\cM$ with the notation $N_i=\cO(A_i)$
and $M_j=\cM(B_j)$.
\end{definition}
We depict the standard setup as follows:
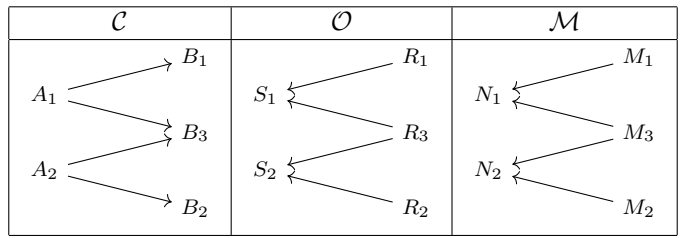
\begin{center}
\begin{tabular}{ |c |c|c|}
\hline
$\mathcal{C}$ & $\mathcal{O}$ & $\mathcal{M}$\\
\hline
\Category & \Struct & \SheafModules \\ 
\hline
\end{tabular}
\end{center}

The skyscraper sheaves of Definition~\ref{de_top_sky} can be 
depicted by the following diagrams:

%

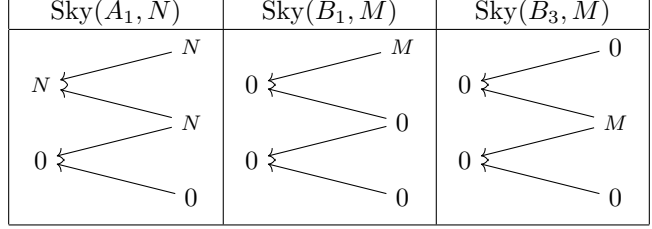
\begin{center}
\begin{tabular}{ |c |c|c|}
\hline
$\Sky(A_1,N)$ & $\Sky(B_1,M)$ & $\Sky(B_3,M)$\\
\hline
\SkyAOneEdits &\SkyBOneEdits & \SkyBThreeEdits \\ 
\hline
\end{tabular}
\end{center}
where $N$ is an $S_1$-module, and $M$ is an $R_i$-module in
$\Sky(B_i,M)$; we have similar diagrams for $\Sky(A_2,N)$ and
$\Sky(A_1,N)$.
In $\Sky(A_1,N)$, the values at $B_i$ is $N$ for $i=1,3$, and
for these $i$ the set $N$ is an $R_i$-module via the restriction
maps $R_i\to S_1$ and the $S_1$-module structure on $N$.
The importance of the skyscraper sheaves are due to the the following
two facts:
\begin{enumerate}
\item there is a canonical isomorphism:
\begin{equation}\label{eq_sky_adjointness}
\Hom_\cO(\cF,\Sky(x,M) ) \to \Hom_{\cO(x)}\bigl( \cF(x),M \bigr)
\end{equation} 
for any $x\in\Ob(\CTwoV)$ and $\cO(x)$-module $M$, and therefore
\item $\Sky(x,I)$ is an injective $\cO$-module for any injective
$\cO(x)$-module $I$.
\end{enumerate}
The second fact is an easy consequence of \eqref{eq_sky_adjointness}.
To prove \eqref{eq_sky_adjointness}, the reader may either accept
Lemma~\ref{le_top_sky}, or verfiy this lemma, or simply verify this
lemma in the case of $\CTwoV$.

For example, to check
\eqref{eq_sky_adjointness} for $x=A_1$, one checks that
any morphism represented by the
solid arrow in Figure~\ref{fi_sky_adjointness} 
extends uniquely to the dashed arrows there to produce a morphism of 
$\cO$-modules.
\begin{figure}[h]
\SkyHomTwo
\caption{The adjointness for the skyscraper $\Sky(B_1,M)$}
\label{fi_sky_adjointness}
\end{figure}
The verification for $x=A_2$ is---by symmetry---the same, and the
verfication for $x=B_1,B_2,B_3$ is immediate.

Analogously we have coskyscraper sheaves of Definition~\ref{de_top_cosky}
depicted for $x=A_1,B_1,B_3$ as follows:
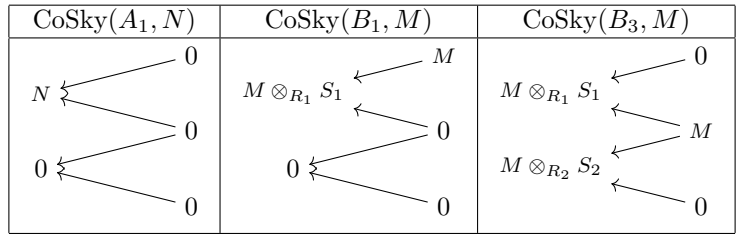
\begin{center}
\begin{tabular}{ |c |c|c|}
\hline
$\CoSky(A_1,N)$ & $\CoSky(B_1,M)$ & $\CoSky(B_3,M)$\\
\hline
\CoSkyAOneEdits &\CoSkyBOneEdits & \CoSkyBThreeEdits \\ 
\hline
\end{tabular}
\end{center}
and the satisfy:
\begin{enumerate}
\item there is a canonical isomorphism:
\begin{equation}\label{eq_cosky_adjointness}
\Hom_\cO(\CoSky(x,M),\cF) \to \Hom_{\cO(x)}\bigl( M,\cF(x) \bigr)
\end{equation}
for any $x\in\Ob(\CTwoV)$ and $\cO(x)$-module $M$, and therefore
\item $\CoSky(x,P)$ is a projective $\cO$-module for any projective
$\cO(x)$-module $P$.
\end{enumerate}

Again, to directly verify
\eqref{eq_cosky_adjointness} for $x=B_3$, one checks that
any morphism represented by the
solid arrow in Figure~\ref{fi_cosky_adjointness} 
extends uniquely to the dashed arrows there to give a morphism of
$\cO$-modules.
\begin{figure}[h]
\CoSkyHom
\caption{The adjointness for the coskyscraper $\CoSky(B_3,M)$}
\label{fi_cosky_adjointness}
\end{figure}
The verification of \eqref{eq_cosky_adjointness}
for $x=B_1,B_2$ is analogous, and the verification
for $x=A_1,A_2$ is easier.

\subsection{Projective Resoltions of $\cO$ and Beyond}

One consequence of the above is that if the restriction maps of
$\cO$ are injections, then $S_i\otimes_{R_j}R_j\isom S_i$ for all
morphisms $(A_i,B_j)$ in $\CTwoV$.  
This leads to a convenient projective resolution of $\cO$ and
similar sheaves.

\begin{lemma}\label{le_pro_O}
Let $\cO$ be as in Definition~\ref{de_standard_setup} and assume that
all restriction maps in $\cO$ are injections.
Then the $\cO$-module $\cO$ has a projective resolution (in the
category of $\cO$-modules) given by:
\begin{equation}\label{eq_proj_res_O}
0 \to \cP_1\to\cP_0\to \cO \to 0,
\end{equation} 
where
$$
\cP_1 = \bigoplus_{i=1}^2 \CoSky(A_i,\tilde S_i), \quad
\cP_0 = \bigoplus_{j=1}^3 \CoSky(B_j,R_j)
$$
where $\tilde S_i$ is the $S_i$-submodule of $S_i\oplus S_i$
generated by $(1,-1)$.
\end{lemma}
\begin{proof}
First note that for any $i,j$ such that there is a morphism $A_i\to B_j$,
we have that $S_i\otimes_{R_j}R_j=S_i$ since the $\cO$ restriction
map $R_j\to S_i$ is an injection.
It follows that $\cP_0\to\cO$ is a surjection.
Since $R_j$ is a free $R_j$-module, it follows that $\cP_0$ is projective.
Next we verify (value-by-value)
that the kernel of $\cP_0\to\cO$ is given by the sheaf in the diagram below:

\begin{tikzpicture} [dot/.style={circle,inner sep=1.5pt,fill}]

\node at (1.5,2)(v11){$0$};
\node at (1.5,1)(v12){$0$};
\node at (1.5,0)(v13){$0$};
\node at (0,1.5)(v14){$\tilde{S_1}$};
\node at (0,0.5)(v15){$\tilde{S_2}$};

\node at (6,2)(v21){$R_1$};
\node at (6,1)(v22){$R_3$};
\node at (6,0)(v23){$R_2$};
\node at (4,1.5)(v24){$S_1 \oplus S_1$};
\node at (4,0.5)(v25){$S_2 \oplus S_2$};

\node at (10.5,2)(v31){$R_1$};
\node at (10.5,1)(v32){$R_3$};
\node at (10.5,0)(v33){$R_2$};
\node at (9,1.5)(v34){$S_1$};
\node at (9,0.5)(v35){$S_2$};

\draw[->,line width=.3mm] (v11) -- (v14);
\draw[->,line width=.3mm] (v12) -- (v15);
\draw[->,line width=.3mm] (v12) -- (v14);
\draw[->,line width=.3mm] (v13) -- (v15);

\draw[->,line width=.3mm] (v21) -- (v24);
\draw[->,line width=.3mm] (v22) -- (v25);
\draw[->,line width=.3mm] (v22) -- (v24);
\draw[->,line width=.3mm] (v23) -- (v25);

\draw[->,line width=.3mm] (v31) -- (v34);
\draw[->,line width=.3mm] (v32) -- (v35);
\draw[->,line width=.3mm] (v32) -- (v34);
\draw[->,line width=.3mm] (v33) -- (v35);

\draw [->,red] (v11) to [bend right=-5](v21);
\draw [->,red] (v12) to [bend right=-3]  (v22);
\draw [->,red] (v13) to [bend left=-5](v23);
\draw [->,red] (v14) to [bend right=-5] (v24);
\draw [->,red] (v15) to [bend right=-5](v25);

\draw [->,red] (v21) to [bend right=-5](v31);
\draw [->,red] (v22) to [bend right=-3]  (v32);
\draw [->,red] (v23) to [bend left=-5](v33);
\draw [->,red] (v24) to [bend right=-5] (v34);
\draw [->,red] (v25) to [bend right=-5](v35);


\end{tikzpicture}

Now we see that
(1) $\tilde S_i$ is a free 
$S_i$-module, so that the kernel $\cP_1$ is a projective sheaf that
is the sum of coskyscrapers as indicted above.
\end{proof}

A similar projective resolution can be made for any sheaf, $\cL$,
such that its values are those of $\cO$ and its restriction maps
$\cO(A_i,B_j)\from R_j\to S_i$ map $1$ to a unit in $S_i$.
These sheaves are what we call {\em line bundles} or {\em invertible
sheaves} in Section~\ref{se_simpleV}, and they include the
$\cO_{k,r}$-modules $\cL_{k,r,\ourv d}$ for any $k,r,\ourv d$.

As a consequene we may compute the cohomology groups 
$H^i(\cM)=\Ext^i(\cO,\cM)$ by taking the above projective resolution
of $\cO$, which yields the following result.

\begin{corollary}
Let $\cO$ be any sheaf of $k$-algebras
on $\CTwoV$ with injective restriction
maps.  Then, for any $\cO$-module, $\cM$,
with notation as in Definition~\ref{de_standard_setup}, we have
that $H^i(\cM)=\Ext^i(\cO,\cM)$ for all $i$ are the same groups as those
defined in 
Definition~\ref{de_cohomology}.
\end{corollary}
\begin{proof}
By \eqref{eq_cosky_adjointness}, we see that applying $\Hom(\,\cdot\,,\cM)$
to \eqref{eq_proj_res_O} we get that the groups $\Ext^i(\cO,\cM)$
for $i=0,1$ are, respectively, the kernel and cokernel of the map
$$
\bigoplus_{j=1}^3 \Hom_{R_j}(R_j,M_j) \to 
\bigoplus_{i=1}^2 \Hom_{S_i}(\tilde S_i,N_i) \ .
$$
Since $\tilde S_i=(1,-1)S_i$, we may replace $\tilde S_i$ with $S_i$ in the
above map and view it as the map
$$
\bigoplus_{j=1}^3 \Hom_{R_j}(R_j,M_j) \to 
\bigoplus_{i=1}^2 \Hom_{S_i}(S_i,N_i) 
$$
where we negate the natural maps for $j=3$ and $i=1,2$.  Lastly, for
any $R$-module $M$ over a ring $R$, we may identify
$\Hom_R(R,M)$ with $M$ via $\phi\in \Hom_R(R,M)$ maps to $\phi 1\in M$.
Hence the groups $\Ext^i(\cO,\cM)$ can be computed as the kernel and cokernel
of 
$$
\bigoplus_{j=1}^3 M_j \to 
\bigoplus_{i=1}^2 N_i
$$
with the above sign conventions; as such, these are the same groups as the
$H^i(\cM)$ in Definition~\ref{de_cohomology}.
\end{proof}

\subsection{Injective Resolutions of $\cM_{k,r,\ourv d}$}

We can similarly use skyscraper sheaves to build injective resolutions.
In the case of $\cM_{k,r,\ourv d}$ (or $\cO_{k,r}$), we can use
the injective resolutions of $k[x_1,1/x_1]$-modules
$$
k[x_1,1/x_1]\to k(x_1)\to k(x_1)/k[x_1,1/x_1]
$$
and similar resolutions to build injective resolutions of $\cM_{k,r,\ourv d}$.

In this article we prefer to use projective resolutions for all of our
specific computations.

\section{Computation of $\Hom(\cM_{\ourv d},\cM_{\ourv d'})$}
\label{se_global_hom}

\begin{remark}
{\red
A previous version of the articles erroneously claimed that $\Hom(\cM_{\ourv d},\cM_{\ourv d'}) \isom \Gamma(\cM_{k,r,\ourv d'-\ourv d})$. This is in general not the case.
}
\end{remark}

For the theorem will need the following result.

\begin{theorem}\label{th_hom_md}
Let $k$ be a field, $r\ge 1$ be an integer, and
$\ourv d,\ourv d'\in\integers^2$.
Then there is a canonical morphism
{\red (not an isomorphism)}
\begin{equation}\label{eq_hom_md}
\Hom\bigl(\cM_{k,r,\ourv d},\cM_{k,r,\ourv d'}\bigr) \to 
\Gamma(\cM_{k,r,\ourv d'-\ourv d}).
\end{equation} 
\end{theorem}
In Section~\ref{se_simpleV} we shall define {\em sheaf Hom}; it will
be clear that the proof we now give of the above theorem is based on
{\em local} considerations, and more generally gives a morphism
$$
\SHom\bigl(\cM_{\ourv d},\cM_{\ourv d'}\bigr) \rightarrow \cM_{\ourv d'-\ourv d} .
$$
In Section~\ref{se_simpleV} we shall
also discuss the line bundles $\cL_{k,r,\ourv d}$ defined in
Definition~\ref{de_line_L}; the formulas
$$
\cL_{k,r,\ourv d}\otimes\cM_{k,r,\ourv d'} \isom
\cM_{k,r,\ourv d + \ourv d'}
$$
to conceptually 
simplify matters a bit; indeed, such formulas imply that we need only
consider the case $\ourv d=\ourv 0$
in Theorem~\ref{th_hom_md}.

{\red The morphism \eqref{eq_hom_md} can also be obtained from
the morphism $\cL_{k,r,\ourv d}\to\cM_{k,r,\ourv d}$; the
above theorem is much easier since we are only claiming
the existence of a morphism.}

\begin{proof}
Let $S_1=\cO(A_1)=k[x_1,1/x_1]$, and consider the map
taking 
$$
\phi\in \Hom_{\cO_{k,r}}\bigl(\cM_{k,r,\ourv d},\cM_{k,r,\ourv d'}\bigr)
$$
to 
$$
\phi(A_1)\in
\Hom_{\cO(A_1)}\bigl(\cM_{k,r,\ourv d}(A_1),\cM_{k,r,\ourv d'}(A_1)\bigr)
=
\Hom_{S_1}(S_1,S_1).
$$
Identifying $\Hom_{S_1}(S_1,S_1)$ with $S_1$ in the usual fashion, we get
a canonical morphism
$$
u\from 
\Hom_{\cO_{k,r}}\bigl(\cM_{k,r,\ourv d},\cM_{k,r,\ourv d'}\bigr) \to S_1.
$$
We similarly define a map
$$
v\from 
\Hom_{\cO_{k,r}}\bigl(\cO_{k,r,\ourv 0},\cM_{k,r,\ourv d'-\ourv d}\bigr) 
\to S_1
$$
(since $\cO_{k,r}(A_1)=S_1$).
Clearly $u,v$ are $k$-linear maps.
To prove the theorem it suffices to show that 
\begin{enumerate}
\item 
$u$ and $v$ are injections, and
\item 
the image of $u$ lies in $v$
({\red it was previously claimed that the image of $u$ was equal to
that of $v$}).
\end{enumerate}

To show that $u$ is an injection, say that $u\phi=0$, i.e.,
$\phi(A_1)=0$
Since the restriction maps
$$
\cM_{k,r,\ourv d}(A_1,B_3), \quad \cM_{k,r,\ourv d'}(A_1,B_3)
$$
are isomorphisms, it follows $\phi(B_3)=0$.  Since the restriction maps
$$
\cM_{k,r,\ourv d}(A_2,B_3), \quad \cM_{k,r,\ourv d'}(A_2,B_3)
$$
are isomorphisms, it follows $\phi(A_2)=0$.
Since the restriction maps
$$
\cM_{k,r,\ourv d}(A_1,B_1), \quad \cM_{k,r,\ourv d'}(A_1,B_1)
$$
are injections, it follows that $\phi(B_1)=0$; similarly $\phi(B_2)=0$.
Hence $\phi$ is zero at $A_1,A_2,B_1,B_2,B_3$.

To show that $v$ is an injection, say that $v\psi=0$.  Since
$\cM_{k,r,\ourv d'-\ourv d}(A_1,B_3)$ is an isomorphism and
$\cO_{k,r}(A_1,B_3)$ is an injection, it follows that $\psi(B_3)=0$.
It follows that $\psi(A_2)$ maps
the image of $1\in\cO(A_2,B_3)$ to zero; 
since $\psi$ is a map of $\cO_{r,K}$ modules,
$\psi(A_2)$ is a map of $\cO_{r,K}(A_2)$-modules taking $1$ to zero,
and hence $\psi(A_2)=0$.
Similar to the last paragraph, we then have $\psi$ is zero at $B_1$
and $B_2$, and hence $\psi=0$.

Now consider the image of $u$.
Say that $u\phi=\gamma\in S_1$; then we may uniquely write
$\gamma$ as
\begin{equation}\label{eq_phi_at_A_1}
\phi(A_1)=\gamma=\sum_{i=1}^r p_i\bigl(x_1^r\bigr) x_1^i.
\end{equation} 
This uniquely determines $\phi(B_3)$ and $\phi(A_2)$ to be
\begin{equation}\label{eq_phi_at_A_2}
\phi(A_2)=\gamma'=\sum_{i=1}^r p_i\bigl( x_2^{-r} \bigr) x_2^i
\end{equation} 
{\red assuming they exist; the problem is that there are
conditions on the $p_i$ above for $\phi(B_3)$ to exist as a morphism
that intertwines the
$\cO(B_3)=k[v,1/v]$ action on $\cM_{k,r,\ourv d}(B_3)$
with the
$\cO(A_2)=k[x_2,1/x_2]$ action on the $A_2$ values; it is only
for $r=1$ that there are no extra conditions.
}
We see that the condition that
$\phi(A_1)$ extends to $\phi(B_1)$ is that
that $\gamma$ have no terms in $x_1$ of degree $d_1'-d_1'$;
similarly for $\gamma'$, $x_2$, and $d_2'-d_2$.

Now consider the image of $v$.
If $v\psi=\psi(A_1)=\gamma_1$, then $\gamma_1$ determines
the image of $1\in\cO_{k,r}(A_2)$ in 
$\cM_{k,r,\ourv d'-\ourv d}$ under $\psi(A_2)$, which gives
forces the relations \eqref{eq_phi_at_A_1} and \eqref{eq_phi_at_A_2}
with $\psi$ replacing $\phi$ in both equations.
The conditions that the value of $\psi(A_i)$ extend to $\psi(B_i)$
give the same degree conditions.

Hence {\red any point in the image of $u$ lies in the image of $v$.}
\end{proof}

It should be clear that the above argument is ``local'' in the sense
that we have studied a global section by studying what happens at $A_1$,
and then at $A_2$ by considering the ``neighbourhood'' $\{B_3,A_1,A_2\}$
of $B_3$, and then studying the neighbourhoods $\{B_1,A_1\}$ and
$\{B_2,A_2\}$.  For this reason the above proof really shows that
the relation \eqref{eq_hom_md} is really a ``local'' 
morphism {\red (not generally an isomorphsim)}
$$
\SHom\bigl(\cM_{\ourv d},\cM_{\ourv d'}\bigr) \rightarrow \cM_{\ourv d'-\ourv d}$$
as we will describe in Section~\ref{se_simpleV}
(where we formalize the $\SHom$ above as {\em sheaf Hom}).

\section{Proof of the Duality Theorems}
\label{se_duality_proof}

In this section we 
{\red 
show that skyscrapers satisfiy a strong duality property;
we cannot expect that this duality holds for $\omega=\cM_{k,r,\ourv K_r}$
or any other $\cM_{k,r,\ourv d}$ when $r\ge 2$;
the older version claimed strong duality holds for the $\cM_{k,r,\ourv d}$.}

\begin{definition}\label{de_strong_duality}
Let $k$ be a field, and $\cO$ a sheaf of $k$-algebras on a bipartite
category $\cC$.
Let $\omega$ be a sheaf of $k$-on a bipartite category $\cC$ and fix
a morphism
$\phi\from H^1(\omega)\to k$ of $k$-vector spaces.
For each sheaf $\cF$ of $\cO$-algebras we get a pairing
\begin{equation}\label{eq_duality_omega}
H^i(\cF)\times \Ext^{1-i}(\cF,\omega) \to H^1(\omega) 
\xrightarrow{\phi} k.
\end{equation}
For any $i$ we say that {\em $H^i$-duality} holds for $\cF$ 
with respect to $\phi,\omega$
if \eqref{eq_duality_omega} is a perfect paring; for any $i$ we use
${\rm Duality}_{\cO,i}(\phi,\omega)$ to denote the class of sheaves
satisfying $H^i$-duality;
we say that {\em strong duality} holds for $\cF$ 
with respect to $\phi,\omega$
if \eqref{eq_duality_omega} 
\begin{enumerate}
\item is a perfect pairing
for $\cF$ and both $i=0,1$, and
\item $H^i(\cF)$ and $\Ext^i(\cF,\omega)$ vanish for $i\ge 2$;
\end{enumerate}
we use ${\rm Strong-Duality}_\cO(\phi,\omega)$ to 
denote the set of sheaves, $\cF$,
for which strong duality holds.
If $H^1(\omega)\isom k$, then the class of sheaves for which any of
the above notions of duality holds 
is the same
for all isomorphisms $\phi$, and we drop $\phi$ in the above notation
(and assume that $\phi$ is any isomorphism).
\end{definition}
Of course, the zero sheaf always lies in ${\rm Duality}_\cO(\phi,\omega)$.
We also drop the $\cO$ from the notation if $\cO$ is understood.

Now we state the main theorem in this section; before doing so we remark
that the cokernel of
$$
k[y_1] \xrightarrow{p(y_1)\mapsto y p(y_1) } k[y_1]
$$
is an $k[y_1]$-algebra $\tilde k=k[y_1]/y_1 k[y_1]$, which is as a set
is isomorphic to $k$ in on which $k[y_1]$ acts by the rule
$a_0+ya_1+\cdots+y^m a_m$ times $b$ is taken to $a_0 b$
(for $a_0,\ldots,a_m,b\in k$).
We may write $k$ for $\tilde k$, but we often write $\tilde k$ to
remind ourselves of this particular $k[y_1]$-module structure;
similarly for $2$ replacing $1$ in the subscripts.

These are the main duality theorems that we prove in this article.

\begin{theorem}\label{th_strong_duality}
Let $k$ be a field and $r\ge 1$ an integer.
Then $\omega=\omega_{k,r}=\cM_{k,r,K_r}$ has $H^1(\omega)\isom k$
and ${\rm Strong-Duality}(\omega)$ contains
\begin{enumerate}
\item the skyscraper sheaf $\Sky(B_1,\tilde k)$, where 
$\tilde k=k[y_1]/(y_1 k[y_1])$; and
\item the skyscraper sheaf $\Sky(B_2,k[y_2]/y_2 k[y_2])$.
\end{enumerate}
\end{theorem}

{\color{red}
An earlier version claimed that $\omega$ satisfies $H^1$-duality.
}

We have already computed $H^1(\omega_{k,r})\isom k$; hence to prove
the theorem above it suffices to verify the duality statements.
We shall divide this proof into several subsections.

\subsection{Proof of Theorem~\ref{th_strong_duality}}
\label{su_proof_strong_duality}

\begin{proof}[Proof of Theorem~\ref{th_strong_duality}]
By symmetry it suffices to verify the case $i=1$, i.e., that
$$
\cS_1 = \Sky(B_1,k[y_1]/y_1 k[y_1]) \in {\rm Duality}(\omega_{k,r}) \ .
$$

To verify that \eqref{eq_duality_omega} holds for $\cF=\cS_1$ and
$i=1$ is easy: since $\cS_1$ is a skyscraper sheaf, $H^1(\cS_1)=0$.
Furthermore any $\gamma\in\Hom(\cS_1,\omega)$ is determined by its
only possible nonzero map $\gamma(B_1)\from \cS_1(B_1)\to\omega(B_1)$;
but this map must be zero, since $\omega(A_1,B_1)$ is an injection
and $\cS_1(A_1,B_1)$ is the zero map.

This shows that \eqref{eq_duality_omega} holds for $\cF=\cS_1$ and
$i=1$; Lemma~\ref{le_sky_Betti} also shows that $H^i(\cS_1)=0$ for
$i\ge 2$ (which is true where $\cS_1$ is replaced with any sheaf
of $k$-vector spaces).
Hence it
remains to verify that \eqref{eq_duality_omega} holds for $\cF=\cS_1$ and
$i=0$,
and to verify that $\Ext^i(\cF,\omega_{k,r})=0$ for $i\ge 2$.

In the short exact sequence of $k$-vector spaces \eqref{eq_basic_exact},
the source and target of $\mu$ are $\cO$-modules, and hence this
becomes a short exact sequence of $\cO$-modules
\begin{equation}\label{eq_basic_exact_O_mods}
0 \to \cM_{\ourv d} \xrightarrow{\mu}
 \cM_{\ourv d+(1,0)} \to \Sky(B_1,\tilde k) \to 0  
\end{equation}
with $\tilde k = k[y_1]/y_1 k[y_1]$.
Now take $\ourv d=\omega_{k,r}$, so that we get an exact sequence
$$
\begin{tikzcd}
H^0({\rm Sky}(B_1,k)) \arrow[r] \arrow[d] 
& H^1(\cM_{r;K}) \arrow[r] \arrow[d]
& H^1(\cM_{r;K+(1,0)}) \arrow[d] \\
\Ext^1({\rm Sky}(B_1,k),\omega_r)^* \arrow[r]
& \Hom(\cM_{r;K},\omega_r)^* \arrow[r] 
& \Hom(\cM_{r;K+(1,0)},\omega_r)^*
\end{tikzcd}
$$
Using Theorem~\ref{th_hom_md}, we now observe, that
\begin{enumerate}
\item
$\Hom(\cM_{r;K},\omega_r) = \Hom(\cM_{r;K},\cM_{r;K})\isom
\Gamma(\cM_{k,r,\ourv 0})$ which is one-dimensional;
\item
similarly 
$\Hom(\cM_{r;K}+(1,0),\omega_r) \isom
\Gamma(\cM_{k,r,(-1,0)})$ which vanishes;
\item
by Corollary~\ref{co_important_first_Betti},
$H^1(\cM_{k,r,K_r})\isom k$ and $H^1(\cM_{k,r,K_r+(1,0)})=0$;
\item
the middle downward arrow $H^1(\cM_{k,r,K_r}) \to \Hom(\cM_{r;K},\omega_r)^*$
is an isomorphism, since both these spaces are one-dimensional
and the identity in 
$\Hom(\cM_{k,r,K_r},\omega_r)=\Hom(\omega_{k,r},\omega_{k,r})$
induces, via the pairing \eqref{eq_Yoneda_pairing},
the identity map $H^1(\omega_{k,r})\to H^1(\omega_{k,r})$.
\end{enumerate}
Hence the above diagram amounts to
\begin{equation}\label{eq_sky_almost_done}
\begin{tikzcd}
H^0({\rm Sky}(B_1,\tilde k)) \arrow[r] \arrow[d] 
& k \arrow[r] \arrow[d,"\isom"]
& 0 \arrow[d] \\
\Ext^1({\rm Sky}(B_1,\tilde k),\omega_r)^* \arrow[r]
& k \arrow[r] 
& 0
\end{tikzcd}
\end{equation} 
Furthermore $H^0({\rm Sky}(B_1,\tilde k))\isom \tilde k$ by 
Lemma~\ref{le_sky_Betti}.  
Finally, consider the short exact sequence 
\begin{equation}\label{eq_cosky_res_Sone}
0 \to \CoSky(B_1,k[y_1])\xrightarrow{\mbox{mult by $y_1$ at $A_1,B_1$}}
\CoSky(B_1,k[y_1]) \to \Sky(B_1,\tilde k) \to 0 \ 
\end{equation} 
obtained from \eqref{eq_basic_exact} by setting all values at $A_2,B_2,B_3$
(and appropriate restriction maps) to zero.  
Setting $\cF=\CoSky(B_1,k[y_1])$ we get a long exact sequence
$$
0\to \Ext^0(\cS_1,\omega) 
\to \Ext^0(\cF,\omega) 
\to \Ext^0(\cF,\omega) 
\to \Ext^1(\cS_1,\omega) \to 0
$$
since $\Ext^1(\cF,\omega)=0$ since $\cF$ is a projective $\cO_{k,r}$-module.
We have seen above that $\Ext^0(\cS_1,\omega)=\Hom(\cS_1,\omega)=0$,
and \eqref{eq_sky_adjointness} implies that
$$
\Ext^0(\cF,\omega)=\Hom(\cF,\omega)=\Hom_{k[y_1]}(k[y_1],\omega(B_1))
\isom \omega(B_1)=k[y_1],
$$
so the long exact sequence becomes
$$
0\to k[y_1] \xrightarrow{\phi} k[y_1] \to \Ext^1(\cS_1,\omega)\to 0,
$$
and the functoriality of the coskyscraper functor shows that $\phi$
is just multiplication by $y_1$.  It follows that 
$$
\Ext^1(\cS_1,\omega) \isom k[y_1]/y_1 k[y_1],
$$
which is one dimensional.  Hence \eqref{eq_sky_almost_done} becomes
$$
\begin{tikzcd}
H^0({\rm Sky}(B_1,\tilde k))\isom k \arrow[r] \arrow[d] 
& k \arrow[r] \arrow[d,"\isom"]
& 0 \arrow[d] \\
\Ext^1({\rm Sky}(B_1,\tilde k),\omega_r)^* \isom k\arrow[r]
& k \arrow[r] 
& 0
\end{tikzcd} \ .
$$
The exactness in the rows show that the upper left and lower left 
horizontal arrows are isomorphisms, and hence and leftmost downward
arrow is an isomorphism.

This verifies \eqref{eq_duality_omega}
for $\cF=\cS_1$ and
$i=0$.
But \eqref{eq_cosky_res_Sone} shows that $\cS_1$ has a projective
resolution of length two, and hence
$\Ext^i(\cS_1,\omega_{k,r})=0$ for $i\ge 2$.
\end{proof}

\subsection{A Lemma Related to the Five-Lemma}

\ %

{\red 
This lemma was used to study the strong duality of $\omega$;
although this strong duality does not hold for $r\ge 2$, we anticipate
that this lemma may be useful in fixing the duality result in
future work.
}

\begin{lemma}\label{le_pseudo_five}
Let
$$
\begin{tikzcd}
A_3 \arrow[r] \arrow[d] &
A_2 \arrow[r] \arrow[d] &
A_1 \arrow[r] \arrow[d] &
0 \\
B_3 \arrow[r] &
B_2 \arrow[r] &
B_1 \arrow[r] &
0\\
\end{tikzcd}
$$
be a morphism of exact chains of $k$-vector spaces
such that the maps $A_3\to B_3$ and
$A_1\to B_1$ are surjective.  Then the map
$A_2\to B_2$ is surjective.
\end{lemma}
Let us make three remarks regarding this lemma.
First, the proof we give below would still work if the top and
bottom $0$'s in the diagram were replaced by $A_0$ and $B_0$ and 
a downward isomorphism $A_0\to B_0$; this generalization and its
dual proves the five-lemma.
Second, the special case where $A_2=A_3=B_3=k$ shows that one cannot
replace ``surjective'' with ``injective'' in this lemma.
Third,
any proof via ``diagram chasing''---such as the one 
below---generalizes to the same statement in any abelian category,
by the Freyd-Mitchell Embedding Theorem.
\begin{proof}[Proof of Lemma~\ref{le_pseudo_five}.]
This is an easy diagram chase: let $\beta_2\in B_2$; we need to
show that $\beta_2$ has a preimage in $A_2$.

Let $\beta_1$ be the image of $\beta_2$ in $B_1$; by surjectivity,
$\beta_1$ has a preimage $\alpha_1$ in $A_1$; by the exactness of the
rows, $\alpha_1$ has a preimage
$\alpha_2\in A_2$; choose any such $\alpha_2$, and let $\beta_2'$
be the image of $\alpha_2$.  Since $\beta_2'$ and $\beta_2$ both map to
$\beta_1$ ($\beta_2'$ by commutativity of the diagram, since
$\alpha_2$ that maps to $\alpha_1$ that maps to $\beta_2'$), we have 
that $\beta_2'-\beta_2$ has a preimage $\beta_3$ in $B_3$.  Since
$A_3\to B_3$ is surjective, $\beta_3$ has a preimage $\alpha_3$
in $A_3$, and hence if $\alpha_2'$ is the image of $\alpha_3$ in $A_2$,
then $\alpha_3$ maps to $\beta_2'-\beta_2$.  
Hence $\alpha_2-\alpha_2'$ maps to $\beta_2$ in $B_2$.
Hence $\beta_2$ has a preimage in $A_2$.
\end{proof}


\subsection{The Method of Grothendieck}

\ %

{\red We anticipate that this discussion will be useful in future
work.
}

The following is a special case of what is
sometimes called the ``method of Grothendieck:''

\begin{lemma}\label{le_method_Grothendieck}
Let $k,\cO,\cC,\omega,\phi$ be a in Definition~\ref{de_strong_duality},
and let 
$$
0\to\cF_1\to\cF_2\to\cF_3\to 0
$$
be a short exact sequence of $\cO$-modules.  
\begin{enumerate}
\item Then if any two of
$\cF_1,\cF_2,\cF_3$ lie in ${\rm Strong-Duality}(\phi,\omega)$, then all three
do; and
\item If $\cF_1,\cF_2$ lie in ${\rm Duality}_{\cO,i}(\phi,\omega)$
and $\cF_1$ lies in ${\rm Duality}_{\cO,i+1}(\phi,\omega)$ with
$H^{i+1}(\cF_1)=0$,
then $\cF_3$ lies in ${\rm Duality}_{\cO,i}(\phi,\omega)$.
\end{enumerate}
\end{lemma}
The proof is immediate from the five-lemma applied to
$$
\begin{tikzcd}[column sep=tiny]
0 \arrow[r] \arrow[d]
& H^1(\cF_3)^{\ast} \arrow[r] \arrow[d]
& H^1(\cF_2)^{\ast} \arrow[r] \arrow[d]
& H^1(\cF_1)^{\ast} \arrow[r] \arrow[d]
& H^0(\cF_3)^{\ast} \arrow[r] \arrow[d]
& H^0(\cF_2)^{\ast} \arrow[r] \arrow[d]
& H^0(\cF_1)^{\ast} \arrow[r] \arrow[d]
&0 \\
0 \arrow[r]
& \Hom(\cF_3, \omega) \arrow[r]
& \Hom(\cF_2, \omega) \arrow[r]
& \Hom(\cF_1, \omega) \arrow[r]
& \Ext^1(\cF_3, \omega) \arrow[r]
& \Ext^1(\cF_2, \omega) \arrow[r]
& \Ext^1(\cF_1, \omega) \arrow[r]
&0
\end{tikzcd}
$$

\section{Foundations, Part 5: Sheaf Hom and Local Considerations}
\label{se_simpleV}

It conceivable that there is a more general strong duality theorem
that is even easier to prove, namely a local theorem whose
global version Theorem~\ref{th_strong_duality}.
In this section we describe some ``local considerations.''
In particular our proof Theorem~\ref{th_hom_md} really proves that there 
is an injection {\red (not an isomorphism)} 
\begin{equation}\label{eq_SHom_M_result_again}
\SHom_{\cO_r}(\cM_{r,\ourv d},\cM_{r,\ourv d'})\rightarrow \cM_{r,\ourv d'-\ourv d},
\end{equation} 
which immediately implies that
\begin{equation}\label{eq_second_main_formula}
dim(\Hom_{\cO_r}(\cM_{r,\ourv d},\cM_{r,\ourv d'}))\leq
b_0(\cM_{r,\ourv d'-\ourv d})
\end{equation} 
by taking global sections.
We will also define the {\em tensor product} of sheaves
and derive some convenient formulas such as
$$
\SHom\bigl( \cL_{k,r,\ourv d}\otimes \cF,\cG) \isom
\SHom\bigl( \cF,\cL_{k,r,-\ourv d}\otimes\cG) ,\quad
\cL_{k,r,\ourv d}\otimes \cL_{k,r,\ourv d'}\isom
\cL_{k,r,\ourv d+\ourv d'},
$$
$$
\cL_{k,r,\ourv d}\otimes \cM_{k,r,\ourv d'}\isom
\cM_{k,r,\ourv d+\ourv d'}.
$$
We will also make some remarks about the curious nature of 
$$
\cM_{k,r,\ourv d}\otimes \cM_{k,r,\ourv d'} \ .
$$

\subsection{Localization}

Topos theory strongly suggests how we should define ``localization,''
and thereby a number of concepts such as {\em sheaf Hom}.  Here is 
the upshot in our case.

\begin{definition}
Let $G=(\cA,\cB,E)$ be a bipartite graph.  If $P\in\cA\amalg \cB$, we define
the {\em restriction of $G$ to $P$} as follows:
\begin{enumerate}
\item
if $P=A\in\cA$, we define
$G/A$ to be the bipartite graph  $(\{A\},\emptyset,\emptyset)$;
\item
if $P=B\in\cB$, we define $G/B$ to be the bipartite graph
$(\cA_B,\{B\},\cA_B\times\{B\})$ where $\cA_B$ consists of those elements
of $\cA$ incident upon $B$.
\end{enumerate}
Furthermore, if
$\cC$ is the bipartite category associasted to $G=(\cA,\cB,E)$,
for any $P\in\Ob(\cC)=\cA\amalg\cB$, we define the
{\em restriction of\, $\cC$ to $P$}, denoted $\cC/P$, to be the subcategory
of $\cC$ associated to $G/P$.
\end{definition}

More generally, if $\cC$ is any category
endowed with its coarsest topology, then
[SGA4], Expos\'e IV, gives a recipe where the {\em points} of the
topos are naturally identified with $\Ob(\cC)$, each of which is
contained in a minimal {\em open subtopos}, which is precisely the
``slice category'' $\cC/P$, often called the category
``of objects over $P$'' or ``of morphisms to $\cC$.''

\begin{definition}
If $\cF$ is a sheaf of $k$-vector spaces on a bipartite category, $\cC$,
and $P\in\Ob(\cC)$, then by the {\em localization 
of $\cF$ to $P$},
denoted $\cF|_P$, we mean the
restriction of $\cC$ to $\cC/P$.  Similary if $u\from\cF\to\cG$ is
a morphism, the {\em localization of $u$ to $P$} is the evident
restriction map $u|_P$ from $\cF|_P$ to $\cG_P$.
\end{definition}
This definition is valid for any category
$\cC$ (viewed as a topos with the coarsest topology), although 
$\cC/P$ is not generally a subcategory of $\cC$, and the
natural ``localization map'' $\cC/P\to\cC$ is not
an inclusion; the exception is when $\cC$ is topological in the
sense of Definition~\ref{de_top_semitop}, which includes the case
of bipartite categories.

In the following subsections we will see the importance of the
notion of localization.

\subsection{Sheaf Hom}

\begin{definition}
Let $\cC$ be the simple category associated to a bipartite graph, $G$,
$\cO$ a sheaf of rings, and
$\cF,\cG$ sheaves of $\cO$-modules.
We define a sheaf $\cH=\SHom(\cF,\cG)$ as the sheaf of $\cO$-modules
whose value at $P$ is
$$
\cH(P) = \Hom_{\cO|_P}(\cF|_P,\cG|_P) ,
$$
and whose restirction maps $\cH(B)\to\cH(A)$, for a morphism $A\to B$,
is given by the restriction since $G_A$ is a subgraph of $G_B$.
\end{definition}

In the above definition, if $(A,B)$ is an edge in $G$, then
a morphism $\cF(B)\to\cG(B)$ does not necessarily extend to a morphism
$\cF(A)\to\cG(A)$, and if it does then the extension is not necessarily
unique.  
Hence we cannot define a good notion of sheaf Hom by looking 
``value-by-value.''  Furthermore
there is a canonical morphism
$$
\Hom_{\cO|_P}(\cF|_P,\cG|_P) \to \Hom_{\cO(P)}\bigl( \cF(P),\cG(P) \bigr)
$$
(see [EGA1], Section~0.5.2.6 or \cite{hartshorne}, Proposition~III.6.8),
but this morphism is not an isomorphism when $\cF,\cG$ are of the
form $\cM_{k,r,\ourv d},\cM_{k,r,\ourv d'}$, since in this case,
the left-hand-side 
$\cO(P)$ module is of rank $r$ for $P=B_3$, whereas the right-hand-side
is of rank $r^2$.

\subsection{Proof of \eqref{eq_SHom_M_result_again}}

Let us briefly indicate a proof of \eqref{eq_SHom_M_result_again}.

{\red Again, since we are only claiming the existence of a morphism,
this result is far simpler.
Yet, we expect that future work may use such local considerations.
Again, the reason that we get a morphism and not an isomorphism
is the problem with the $B_3$ neighbourhood; all other neighbourhoods
are fine, but the $B_3$ neighbourhood is essential in connecting
the points $B_1,A_1$ to $B_2,A_2$.
}

The proof of Theorem~\ref{th_hom_md} really gives 
\begin{enumerate}
\item a morphism
$$
\SHom_{\cO_r}(\cM_{r,\ourv d},\cM_{r,\ourv d'})(B_3) 
\to \cM_{r,\ourv d'-\ourv d}(B_3)
$$
\item morphisms for $i=1,2$
$$
\SHom_{\cO_r}(\cM_{r,\ourv d},\cM_{r,\ourv d'})(B_i)
\to \cM_{r,\ourv d'-\ourv d}(B_i)
$$
which are both isomorphic to $\Hom_{R_i}(R_i,R_i)=R_i$ with
$R_i=k[y_i]$,
\item
such that for $i=1,2$, these
isomorphisms at $B_3$ and $B_i$ restrict to the same
isomorphism
$$
\SHom_{\cO_r}(\cM_{r,\ourv d},\cM_{r,\ourv d'})(A_i)
\to \cM_{r,\ourv d'-\ourv d}(A_i)
$$
which are both isomorphic to $\Hom_{S_i}(S_i,S_i)\isom S_i$ with
$S_i=k[x_i,1/x_i]$.
\end{enumerate}
Hence Theorem~\ref{th_hom_md} really gives
morphisms for \eqref{eq_SHom_M_result_again} restricted to open
neighbourhoods of $B_1,B_2,B_3$ which (1) are isomorphisms in 
the neighbourhoods 
{\red of $B_1$ and of $B_2$, but only
morphisms
in that of $B_3$,} and (2) agree on their overlap
when restricted to $A_1,A_2$
{\red (this overlap is small because of the problems at $B_3$)}.
Hence we get a global morphism~\eqref{eq_SHom_M_result_again}.

\subsection{Tensor Product and Line Bundles}

\begin{definition}
Let $\cF,\cG$ be sheaves of $\cO$-modules for some sheaf of rings
$\cO$ on a bipartite category $\cC$.  We define the {\em tensor
product} of $\cF,\cG$, denoted $\cF\otimes_\cO \cG$, to be the
sheaf whose values are
$$
\bigl( \cF\otimes_\cO \cG \bigr)(P) = 
\cF(P)\otimes_{\cO(P)} \cG(P),
$$
and whose restriction maps are obtained as the tensor product
of the restriction maps of $\cF$ and $\cG$.
\end{definition}

Analogous to \eqref{eq_SHom_M_result_again} we easily see that
$$
\cL_{k,r,\ourv d}\otimes \cL_{k,r,\ourv d'}\isom
\cL_{k,r,\ourv d+\ourv d'}
$$
and
\begin{equation}\label{eq_otimes_L_M}
\cL_{k,r,\ourv d}\otimes \cM_{k,r,\ourv d'}\isom
\cM_{k,r,\ourv d+\ourv d'} \ .
\end{equation} 

The fact that
$$
\cL_{k,r,\ourv d}\otimes \cL_{k,r,-\ourv d}\isom
\cL_{k,r,\ourv 0}=\cO_{k,r}
$$
justifies calling the $\cL_{k,r,\ourv d}$ {\em invertible sheaves}.
Let us summarize this idea.

By a {\em vector bundle} we mean a sheaf, $\cF$, such that for some 
integer $n\ge 1$, $\cF$ is {\em locally
isomorphic to $\cO^n$}, i.e., for any $P\in\Ob(\cC)$ we have
$$
\cF|_P \isom \cO^n|_P .
$$
The case $n=1$ is called a {\em line bundle}.

\begin{example}
Consider any sheaf $\cF$ on $\CTwoV$ such that
$\cF(P)=\cO(P)$ for all $P\in\Ob(\CTwoV)$ and that 
$\cF(A_i,B_3)=\cO(A_i,B_3)$ for $i=1,2$.
Then for $\cF$ to be a line bundle it is necessary that
$\cF(A_1,B_1)1$ be multiplicative unit in $\cF(A_1)=k[y_1,1/y_1]$,
since if not we cannot have $\cF|_{B_1}\isom\cO|_{B_1}$;
we easily check that if $\cF(A_i,B_i)$ is a multiplicative
unit in $\cF(A_i)$ for $i=1,2$, then $\cF$ is indeed a line
bundle.
Hence for $i=1,2$ we have $\cF(A_i,B_i)=x_i^{d_i}c_i$ for some
nonzero $c_i\in k$; our sheaves $\cL_{k,r,\ourv d}$ are the
special cases where $c_1=c_2=1$, and for general nonzer $c_1,c_2$
the sheaf $\cF$ is isomorphic to $\cL_{k,r,\ourv d}$.
\end{example}

There are a number of standard facts about line bundles and vector
bundles: for example, any line bundle $\cL$ has an ``inverse''
line bundle $\cL^{-1}$ for which $\cL\otimes\cL^{-1}\isom\cO$, and
$$
\SHom(\cL\otimes\cF,\cG) \isom \SHom(\cF,\cL^{-1}\otimes\cG)
$$
for any $\cO$-modules $\cF,\cG$ (where we have dropped the subscripts
of $\cO$ from $\SHom$ and $\otimes$).
It follows immediately that 
$$
\SHom\bigl( \cM_{k,r,\ourv d} ,\cM_{k,r,\ourv d'})
\isom
\SHom\bigl( \cM_{k,r,\ourv 0} ,\cM_{k,r,\ourv d' - \ourv d})
$$

We easily verify a number of other standard formulas such as the
existence of an isomorphism
\begin{equation}\label{eq_shom_tensor}
u\from \SHom_\cO\bigl( \cF,\SHom_\cO(\cG,\cH) \bigr) \to
\SHom_\cO(\cF\otimes_\cO\cG,\cH)
\end{equation}
that is natural in $\cF,\cG,\cH$.
To describe this isomorphism we need to specify $u$ at each 
$P\in\Ob(\cC)$, which amounts to a morphism
$$
u(P)\from 
\Hom_{\cO|_P}\Bigl( \cF|_P, \Hom_{\cO|_P}\bigl(\cG|_P,\cH|_P\bigr)\Bigr)
\to
\Hom_{\cO|_P}\bigl( (\cF\otimes_\cO\cG)|_P,\cH|_P \bigr) 
$$
which is straightforwards to describe
(see \cite{hartshorne} Exericse~II.5.1(c), or [SGA4] Section~V.12.8
for the more general case).
By Yoneda's lemma, \eqref{eq_shom_tensor}
uniquely determines $\cF\otimes\cG$.

\subsection{Remarks on $\cM_{k,r,\ourv 0}\otimes\cM_{k,r,\ourv d}$}

It is curious to note that
$$
\bigl( \cM_{k,r,\ourv 0}\otimes\cM_{k,r,\ourv d} \bigr)(B_3)
$$
is of rank $r^2$, while the inclusion
$$
\cO_{k,r}\to \cM_{k,r,\ourv 0}
$$
sending $\cO_{k,r}(B_3)$ to the first component of $\cM_{k,r,\ourv 0}(B_3)$
gives an inclusion
$$
\cO_{k,r}\otimes\cM_{k,r,\ourv d} \to 
\cM_{k,r,\ourv 0}\otimes\cM_{k,r,\ourv d}
$$
where the left-hand-side is just $\cM_{k,r,\ourv d}$. 
Note that this map takes $e_1\otimes e_j$ to $e_j$ and
$e_{j'}\otimes e_j$ to zero for $j'\ge 2$; one could do the reverse map;
more generally, for each $j=1,\ldots,r$ and each $j',j''$ with
$j'+j''=r$, one can map $e_{j'}\otimes e_{j''}$ to some multiple of $e_j$
and get a morphism
$$
\cM_{k,r,\ourv 0}\otimes\cM_{k,r,\ourv d}\to\cM_{k,r,\ourv d}.
$$
We check that the cokernel of any such map (that is an inclusion)
is supported at $B_3$ and
seems a bit strange: it contains the classes of elements
\begin{equation}\label{eq_zero_stuff}
(e_{i_1}\otimes e_{i_2})-(e_{i_3}\otimes e_{i_4}) \quad{\rm s.t.}
\quad i_1+i_2=i_3+i_4 
\end{equation} 
that restrict to zero; this gives a $k[v,1/v]$-module of rank
$\binom{r}{2}$ of elements of
$$
\bigl( 
\cM_{k,r,\ourv 0}\otimes\cM_{k,r,\ourv d}
\bigr)(B_3)
$$
that map to zero along all restriction maps.
However when $i_1+i_2=i_3+i_4 - r$, the class
of
$$
(e_{i_1}\otimes e_{i_2})v -(e_{i_3}\otimes e_{i_4}) 
$$
is taken to $0$ in the $(A_1,B_3)$ restriction but not in the
$(A_2,B_3)$ restriction, and vice versa for the class of
$$
(e_{i_1}\otimes e_{i_2})v^{-1} -(e_{i_3}\otimes e_{i_4})  \ ;
$$
similarly, this generates a module of rank $\binom{r}{2}$ of 
$$
\bigl( 
\cM_{k,r,\ourv 0}\otimes\cM_{k,r,\ourv d}
\bigr)(B_3)
$$
with a curious ``twisting'' property.  It follows that
$\cM_{k,r,\ourv 0}\otimes\cM_{k,r,\ourv d}$ is the direct sum of
an module supported at $B_3$ generated by the elements of
\eqref{eq_zero_stuff} and the rest, which admits an inclusion
from $\cM_{k,r,\ourv d}$, but has some curious ``twisting'' elements.

Of course, the commutativity (up to isomorphism) of $\otimes$ 
and \eqref{eq_otimes_L_M}
shows that
$$
\cM_{\ourv d}\otimes\cM_{\ourv d'} \isom
\cM_{\ourv 0}\otimes\cM_{\ourv d + \ourv d'} \ ,
$$
and so similar remarks hold for this more general left-hand-side
tensor product.

\section{Remarks for Future Research}
\label{se_future}

{\red 
A main goal of future research is to get a duality theorem
for $\cO$-modules with $\cO=\cO_{k,r}$ that holds for
the $\cM_{k,r,\ourv d}$.
We hope to address this in future work.
}

In this section we make some remarks regarding future research.  We 
will discuss ``local methods'' and some issues involving them, both
for $\CTwoV$ and for more general categories including those of
interest in previous works on sheaves on graphs.

We conclude with some brief remarks regarding other directions.

\subsection{Local Duality Theorems}
\label{su_sky_duality}

In this subsection we explore possible 
generalizations stronger duality theories.

It would seem simpler and more natural to prove that 
for 
$$
\Sky\bigl(B_1,\cO_{k,r}(B_1) \bigr)
=
\Sky\bigl( B_1, k[y_1] \bigr)
$$
satisfies strong duality, and then infer this for
$\Sky\bigl( B_1,\widetilde k)$
from the method of Grothendieck and the exact sequence
$$
0 \to 
\Sky\bigl( B_1, k[y_1] \bigr)
\to 
\Sky\bigl( B_1, k[y_1] \bigr)
\to
\Sky\bigl( B_1, \widetilde k \bigr)
\to 0
$$
which is immediate from the exact sequence of $k[y_1]$-algebras
$$
0 \to k[y_1] \to k[y_1]\to \widetilde k \to 0 \ .
$$
The problem is that $\Sky\bigl( B_1, k[y_1] \bigr)$ does not
have finite zeroth Betti number, and two infinite dimensional
vector spaces cannot be perfectly paired with each other
since their dual spaces are too large.
Here we make a number of possible ways in which
one can study strong duality of $\Sky\bigl( B_1, k[y_1] \bigr)$.

First, we have a canonical isomorphism
\begin{equation}\label{eq_H_0_Sky}
H^0\Bigl( \Sky\bigl( B_1, k[y_1] \bigr) \Bigr) \isom k[y_1]
\end{equation} 
and---in view of the projective resolution of $\omega_{k,r}$---a
canonical isomorphism
$$
\Ext^1\Big( \Sky\bigl( B_1, k[y_1] \bigr), \omega_{k,r} \Bigr)
\isom
k[x_1,1/x_1]/k[y_1] = k[x_1,1/x_1]/k[1/x_1] \ ;
$$
with these identifications we could compute the pairing of these spaces
and presumably infer an explicit formula for the pairing on 
$\Sky\bigl( B_1, \widetilde k)$.

Second, one could also write
\begin{equation}\label{eq_Ext_1_Sky}
\Ext^1\Big( \Sky\bigl( B_1, k[y_1] \bigr), \omega_{k,r} \Bigr)
\isom
k[y_1,1/y_1]/k[y_1],
\end{equation} 
giving this group the structure of a finitely generated $k[y_1]$-algebra;
the group \eqref{eq_H_0_Sky} also has such a structure.  Hence, although
these two groups are infinite dimensional $k$-vector spaces, they
are both finitely generated $k[y_1]$-algebras.

Third, although the dual $k$-vector space of \eqref{eq_H_0_Sky} is 
too large to be isomorphic to \eqref{eq_Ext_1_Sky},
there is a natural subspace, $S$, of the dual space---which one could
call the {\em tame dual} (perhaps {\em finitely supported dual}) of
\eqref{eq_H_0_Sky}---to which
\eqref{eq_Ext_1_Sky} appears to be isomorphic: namely, define the
{\em tame dual} of $k[y_1]$ to be the set of linear functionals
$\ell\from k[y_1]\to k$ such that
$$
\ell(a_0 + a_1 y_1 + a_2 y_1^2 +\cdots+a_n y_1^n)
$$
depends only on finitely many of the $a_i$'s (independent of $n$, of course).

More generally, if $W$ is any $k$-vector space with a countable basis
$w_1,w_2,\ldots$, we define the {\em tame dual} or {\em finitely supported
dual} to be the linear functionals that depend on only finitely many
of the coefficients of $w_i$ in the unique representation of any vector.
Although this notion depends on the basis we choose for $W$, it is pretty
clear that the pairing
$$
H^0\Bigl( \Sky\bigl( B_1, k[y_1] \bigr) \Bigr) \times
\Ext^1\Big( \Sky\bigl( B_1, k[y_1] \bigr), \omega_{k,r} \Bigr) \to k
$$
should be a tame pairing, i.e., the pairing of $y_1^n$ in the first
factor should depend on only finite many of the $1/y_1^m$ in the second.
One possible problem with the tame dual is that if the graph
Riemann-Roch or some other application requires us to model spheres
with more than just $0$ and $\infty$ missing, i.e., three or more points
missing, then it is not clear we can choose local bases for Hom and Ext
groups that are globally compatible.

\subsection{Alternative Views of $\Sky(B_1,k[y_1]/(y_1)$ Strong Duality}

There are two other methods that one could use to understand strong
duality of $\cS_1=\Sky(B_1,k[y_1]/(y_1))$ (and the same with $2$ replacing
$1$ in the subscripts).
First, one could compute the effect of $\Hom_\cO(\cO,\cS_1)=k[y_1]/(y_1)$
on $\Ext^1(\cS_1,\omega_{k,r})$ using the fuctorial nature of
$\Ext$.  Second, one could do this in terms of Yoneda ext groups.
These two are essentially the same computation; however, either of these
two methods may illuminate and provide an easier proof of the 
perfect pairing of $\Hom_\cO(\cO,\cS_1)$ with $\Ext^1(\cS_1,\omega_{k,r})$.

\subsection{Localization for Graphs and Graphs of Groups}

The categories used in \cite{friedman_memoirs_hnc} were based on 
oriented graphs, $G$, that could have self-loops.  In the event that 
such a $G$
has no self-loops, then the resulting category is 
a bipartite category with left objects being $V_G$, the vertices of the
oriented graph, and the right objects were the edges of the oriented
graph, $E_G$.
However, when such a $G$ has self-loops (which would be
{\em whole-loops} in the sense of \cite{friedman_geometric_aspects}), 
one gets two morphisms
from the vertex to its self-loop; this therefore a
semi-topological category (Definition~\ref{de_top_semitop}), but not
a topological category, and the resulting topos is not a topological
space.

One aspect of localization in general finite categories, $\cC$, is that
when we work with the slice category $\cC/P$ as the ``neighbourhood of
$P$,'' we seem to get very reasonable definitions.
One superficial aspect of this is that if we generalize our bipartite
categories to include the above categories of \cite{friedman_memoirs_hnc},
then if $e$ is a self-loop about a vertex $v$ as above, 
then $\cC/e$ is the category with three objects, whose objects are
the two morphisms to $e$ from $v$ and the identity morphism $\id_e$;
hence $\cC/e$ is two objects mapping to one object, just as $\cC/e$
is when $e$ is not a self-loop (i.e., when $e$ has distinct endpoints).

We can see that this remark regarding $\cC/e$ also holds when a graph
has {\em half-loops} in the sense of \cite{friedman_geometric_aspects}
(half-loops play an important role in defining regular random graphs of
odd degree and other aspects of covering theory).
A {\em graph} is defined as a directed graph with an orientation reversing
involution; therefore if $e$ is a half-loop, i.e., an edge that is the
fixed point of the involution, then in the corresponding category
one would expect that its incident vertex, $v$, has two morphisms to $e$,
but that $\Hom(e,e)$ is a cyclic group of order two.
In this case the slice category $\cC/e$ has four objects, including
the two elements of $\Hom(e,e)$ which are isomorphic
(as objects of $\cC/e$);
hence the resulting category is still the above three-object category.

Similarly consider a {\em graphs of groups}, $(T,G)$,
as in \cite{serre_trees} Section~4.4, Definition 8, with the notation there.
It is natural associate to $(T,G)$ a category, $\cC$, whose objects are
$V\amalg E'$, where $V$ is the vertex set of $T$, $E$ the edge set,
and $E'$ is the collection of unordered pairs $\{e,\overline e\}$ with
$e\in E$; the morphisms of $\cC$ are as follows:
(1) $\Hom(v,v)=G_v$ for $v\in V$;
(2) $\Hom(e',e')=G_e=G_{\overline e}$ for $e'=\{e,\overline e\}\in E'$;
(3) we introduce a set of morphisms $t(e)\to \{e,\overline e\}$ for 
each $e\in E$
that consists of $G_v$.  The compositions with elements of 
$\Hom(t(e),\{e,\overline e\})$
are multiplication in $G_v$
either on the left with $G_v$ (for $\Hom(v,v)$) 
or on the right with $G_e$ acting via its
image in $G_v$ under $G_e\to G_{t(e)}$.\footnote{
  See also \cite{bridson}, 
  III.C.2.8 (page~538), which is presumably equivalent (we do not understand
  what $t(\alpha)$ would mean for $\alpha\in\cY$ if $\alpha\in V(\cY)$ in
  the second paragraph of this section).
}
In this case, for $e'=\{e,\overline e\}\in E'$ we have $\cC/e'$ 
consists of objects $G_{t(e)}\amalg G_{t(\overline e)}\amalg G_e$,
where any two objects in any of these three summands are isomorphic; we
therefore see that $\cC/e'$ is equivalent to the same three-object category
as $\cC/e$ for graphs above.
[Similarly, we see that $\cC/v$ for $v\in V$ is equivalent to the category
with one element and one morphism, just as it would be in a graph.]

From the above examples, the view of topos theory
([SGA4.IV]) that $\cC/P$ is the ``smallest neighbourhood
of $P$'' indicates that graphs, graphs with whole-loops and/or half-loops,
and graphs of groups all have the same ``local'' structure.
This seems quite promising when we try to combine the methods of this
article with those of \cite{friedman_memoirs_hnc}, and to generalize
these methods to graphs with half-loops and graphs of groups.

\subsection{Other Remarks}

For any integers $r',r''\ge 1$ and field $k$, there is a morphism 
$\cO_{k,r'}\to \cO_{k,r'r''}$ that maps the indeterminate $v\in\cO_{k,r'}$
to $v^{r''}$ in $\cO_{k,r'r''}$ and is otherwise the identity.
Hence there is a morphism of ringed spaces
(\cite{hartshorne}, page~72)
from $(\CTwoV,\cO_{k,r'r''})$ to $(\CTwoV,\cO_{k,r'})$.
It would be interesting to study such morphisms; in particular,
each $\cO_{k,r}$ can be considered as a (presumably $r$-to-$1$)
covering space of $\cO_{k,1}$, which is just the sphere.

Of course, ultimately we would like to generalize this theory to
include the Riemann-Roch Graph Theorem on an arbitrary number of
vertices, and, more generally, to study $\integers^n/L$ where $L$
is any lattice in $\integers^n$ (for the Riemann-Roch Graph Theorem,
$n$ is the number of vertices and $L$ is the image of the graph
Laplacian).

One could also study for a covering map $G\to G'$ of graphs, how
the Graph Riemann-Roch Theorem and how our type of algebraic models
behave.


{\tiny\red

}

\providecommand{\bysame}{\leavevmode\hbox to3em{\hrulefill}\thinspace}
\providecommand{\MR}{\relax\ifhmode\unskip\space\fi MR }
\providecommand{\MRhref}[2]{%
  \href{http://www.ams.org/mathscinet-getitem?mr=#1}{#2}
}
\providecommand{\href}[2]{#2}


\end{document}